\documentclass[11pt]{article}

\usepackage[pdfpagelabels]{hyperref} \hypersetup{colorlinks = true, linktoc = all, linkcolor = black, citecolor = black}
\usepackage[normalem]{ulem} \usepackage{algorithm}
\usepackage{algorithmicx} \usepackage{algpseudocode}
\usepackage{pgfplots} \usepackage[toc,page]{appendix}
\usepackage{upgreek} \usepackage{fullpage}
\usepackage[english]{babel} \usepackage[latin1]{inputenc}
\usepackage{amsmath}
\usepackage{amsfonts}
\usepackage{xcolor}
\usepackage{graphicx}
\usepackage{amsthm}
\usepackage{psfrag}
\usepackage{subfig}
\usepackage{amssymb}
\usepackage{tikz}
\usetikzlibrary{hobby}
\usepackage{ifthen} \usepackage{xifthen} \usepackage{dsfont}
\usepackage{bm} \usepackage{accents}
\usepackage{xparse} %
\usepackage{mathtools} %
\usetikzlibrary{arrows,calc} \usetikzlibrary{decorations,arrows}
\usetikzlibrary{decorations.pathreplacing}

\usepackage[textsize=tiny]{todonotes}
\definecolor{jz}{rgb}{0.1,0.45,0.1}

\newcommand{\proofref}[1]{}
\newcommand{\D}{\mathrm{D}} \newcommand{\measi}{{\rho}}
\newcommand{\measii}{{\pi}}

\newcommand{\U}[1]{U_{#1}}

\newcommand{\temp}[1]{}
\newcommand{\norm}[2][]{\| #2 \|_{#1}}

\newcommand{\hide}[1]{}

\newcommand{\frf}{\mathfrak{f}}

\newcommand{\normc}[2][]{\left\| #2 \right\|_{#1}}

\newcommand{\set}[2]{\{#1\,:\,#2\}} \newcommand{\setc}[2]{\left\{#1\,
    :\,#2\right\}} \usepackage{enumitem}

\newcommand{\dd}{\;\mathrm{d}}
\newcommand{\ddd}{\mathrm{d}}

\DeclareMathOperator{\supp}{supp} 
 
\DeclareMathOperator{\argmin}{argmin}

\newcommand{\dfn}{\vcentcolon=} \newcommand{\dfnn}{=\vcentcolon}

\newtheorem{theorem}{Theorem}[section]

\newtheorem{corollary}[theorem]{Corollary}

\newtheorem{example}[theorem]{Example}

\newtheorem{lemma}[theorem]{Lemma}

\newtheorem{proposition}[theorem]{Proposition}
\newtheorem{remark}[theorem]{Remark}

\newtheorem{assumption}[theorem]{Assumption}

\numberwithin{equation}{section}
\newcommand{\cA} {{\mathcal A}} \newcommand{\cB} {{\mathcal B}}

\newcommand{\CN}{{\mathcal N}}

\newcommand{\C}{{\mathbb C}} \newcommand{\N}{{\mathbb N}}
\newcommand{\R}{{\mathbb R}} \newcommand{\bbP}{{\mathbb P}}
 
 \newcommand{\scr}[1]{{\mathfrak{#1}}}
 
\newcommand{\eps}{\varepsilon}

 \newcommand{\bsb}{{\boldsymbol b}}

\newcommand{\bsx}{{\boldsymbol x}} \newcommand{\bst}{{\boldsymbol t}}
\newcommand{\bsy}{{\boldsymbol y}} \newcommand{\bsz}{{\boldsymbol z}}

\newcommand{\bszeta}{{\boldsymbol \zeta}}

\newcommand{\bskappa}{{\boldsymbol \kappa}}

 \newcommand{\bsvarrho}{{\bm
    \varrho}} 
\newcommand{\bsnu}{{\bm \nu}} \newcommand{\bsmu}{{\bm \mu}}
\newcommand{\bseta}{{\bm \eta}} 
 
\newcommand{\bsdelta}{{\bm \delta}}

\newcommand{\bsnul}{{\bm 0}} 

\usepackage{authblk}

\begin{document}
\title{Sparse approximation of triangular transports. Part II: the
  infinite dimensional case\thanks{This paper was written during the
    postdoctoral stay of JZ at MIT.  JZ acknowledges support by the
    Swiss National Science Foundation under Early Postdoc Mobility
    Fellowship 184530. YM and JZ acknowledge support from the United
    States Department of Energy, Office of Advanced Scientific
    Computing Research, AEOLUS Mathematical Multifaceted Integrated
    Capability Center.}  \thanks{The original manuscript
    \cite{2006.06994} has been split into two parts; the first part is
    \cite{zm1}, and the present paper is the second part.}}
\author[1]{Jakob Zech} \author[2]{Youssef Marzouk}
\affil[1]{\footnotesize
  Heidelberg University, 69120 Heidelberg, Germany\\
  \href{mailto:jakob.zech@uni-heidelberg.de}{jakob.zech@uni-heidelberg.de}}
\affil[2]{\footnotesize Massachusetts Institute of Technology,
  Cambridge, MA 02139, USA\\
  \href{mailto:ymarz@mit.edu}{ymarz@mit.edu}}
  
\maketitle

{ \abstract{For two probability measures $\measi$ and $\measii$ on
    $[-1,1]^\N$ we investigate the approximation of the triangular
    Knothe--Rosenblatt transport $T:[-1,1]^\N\to [-1,1]^\N$ that pushes
    forward $\measi$ to $\measii$. Under suitable assumptions, we show
    that $T$ can be approximated by rational functions without
    suffering from the curse of dimension. Our results are applicable
    to posterior measures arising in certain %
    inference problems
    where the unknown belongs to an (infinite dimensional) Banach
    space. In particular, we show that it is possible to efficiently
    approximately sample from certain high-dimensional
    measures %
    by transforming a
    lower-dimensional latent variable.
  }%
\\



\numberwithin{equation}{section}

\section{Introduction}\label{sec:infinite}
In this paper we discuss the approximation of transport maps on
infinite dimensional domains. Our main motivation are inference
problems, in which the unknown belongs to a Banach space $Y$. Two
examples could be the following:
\begin{itemize}
\item {\bf Groundwater flow}: Consider a porous medium in a domain
  $\D\subseteq\R^3$. Given observations of the subsurface flow, we are
  interested in the permeability (hydraulic conductivity) of the medium
  in $\D$. The physical system is described by an elliptic partial
  differential equation, and the unknown quantity describing the permeability can be
  modelled as a function $\psi\in L^\infty(\D)=Y$ \cite{groundwater}.
\item {\bf Inverse scattering}: Suppose that
  $\D_{\rm scat}\subseteq\R^3$ is filled by a perfect conductor and
  illuminated by an electromagnetic wave. Given measurements of the
  scattered wave, we are interested in the shape of the scatterer
  $\D_{\rm scat}$. Assume that this domain can be described as the
  image of some bounded reference domain $\D\subseteq\R^3$ under a
  bi-Lipschitz transformation $\psi:\D\to\R^3$, i.e.,
  $\D_{\rm scat}=\psi(\D)$.
  The unknown is then the function $\psi\in W^{1,\infty}(\D)=Y$. We
  describe the forward model in \cite{JSZ16}.
\end{itemize}

The Bayesian approach to these problems is to model $\psi$ as a
$Y$-valued random variable and determine the distribution of $\psi$
conditioned on a noisy observation of the system. Bayes'
theorem can be used to specify this ``posterior'' distribution via the
prior and the likelihood. The prior is a measure on $Y$
that represents our information on $\psi\in Y$ before making an
observation. Mathematically speaking, assuming that the observation
and the unknown follow some joint distribution, the prior is the
marginal distribution of the unknown $\psi$. %
The goal is to explore the posterior and in this way to make inferences
about $\psi$. We refer to \cite{MR3839555} for more details on the
general methodology of Bayesian inversion in Banach spaces.

For the analysis and implementation of such methods, instead of
working with (prior and posterior) measures on the Banach space $Y$,
it can be convenient to parameterize the problem and work with measures
on $\R^\N$. To demonstrate this, choose a sequence $(\psi_j)_{j\in\N}$
in $Y$ and a measure $\mu$ on $\R^\N$. With $\bsy\dfn (y_j)_{j\in\N}\in\R^\N$
\begin{equation}\label{eq:prior}
  \Phi(\bsy)\dfn \sum_{j\in\N}y_j\psi_j
\end{equation}
we can formally define a prior measure on $Y$ as the pushforward
$\Phi_\sharp\mu$. Instead of inferring $\psi\in Y$ directly, we may
instead infer the coefficient sequence $\bsy=(y_j)_{j\in\N}\in\R^\N$,
in which case $\mu$ holds the prior information on the unknown
coefficients. %
These viewpoints are equivalent in the sense that the
conditional distribution of $\psi$ given an observation is the
pushforward, under $\Phi$, 
of the conditional distribution of $\bsy$ given the
observation. Under certain assumptions on the prior and
the space $Y$, the construction \eqref{eq:prior} arises naturally through the
Karhunen--Lo\`eve expansion; see, e.g., \cite{MR3308418,1509.07526}. In
this case the $y_j\in\R$ are uncorrelated random variables with
unit variance, and the $\psi_j$ are eigenvectors of the prior
covariance operator, with their norms equal to the square root of the
corresponding eigenvalues.

In this paper we concentrate on the special case where the
coefficients $y_j$ are known to belong to a bounded interval. Up to a
shift and a scaling this is equivalent to $y_j\in [-1,1]$, which will
be assumed throughout.
We refer to \cite[Sec.~2]{MR3839555} for the construction and further
discussion of such (bounded) priors. The goal then becomes to
determine and explore the posterior measure on $U\dfn [-1,1]^\N$. Denote
this measure by $\measii$ and let $\mu$ be %
the prior measure on $U$ such that $\measii\ll\mu$.  
Then the Radon-Nikodym derivative
$f_\measii\dfn \frac{\ddd\measii}{\ddd\mu}:U\to [0,\infty)$
exists.
Since the forward model (and thus the likelihood)
only depends on
$\Phi(\bsy)$ in the Banach space $Y$, $f_\measii$ must be of the type
\begin{equation}\label{eq:posterior}
  f_\measii(\bsy) = \frf_{\measii}(\Phi(\bsy))=\frf_{\measii}\Big(\sum_{j\in\N}y_j\psi_j\Big)
\end{equation}
for some $\frf_\measii:Y\to [0,\infty)$. We give a concrete example
in Ex.~\ref{ex:bayes}
where this relation holds. 

``Exploring'' the posterior refers to computing expectations and
variances w.r.t.\ $\measii$, or detecting areas of high probability
w.r.t.\ $\measii$. A standard technique to do so in high dimensions is
Monte Carlo---or in this context Markov chain Monte Carlo---sampling,
e.g., \cite{10.5555/1051451}. Another approach is via transport maps
\cite{MR3821485}. Let $\measi$ be another measure on $U$ from which
it is easy to sample. Then, a map $T:U\to U$ satisfying
$T_\sharp\measi=\measii$ (i.e.,
$\measii(A)=\measi(\set{\bsy}{T(\bsy)\in A})$ for all measurable $A$)
is called a transport map that pushes forward $\measi$ to $\measii$.
Such a $T$ has the property that if $\bsy\sim\measi$ then
$T(\bsy)\sim\measii$, and thus samples from $\measii$ can easily be
generated once $T$ has been computed. Observe that
$\Phi\circ T:U\to Y$ will then transform a sample from $\measi$ to a
sample from $\Phi_\sharp T_\sharp\measi=\Phi_\sharp\measii$, which is
the posterior %
in the Banach space $Y$.  Thus, given $T$, we can perform inference on
the quantity in the Banach space.

This motivates the setting we are investigating in this paper: for two
measures $\measi$ and $\measii$ on $U$, such that their densities are
of the type \eqref{eq:posterior} for a smooth (see
Sec.~\ref{sec:main}) function $\frf_\measii$, we are interested in the
approximation of $T:U\to U$ such that $T_\sharp\measi=\measii$. More
precisely, we will discuss the approximation of the so-called
Knothe--Rosenblatt (KR) transport by rational functions. The reason
for using rational functions (rather than polynomials) is to guarantee
that the resulting approximate transport is a bijection from $U\to U$. The rate of
convergence will in particular depend on the decay rate of the
functions $\psi_j$. If \eqref{eq:prior} is a Karhunen--Lo\`eve
expansion, this is the decay rate of the square root of the
eigenvalues of the covariance operator of the prior. The faster this
decay, the larger the convergence rate will be. {%
  The reason for analyzing the triangular KR transport is its wide use
  in practical algorithms \cite{MR2972870,
    spantini2018inference,jaini2019sum,wehenkel2019unconstrained}, and
  the fact that its concrete construction makes it amenable to a
  rigorous analysis.}

Sampling from high-dimensional distributions by transforming a
(usually lower-dimensional) ``latent'' variable into a sample from the
desired distribution is a standard problem in machine learning. It is
tackled by methods such as generative adversarial networks \cite{goodfellow2014generative}
and variational autoencoders \cite{doersch2016tutorial}.
In the setting above, the
high-dimensional distribution is the posterior on $Y$. We will show
that under the assumptions of this paper, it is possible to
approximately sample from this distribution by transforming a low
dimensional latent variable, and without suffering from the curse of
dimensionality.
While Bayesian inference is our motivation, for the rest of the
manuscript the presentation remains in an abstract setting, and our
results therefore have ramifications on the broader task of
transforming high-dimensional distributions.

\subsection{Contributions and outline}
In this manuscript we generalize the analysis of \cite{zm1} to the
infinite dimensional case. Part of the proofs are based on the results
in \cite{zm1}, which we recall in the appendix where appropriate to
improve readability.

In Sec.~\ref{sec:main} we provide a short description of our main
result. Sec.~\ref{SEC:Tinf} discusses the KR map
in infinite dimensions. Its well-definedness in infinite dimensions has
been established in \cite{bogachevtri}. In Thm.~\ref{THM:KNOTHEINF} we
additionally give a formula for the pushforward density assuming
continuity of the densities w.r.t.\ the product topology. In
Sec.~\ref{sec:infanalyticity} we analyze the regularity of the KR
transport. %
The fact that a transport inherits the smoothness of
the densities is known for certain function classes: for example,
in the case of $C^k$ densities, \cite{MR3349831} shows that the
optimal transport also belongs to $C^k$, and a similar statement holds
for the KR transport; see for example \cite[Remark
2.19]{santambrogio}. In Prop.~\ref{PROP:COR:DININF}, assuming analytic
densities we show analyticity of the KR transport. Furthermore, and
more importantly, we carefully examine the domain of holomorphic
extension to the complex numbers.
These results are
exploited in Sec.~\ref{sec:polinfty} to show convergence of rational
function approximations to $T$ in Thm.~\ref{THM:TINF}. This result
proves a \emph{dimension-independent} higher-order convergence rate for
the transport of measures supported on infinite dimensional spaces (which
need not be supported on finite dimensional
subspaces). %
In this result, \emph{all} occurring constants (not just the
convergence rate) are controlled independently of the dimension.
In Sec.~\ref{sec:measinfty} we show that this implies convergence of
the pushforward measures (on $U$ and on the Banach space $Y$) in the
Hellinger distance, the total variation distance, the KL divergence,
and the Wasserstein distance. These results are formulated in
Thm.~\ref{THM:MEASCONVINF} and Thm.~\ref{thm:wassersteinconv}. To
prove the latter, in Prop.~\ref{PROP:WASSERSTEIN} we slightly
extend a statement from \cite{MR4120535} to compact Polish spaces
to show that the Wasserstein distance between two pushforward measures
can be bounded by the maximal distance of the two maps pushing forward
the initial measure. Finally, we show that it is possible to compute
approximate samples for the pushforward measure in the Banach space
$Y$, by mapping a low-dimensional reference sample to the Banach
space; see Cor.~\ref{COR:MEASCONVINF}.
All proofs can be found in the appendix.

  \section{Main result}\label{sec:main}
  Let for $k\in\N$
  \begin{equation}\label{eq:U}
    \U{k}\dfn [-1,1]^k\qquad\text{and}\qquad U\dfn [-1,1]^\N
  \end{equation}
  where these sets are equipped with the product topology and the
  Borel $\sigma$-algebra, which coincides with the product
  $\sigma$-algebra \cite[Lemma 6.4.2 (ii)]{bogachev}. Additionally, let
  $U_0\dfn \emptyset$. Denote by $\lambda$ the Lebesgue measure on
  $[-1,1]$ and by
  \begin{equation}\label{eq:mu}
    \mu=\bigotimes_{j\in\N}\frac{\lambda}{2}
  \end{equation}
  the infinite product measure. Then $\mu$ is a (uniform) probability
  measure on $U$. By abuse of notation for $k\in\N$ we additionally 
  denote $\mu=\otimes_{j=1}^k\frac{\lambda}{2}$, where $k$ will always
  be clear from context.

  For a \emph{reference} $\measi\ll\mu$ and a \emph{target} measure
  $\measii\ll\mu$ on $U$, we investigate the smoothness and
  approximability of the %
  KR transport $T:U\to U$ satisfying $T_\sharp\measi=\measii$; the
  notation $T_\sharp\measi$ refers to the pushforward measure defined
  by
  $T_\sharp\measi(A)\dfn \measi(\set{T(\bsy)\in A}{\bsy\in U})$ for
  all measurable $A\subseteq U$.  While in general there exist
  multiple maps $T:U\to U$ pushing forward $\measi$ to $\measii$, the
  KR transport is the unique such map satisfying \emph{triangularity}
  and \emph{monotonicity}.
  Triangularity refers to the $k$th component $T_k$ of $T=(T_k)_{k\in\N}$
  being a function of the variables $x_1,\dots,x_k$ only, i.e.,
  $T_k:\U{k}\to \U{1}$ for all $k\in\N$. Monotonicity means that
  $x_k\mapsto T_k(x_1,\dots,x_{k-1},x_k)$ is monotonically increasing
  on $\U{1}$ for every $k\in\N$ and every fixed
  $(x_1,\dots,x_{k-1})\in \U{k}$.

  Absolute continuity of $\measi$ and $\measii$ w.r.t.\ $\mu$ imply
  existence of the Radon-Nikodym derivatives
  \begin{equation}
    f_\measi\dfn \frac{\ddd\measi}{\ddd\mu}\qquad\text{and}\qquad
    f_\measii\dfn \frac{\ddd\measii}{\ddd\mu}
  \end{equation}
  which will also be referred to as the densities of these
  measures. Assuming for the moment existence of the KR transport $T$,
  approximating $T$ requires approximating the \emph{infinitely many}
  functions $T_k:\U{k}\to \U{1}$, $k\in\N$. This, and the fact that
  the domain $\U{k}$ of $T_k$ becomes increasingly high dimensional as
  $k\to\infty$, makes the problem quite challenging.

  For these reasons, further assumptions on $\measi$ and $\measii$ are
  necessary.
  Typical
  requirements imposed %
  on the measures guarantee some form of intrinsic
  low dimensionality. Examples include densities belonging to
  certain reproducing kernel Hilbert spaces, or to other function classes
  of sufficient regularity. In this paper we concentrate on the
  latter. %
  As is well-known, if
  $T_k:\U{k}\to \U{1}$ belongs to $C^k$, then it can be uniformly
  approximated with the $k$-independent convergence rate of $1$, for
  instance with multivariate polynomials. %
  The convergence rate to approximate $T_k$ then does not deteriorate
  with increasing $k$, but the constants in such error bounds
  usually still depend exponentially on $k$. Moreover, as
  $k\to\infty$, this line of argument requires the components of the
  map to become arbitrarily regular. For this reason, in the present
  work, where $T=(T_k)_{k\in\N}$, it is not unnatural to restrict
  ourselves to transports that are $C^\infty$. More precisely, we in
  particular assume \emph{analyticity} of the densities $f_\measi$ and
  $f_\measii$, which in turn implies analyticity of $T$ as we shall
  see.  This will allow us to control all occurring constants
  \emph{independent} of the dimension, and approximate the whole map
  $T:U\to U$ using only finitely many degrees of freedom in our
  approximation.

  Assume in the following that $Z$ is a Banach space with
  complexification $Z_\C$; see, e.g., \cite{padraig,munoz99} for the
  complexification of Banach spaces. We may think of $Z$ and $Z_\C$ as
  real and complex valued function spaces, e.g., $Z=L^2([0,1];\R)$ and
  $Z_\C=L^2([0,1];\C)$.
  To guarantee analyticity and the structure in \eqref{eq:prior} we
  consider densities $f$ of the following type:
  \begin{assumption}\label{ass:density}
    For constants $p\in (0,1)$, $0<M\le L<\infty$, a sequence
    $(\psi_j)_{j\in\N}\subseteq Z$, and a differentiable function
    $\frf:O_Z\to \C$ with $O_Z\subseteq Z_\C$ open, the following hold:
    \begin{enumerate}[label=(\alph*)]
    \item $\sum_{j\in\N}\norm[Z]{\psi_{j}}^p<\infty$,
    \item $\sum_{j\in\N}y_j\psi_{j}\in O_Z$ for all $\bsy\in U$,
    \item $\frf(\sum_{j\in\N}y_j\psi_{j})\in\R$ for all $\bsy\in U$,
    \item\label{item:densityML} $M= \inf_{\psi\in O_Z}|\frf(\psi)|\le \sup_{\psi\in
        O_Z}|\frf(\psi)| = L$.
    \end{enumerate}
    The function $f:U\to\R$ given by
  \begin{equation}\label{eq:density}
    f(\bsy)\dfn \frf\bigg(\sum_{j\in\N}\psi_{j}y_j\bigg)
  \end{equation}
  satisfies $\int_U f(\bsy)\dd\mu(\bsy)=1$.
\end{assumption}

\begin{assumption}\label{ass:densities}
  For two sequences $(\psi_{*,j})_{j\in\N}\in Z$ with
  $(*,Z)\in \{(\measi,X),(\measii,Y)\}$, the functions
  \begin{equation*}
    f_\measi(\bsy)=\frf_\measi\bigg(\sum_{j\in\N}y_j\psi_{\measi,j}\bigg),\qquad
    f_\measii(\bsy)=\frf_\measii\bigg(\sum_{j\in\N}y_j\psi_{\measii,j}\bigg)
  \end{equation*} %
  both satisfy Assumption \ref{ass:density} for some fixed constants
  $p\in (0,1)$ and $0<M\le L<\infty$.
\end{assumption}

The summability parameter $p$ determines the decay rate of the
functions $\psi_j$---the smaller $p$ the stronger the decay of the
$\psi_j$. Because $p<1$, the argument of $\frf$ in \eqref{eq:density}
is well-defined for $\bsy\in U$ since
$\sum_{j\in\N}|y_j|\norm[Z]{\psi_{j}}<\infty$. %

Our main result is about the existence and approximation of the
KR-transport $T:U\to U$ satisfying
$T_\sharp\measi=\measii$. We state the result here in a simplified
form; more details will be given in Thm.~\ref{THM:TINF},
Thm.~\ref{THM:MEASCONVINF}, and Thm.~\ref{thm:wassersteinconv}.
We only mention that the trivial approximation
$T_k(x_1,\dots,x_k)\simeq x_k$ is interpreted as not requiring any
degrees of freedom in the following theorem.
  
  \begin{theorem}\label{thm:main}
    Let $f_\measi:U\to (0,\infty)$ and $f_\measii:U\to (0,\infty)$ be
    two %
    probability densities as in Assumption
    \ref{ass:densities} for some $p\in (0,1)$. Then there exists a
    unique triangular, monotone, and bijective map $T:U\to U$
    satisfying $T_\sharp\measi=\measii$.

    Moreover, for $N\in\N$ there exists a space of rational functions
    employing $N$ degrees of freedom, and a bijective, monotone, and
    triangular $\tilde T:U\to U$ in this space such that
    \begin{equation}\label{eq:error}
      {\rm dist}(\tilde T_\sharp\measi,\measii)\le C N^{-\frac{1}{p}+1}.
    \end{equation}
    Here $C$ is a constant independent of $N$ and ``${\rm dist}$''
    may refer to the total variation distance, the Hellinger distance,
    the KL divergence, or the Wasserstein distance.
  \end{theorem}

  Equation \eqref{eq:error} shows a dimension-independent
  convergence rate (indeed our transport is defined on the infinite
  dimensional domain $U=[-1,1]^\N$), so that the curse of
  dimensionality is overcome. %
  The rate of algebraic convergence becomes arbitrarily large as
  $p\in (0,1)$ in Assumption \ref{ass:density} becomes small.  The
  convergence rate $\frac{1}{p}-1$ in Thm.~\ref{thm:main} is
  well-known for the approximation of functions as in
  \eqref{eq:density} by sparse polynomials, e.g.,
  \cite{CDS10,CDS11,CCS15}; also see Rmk.~\ref{rmk:bpe}. There is a
  key difference to earlier results dealing with the approximation of
  such functions: we do not approximate the function $f:U\to \R$ in
  \eqref{eq:density}, but instead we approximate the transport
  $T:U\to U$, i.e., an infinite number of functions. Our main
  observation in this paper is that the sparsity of the densities
  $f_\measi$ and $f_\measii$ carries over to the transport. Even
  though it has infinitely many components, $T$ can still be
  approximated very efficiently if the ansatz space is carefully
  chosen and tailored to the specific densities. In addition to
  showing the error convergence \eqref{eq:error}, in
  Thm.~\ref{THM:TINF} we give concrete ansatz spaces achieving this
  convergence rate.  These ansatz spaces can be computed in linear
  complexity and may be used in applications.

  The main application for our result is to provide a method to sample
  from the target $\measii$ or the pushforward $\Phi_\sharp\measii$ in
  the Banach space $Y$, where
  $\Phi(\bsy)=\sum_{j\in\N}y_j\psi_{\measii,j}$.  Given an
  approximation $\tilde T=(\tilde T_j)_{j\in\N}$ to $T$, this is
  achieved via $\Phi(\tilde T(\bsy))$ for $\bsy\sim\measi$.  It is
  natural to truncate this expansion, which yields
  \begin{equation*}
    \sum_{j=1}^s \tilde T_j(y_1,\dots,y_j)\psi_{\measii,j}
  \end{equation*}
  for some truncation parameter
  $s\in\N$ and $(y_1,\dots,y_s)\in\U{s}$. This map transforms
  {a sample from} %
  a distribution on the $s$-dimensional space $\U{s}$ to a
  sample from an infinite dimensional distribution on
  $Y$. In Cor.~\ref{COR:MEASCONVINF} we show that the error
  {of this truncated representation} in the Wasserstein distance converges
  with the same rate as given in
  Thm.~\ref{thm:main}.

  \begin{remark}
    The reference $\measi$ is a ``simple'' measure whose main
    purpose is to allow for easy sampling. %
    One possible choice for $\measi$ (that we have in mind throughout
    this paper) is the uniform measure $\mu$. It trivially satisfies
    Assumption \ref{ass:density} with $\frf_\measi:\C\to\C$ being the
    constant $1$ function (and, e.g., $\psi_{\measi,j}=0\in\C$).
  \end{remark}
  \begin{remark}
    Even though we can think of $\measi$ as being $\mu$, we
    formulated Thm.~\ref{thm:main} in more generality, mainly for the
    following reason:
    Since the assumptions on $\measi$ and $\measii$
    are the same, we may switch their roles. Thus Thm.~\ref{thm:main}
    can be turned into a statement about the inverse transport
    $S\dfn T^{-1}:U\to U$, which can also be approximated at the rate
    $\frac{1}{p}-1$.
  \end{remark}

  \begin{example}[Bayesian inference]\label{ex:bayes}
    For a Banach space
    $Y$ (``parameter space'') and a Banach space
    $\mathcal{X}$ (``solution space''), let $\scr{u}:O_Y\to
    \mathcal{X}_\C$ be a complex differentiable \emph{forward operator}
    that takes values in (the $\R$-vector space) $\mathcal{X}$
    for inputs in (the open subset of the $\R$-vector space) $Y\cap O_Y$.
    Here $O_Y\subseteq
    Y_\C$ is some nonempty open set. Let $G:\mathcal{X}\to
    \R^m$ be a bounded linear \emph{observation operator}. For some
    unknown $\psi\in
    Y$ we are given a noisy observation of the system in the form
    \begin{equation*}
      \varsigma = G(\scr{u}(\psi))+\eta\in\R^m,
    \end{equation*}
    where
    $\eta\sim\mathcal{N}(0,\Gamma)$ is a centered Gaussian random
    variable with symmetric positive definite covariance
    $\Gamma\in\R^{m\times m}$. The goal is to recover
    $\psi$ given the measurement $\varsigma$.

    To formulate the Bayesian inverse problem, we first
    fix a prior:
    Let $(\psi_{j})_{j\in\N}$ be
    a summable sequence of linearly independent elements in $Y$. With
    \begin{equation*}
      \Phi(\bsy)\dfn \sum_{j\in\N}y_j\psi_j
    \end{equation*}
    and the uniform measure $\mu$ on $U$, we choose the 
    prior  $\Phi_\sharp\mu$ on $Y$. Determining $\psi$ within the set
    $\set{\Phi(\bsy)}{\bsy\in U}\subseteq Y$ is equivalent to
    determining the coefficient sequence $\bsy\in U$. Assuming
    independence of $\bsy\sim \mu$ and $\eta\sim\CN(0,\Gamma)$, the
    distribution of $\bsy$ given $\varsigma$ (the posterior) can then
    be characterized by its density w.r.t.\ $\mu$, which, up to a
    normalization constant, equals
    \begin{equation}\label{eq:post}
      \exp\left(\Bigg(\varsigma-G\Big(\scr{u}\Big(\sum_{j\in\N}y_j\psi_j\Big)\Big)\Bigg)^\top\Gamma^{-1}\Bigg(\varsigma-G\Big(\scr{u}\Big(\sum_{j\in\N}y_j\psi_j\Big)\Big)\Bigg)\right).
    \end{equation}
    This posterior density is of the form \eqref{eq:density} and the
    corresponding measure $\measii$ can be chosen as a target in
    Thm.~\ref{thm:main}. Given $T$ satisfying
    $T_\sharp\measi=\measii$, we may then explore $\measii$ to perform
    inference on the unknown $\bsy$ (or its image $\Phi(\bsy)$
    in the Banach space $Y$); see for instance
    \cite[Sec.~7.4]{zm1}.  For more details on the rigorous derivation
    of \eqref{eq:post} we refer to \cite{CSStBIP2012} and in
    particular \cite[Sec.~3]{MR3839555}.
  \end{example}

  \begin{remark}\label{rmk:bpe}
    Functions as in Assumption \ref{ass:density} belong to the set of
    so-called ``$(\bsb,p,\eps)$-holomorphic'' functions; see, e.g.,
    \cite{CCS15}. This class contains infinite parametric functions
    that are holomorphic in each argument $y_j$, and exhibit some
    growth in the domain of holomorphic extension as $j\to\infty$. The
    results of the present paper and the key arguments remain valid if
    we replace Assumption \ref{ass:density} with the
    $(\bsb,p,\eps)$-holomorphy assumption. Since most relevant
    examples of such functions are of the specific type
    \eqref{eq:density}, we restrict the discussion to this case in
    order to avoid technicalities.
  \end{remark}

  }%
  \section{The Knothe--Rosenblatt transport in infinite
    dimensions}\label{SEC:Tinf}
  Recall that we consider the product topology on $U=[-1,1]^\N$.
  Assume that $f_\measi\in C^0(U;\R_+)$ and $f_\measii\in C^0(U;\R_+)$
  are two positive probability densities.  Here $\R_+\dfn (0,\infty)$,
  and $C^0(U;\R_+)$ denotes the continuous functions from
  $U\to\R_+$. We now recall the construction of the KR
  map.

    For $\bsy=(y_j)_{j\in\N}\in \C^\N$ and
    $1\le k\le n<\infty$ let
  \begin{equation}\label{eq:slices}
    \bsy_{[k]}\dfn (y_j)_{j=1}^k,\qquad
    \bsy_{[k:n]}\dfn (y_j)_{j=k}^n,\qquad
    \bsy_{[n:]}\dfn (y_j)_{j\ge n}.
  \end{equation}
  For $*\in\{\measi,\measii\}$ and $\bsy\in U$ define
  \begin{subequations}\label{eq:fk2}
  \begin{equation}
    \hat f_{*,0}(\bsy)\dfn 1
  \end{equation}
  and for $k\in\N$
  \begin{equation}
    \hat f_{*,k}(\bsy_{[k]})\dfn \int_{U}f_*(\bsy_{[k]},\bst) \dd\mu(\bst)>0,\qquad
    f_{*,k}(\bsy_{[k]})\dfn \frac{\hat f_{*,k}(\bsy_{[k]})}{\hat
      f_{*,k-1}(\bsy_{[k-1]})}>0.
  \end{equation}
  \end{subequations}
  Then, $\bsy_{[k]}\mapsto \hat f_{\measi,k}(\bsy_{[k]})$ is the
  marginal density of $\measi$ in the first $k$ variables
  $\bsy_{[k]}\in \U{k}$, and we denote the corresponding measure on
  $\U{k}$ by $\measi_k$. Similarly,
  $y_k\mapsto f_{\measi,k}(\bsy_{[k-1]},y_k)$ is the conditional
  density of $y_k$ given $\bsy_{[k-1]}$, and the corresponding measure
  on $\U{1}$ is denoted by $\measi_k^{\bsy_{[k-1]}}$. The same holds
  for the densities of $\measii$, and we use the analogous notation
  $\measii_k$ and $\measii_k^{\bsy_{[k-1]}}$ for the {marginal and
  conditional measures.}

  Recall that for two
  atomless measures $\eta$ and $\nu$ on $\U{1}$ with distribution
  functions $F_\eta:\U{1}\to [0,1]$ and $F_\nu:\U{1}\to [0,1]$,
  $F_\eta^{-1}\circ F_\nu:\U{1}\to \U{1}$ pushes forward
  $\nu$ to $\eta$, as is easily checked, e.g.,
  \cite[Thm.~2.5]{santambrogio}. In case $\eta$ and $\nu$ have
  positive densities on $\U{1}$, this map is the unique strictly
  monotonically increasing such function.
  With this in mind, the KR-transport can be constructed as follows:
  Let $T_1:\U{1}\to \U{1}$ be the (unique) monotonically increasing
  transport satisfying
  \begin{subequations}\label{eq:T}
    \begin{equation}
      (T_1)_\sharp \measi_1 = \measii_1.
    \end{equation}
    Analogous to \eqref{eq:slices} denote
    $T_{[k]}\dfn (T_j)_{j=1}^k:\U{k}\to \U{k}$. Let inductively
    for any $\bsy\in U$, $T_{k+1}(\bsy_{[k]},\cdot):\U{1}\to \U{1}$
    be the (unique) monotonically increasing transport such that
    \begin{equation}
      (T_{k+1}(\bsy_{[k]},\cdot))_\sharp \measi_{k+1}^{\bsy_{[k]}}
      =\measii_{k+1}^{T_{[k]}(\bsy_{[k]})}.
    \end{equation}
  \end{subequations}
  Note that $T_{k+1}:\U{{k+1}}\to \U{1}$ and thus
  $T_{[k+1]}=(T_j)_{j=1}^{k+1}:\U{{k+1}}\to \U{{k+1}}$. It can
  then be shown that for any $k\in\N$ \cite[Prop.~2.18]{santambrogio}
  \begin{equation}\label{eq:Tfinite}
    (T_{[k]})_\sharp \measi_k = \measii_k.
  \end{equation}

  By induction this construction yields a map $T\dfn (T_k)_{k\in\N}$
  where each $T_k:\U{k}\to \U{1}$ satisfies that
  $T_k(\bsy_{[k-1]},\cdot):\U{1}\to \U{1}$ is strictly monotonically
  increasing and bijective.
  This implies that $T:U\to U$ is bijective, as follows. First, to show \emph{injectivity}: let
  $\bsx\neq \bsy\in U$ and $j=\argmin\set{i}{x_i\neq y_i}$. Since
  $t\mapsto T_j(x_1,\dots,x_{j-1},t)$ is bijective,
  $T_j(x_1,\dots,x_{j-1},x_j)\neq T_j(x_1,\dots,x_{j-1},y_j)$ and thus
  $T(\bsx)\neq T(\bsy)$. Next, to show \emph{surjectivity:} fix $\bsy\in U$.
  Bijectivity of $T_1:\U{1}\to \U{1}$ implies existence of
  $x_1\in\U{1}$ such that $T_1(x_1)=y_1$. Inductively choose $x_j$
  such that $T_j(x_1,\dots,x_j)=y_j$.  Then $T(\bsx)=\bsy$. Thus:

  \begin{lemma}\label{lemma:bijective}
    Let $T=(T_k)_{k\in\N}:U\to U$ be triangular. If
    $t\mapsto T_k(\bsy_{[k-1]},t)$ is bijective from
    $\U{1}\to \U{1}$ for every $\bsy\in U$ and $k\in\N$, then
    $T:U\to U$ is bijective.
  \end{lemma}

  The continuity assumption on the densities guarantees that the
  marginal densities on $\U{k}$ converge uniformly to the full density,
  as we show next. This indicates that in principle it is possible to
  approximate the infinite dimensional transport map by restricting to
  finitely many dimensions.

  \begin{lemma}\label{LEMMA:FC0}
    Let $f\in C^0(U;\R_+)$, and let $\hat f_k$ and
    $f_k$ be as in \eqref{eq:fk2}. Then
    \begin{enumerate}
    \item $f$ is measurable and $f\in L^2(U,\mu)$,
    \item $\hat f_{k}\in C^0(\U{k};\R_+)$ and
      $f_{k}\in C^0(\U{k};\R_+)$ for every $k\in\N$,
    \item it holds
      \begin{equation}\label{eq:unifconv}
        \lim_{k\to\infty}\sup_{\bsy\in U}|\hat
        f_{k}(\bsy_{[k]})-f(\bsy)|=0.
      \end{equation}
    \end{enumerate}
  \end{lemma}

  Throughout what follows $T$ always stands for the KR transport
  defined in \eqref{eq:T}. Next we show that $T$ indeed pushes forward
  $\measi$ to $\measii$, and additionally we provide a formula for the
  transformation of densities. In the following
  $\partial_jg(\bsx)\dfn \frac{\partial}{\partial
    x_j}g(\bsx)$. Furthermore, we call $f:U\to\R$ a \emph{positive
    probability density} if $f(\bsy)>0$ for all $\bsy\in U$ and
  $\int_U f(\bsy)\dd\mu(\bsy)=1$.

\begin{theorem}\label{THM:KNOTHEINF}
  Let $f_\measii$, $f_\measi \in C^0(U;\R_+)$ be two positive
  probability densities.
  Then
  \begin{enumerate}
  \item\label{item:KT} $T=(T_j)_{j\in\N}:U\to U$ is measurable,
    bijective and satisfies $T_\sharp\measi=\measii$,
  \item\label{item:rdder} for each $k\in\N$ holds
    $\partial_kT_k(\bsy_{[k]})\in C^0(\U{k};\R_+)$ and
    \begin{equation}\label{eq:detinf}
      \det dT(\bsy)\dfn \lim_{n\to\infty}\prod_{j=1}^n\partial_jT_j(\bsy_{[j]})\in C^0(U;\R_+)
    \end{equation}
    is well-defined (i.e., converges in $C^0(U;\R_+)$). Moreover
    \begin{equation}
      f_{\measii}(T(\bsy))\det dT(\bsx)=f_\measi(\bsy)
      \qquad\forall\bsy\in U.
    \end{equation}
  \end{enumerate}
\end{theorem}

\begin{remark}\label{rmk:knotheS}
  Switching the roles of $f_\measi$ and
  $f_\measii$, for %
  $S=T^{-1}$ it holds $f_\measi(S(\bsy))\det
  dS(\bsy)=f_\measii(\bsy)$ for all $\bsy\in U$, where $\det
  dS(\bsy)\dfn
  \lim_{n\to\infty}\prod_{j=1}^n\partial_jS_j(\bsy_{[j]})$ is
  well-defined.
\end{remark}

\section{Analyticity of $T$}\label{sec:infanalyticity}
In this section we investigate the domain of analytic extension of
$T$. To state our results, for $\delta>0$
and $D\subseteq\C$
we introduce the complex
sets
\begin{equation*}
  \cB_\delta\dfn \set{z\in\C}{|z|<\delta}\qquad\text{and}\qquad
  \cB_\delta(D)\dfn \set{z+y}{z\in\cB_\delta,~y\in D},
\end{equation*}
and for $k\in\N$ and $\bsdelta\in (0,\infty)^k$
\begin{equation*}
  \cB_\bsdelta\dfn \bigtimes_{j=1}^k\cB_{\delta_j}\qquad\text{and}\qquad
  \cB_\bsdelta(D) \dfn \bigtimes_{j=1}^k \cB_{\delta_j}(D),
\end{equation*}
which are subsets of $\C^k$.
Their closures will be denoted by $\bar\cB_\delta$, etc. If we write
$\cB_\bsdelta(\U{1})\times U$ we mean elements $\bsy\in\C^\N$ with
$y_j\in \cB_{\delta_j}(\U{1})$ for $j\le k$ and $y_j\in \U{1}$
otherwise. Subsets of $\C^\N$ are always equipped with the product
topology.

In this section we analyze the domain of holomorphic extension of each
component $T_k:\U{k}\to\U{1}$ of $T$ to subsets of $\C^k$. The reason
why we are interested in such statements, is that they allow to upper
bound the expansion coefficients w.r.t.\ certain polynomial bases: For
a multiindex $\bsnu\in \N_0^k$ (where $\N_0=\{0,1,2,\dots\}$) let
$L_\bsnu(\bsy)=\prod_{j=1}^kL_{\nu_j}(y_j)$ be the product of the one
dimensional Legendre polynomials normalized in $L^2(\U{1},\mu)$.  Then
$(L_\bsnu)_{\bsnu\in\N_0^k}$ forms an orthonormal basis of
$L^2(\U{k},\mu)$.  Hence we can expand
$\partial_kT_k(\bsy_{[k]})=\sum_{\bsnu\in\N_0^k}l_{k,\bsnu} L_\bsnu(\bsy_{[k]})$
for $\bsy\in U$ and with the Legendre coefficients
\begin{equation}\label{eq:lkbsnu}
  l_{k,\bsnu}=\int_{\U{k}}\partial_kT_k(\bsy_{[k]})L_\bsnu(\bsy_{[k]})\in\R.
\end{equation}
Analyticity of $T_k$ (and thus of $\partial_kT_k$) on the set
$\cB_\bsdelta(\U{1})$ implies bounds of the type (see Lemma
\ref{lemma:legest})
\begin{equation}\label{eq:legbound}
  |l_{k,\bsnu}|\le C\norm[{L^\infty(\cB_\bsdelta(\U{1}))}]{\partial_kT_k}\prod_{j=1}^k (1+\delta_j)^{-\nu_j}.
\end{equation}
Here $C$ in particular depends on %
$\min_j \delta_j>0$. The exponential decay in each $\nu_j$
leads to exponential convergence of truncated sparse Legendre
expansions. Once we have approximated $\partial_kT_k$, we
integrate this term in $x_k$ to obtain an approximation to $T_k$.  The
reason for not approximating $T_k$ directly is explained after
Prop.~\ref{PROP:COR:DININF} below; see \eqref{eq:l0trivial}.
The size of the holomorphy
domain (the size of $\bsdelta$) determines the constants in these
estimates---the larger the entries of $\bsdelta$, the smaller the upper
bound \eqref{eq:legbound} and the faster the convergence.

We are now in position to present our main technical tool to find
suitable holomorphy domains of each $T_k$ (or equivalently
$\partial_kT_k$). We will work under the following assumption on
the two densities $f_\measi:U\to (0,\infty)$ and
$f_\measii:U\to (0,\infty)$.  The assumption is a modification of
\cite[Assumption 3.5]{zm1}.

\begin{assumption}\label{ass:infdens}
  For constants $C_1$, $M>0$, $L<\infty$,
  $k\in\N$, and $\bsdelta\in (0,\infty)^k$, the following hold:
    \begin{enumerate}[label=(\alph*)]
    \item\label{item:cordinN:0inf} $f\in C^0(\cB_{\bsdelta}(\U{1})\times U;\C)$
      and $f:U\to\R_+$ is a probability
    density,
  \item\label{item:cordinN:1inf}
    $\bsx\mapsto f(\bsx,\bsy)\in C^1(\cB_{\bsdelta}(\U{1});\C)$
    for all $\bsy\in U$,
  \item\label{item:cordinN:2inf}
    $ M\le |f(\bsy)|\le  L$ for all
    $\bsy\in \cB_{\bsdelta}(\U{1})\times U$,    
  \item\label{item:cordinN:3inf} 
    $\sup_{\bsy\in
      \cB_{\bsdelta}\times\{0\}^\N}|f(\bsx+\bsy)-f(\bsx)| \le
    C_1$
    for all $\bsx\in U$,    
  \item\label{item:cordinN:4inf} %
    $\sup_{\bsy\in
        \cB_{\bsdelta_{[j]}}\times \{0\}^{\N}}|f(\bsx+\bsy)-f(\bsx)|\le C_1 \delta_{j+1}$
      for all $\bsx\in U$ and $j\in\{1,\dots,k-1\}$.
  \end{enumerate}
\end{assumption}

Such densities yield certain holomorphy domains for $T_k$ as we show
in the next proposition, which is an infinite dimensional version of
\cite[Theorem 3.6]{zm1}.

\begin{proposition}\label{PROP:COR:DININF}
  Let $k\in\N$, $\bsdelta\in (0,\infty)^k$ and $0< M\le
  L<\infty$. There exist $C_1>0$, $C_2\in (0,1]$ and $C_3>0$ solely
  depending on $M$ and $L$ (but not on $k$ or $\bsdelta$) such that
  if $f_\measi$ and $f_\measii$ satisfy Assumption \ref{ass:infdens}
  with $C_1$, $M$, $L$ and $\bsdelta$, then:
  % with $C_1$, $M$ and $L$, then:

  With $\bszeta=(\zeta_j)_{j=1}^k$ defined by
  \begin{equation}\label{eq:zetak2inf}
    \zeta_{j}\dfn  C_2 \delta_{j} \qquad \forall j\in\{1,\dots,k\},
  \end{equation}
  it holds for all $j\in\{1,\dots,k\}$ with
  $R_j\dfn \partial_jT_j$ (with $T$ as in \eqref{eq:T}) that
  \begin{enumerate}
  \item\label{item:cordinN:ainf}
    $R_j\in C^1(\cB_{\bszeta_{[j]}}(\U{1});\cB_{ C_3}(1))$ and
    $\Re(R_j(\bsx))\ge \frac{1}{C_3}$ for all
    $\bsx\in \cB_{\bszeta_{[j]}}(\U{1})$,
  \item\label{item:cordinN:binf} if $j\ge 2$,
    $R_j:\cB_{\bszeta_{[j-1]}}(\U{1})\times \U{1}\to
    \cB_{\frac{C_3}{\delta_j}}(1)$.
  \end{enumerate}
\end{proposition}

Let us sketch how this result can be used to show that $T_k$ can be
approximated by polynomial expansions. In appendix
\ref{app:lemma:scrfhol} we will verify Assumption \ref{ass:infdens}
for densities as in \eqref{eq:density}. Prop.~\ref{PROP:COR:DININF}
\ref{item:cordinN:ainf} then provides a holomorphy domain for
$\partial_kT_k$, and together with \eqref{eq:legbound} we can bound
the expansion coefficients $l_{k,\bsnu}$ of
$\partial_kT_k=\sum_{\bsnu\in\N_0^k}l_{k,\bsnu} L_\bsnu(\bsy)$.
However, there is a catch: In general one can find different
$\bsdelta$ such that Assumption \ref{ass:infdens} holds. The
difficulty is to choose $\bsdelta$ in a way that depends on $\bsnu$ to
obtain a possibly sharp bound in \eqref{eq:legbound}. To do so we will
use ideas from, e.g., \cite{CCS15} where similar calculations were
made.

The outlined argument based on Prop.~\ref{PROP:COR:DININF}
\ref{item:cordinN:ainf} suffices to prove convergence of sparse
polynomial expansions in the finite dimensional case; see
\cite[Thm.~4.6]{zm1}. %
In the infinite dimensional case where we want to approximate
$T=(T_k)_{k\in\N}$ with only finitely many degrees of freedom we
additionally need to employ Prop.~\ref{PROP:COR:DININF}
\ref{item:cordinN:binf}: for $\bsnu\in\N_0^k$ such that
$\bsnu\neq\bsnul\dfn (0)_{j=1}^k$ but $\nu_k=0$,
Prop.~\ref{PROP:COR:DININF} \ref{item:cordinN:binf} together with
\eqref{eq:legbound} implies a bound of the type
\begin{equation}\label{eq:legbound2}
  |l_{k,\bsnu}|=\left|\int_{\U{k}} (\partial_k T_k(\bsy_{[k]})-1)L_\bsnu(\bsy_{[k]})\dd\mu(\bsy_{[k]}) \right|\le C \frac{1}{\delta_k}
  \prod_{j=1}^k (1+\delta_j)^{-\nu_j},
\end{equation}
where the additional $\frac{1}{\delta_k}$ stems from
$\norm[{L^\infty(\cB_{\bszeta_{[j-1]}}(\U{1})\times
  \U{1})}]{\partial_kT_k-1}\le \frac{C_3}{\delta_k}$. Here we used the
fact $\int_{\U{k}} L_\bsnu(\bsy_{[k]})\dd\mu(\bsy_{[k]})=0$ for all
$\bsnu\neq\bsnul$ by orthogonality of the $(L_\bsnu)_{\bsnu\in\N_0^k}$
and because $L_\bsnul\equiv 1$. In case $\nu_k\neq 0$, then the factor
$\frac{1}{1+\delta_k}$ occurs on the right-hand side of
\eqref{eq:legbound} .  Hence, %
\emph{all} coefficients $l_{k,\bsnu}$ for which $\bsnu\neq\bsnul$ are
of size $O(\frac{1}{\delta_k})$. In fact one can show that even
$\sum_{\bsnu\neq\bsnul}|l_{k,\bsnu}||L_\bsnu(\bsy_{[k]})|$ is of size
$O(\frac{1}{\delta_k})$. Thus
\begin{equation*}
  \partial_kT_k(\bsy_{[k]})=\sum_{\bsnu\in\N_0^k}l_{k,\bsnu}
  L_\bsnu(\bsy_{[k]})=l_{k,\bsnul} L_\bsnul(\bsy_{[k]})+O\left(\frac{1}{\delta_k}\right).
\end{equation*}
Using $L_\bsnul\equiv 1$
\begin{equation*}
  l_{k,\bsnul} = \int_{\U{k}} \partial_kT_k(\bsy_{[k]}) L_\bsnul(\bsy_{[k]})\dd\mu(\bsy_{[k]})
  =\int_{\U{{k-1}}} T_k(\bsy_{[k-1]},1)-T_k(\bsy_{[k-1]},-1)\dd\mu(\bsy_{[k-1]})
  =2,
\end{equation*}
and therefore if $\delta_k$ is very large, since $L_\bsnul\equiv 1$
\begin{equation}\label{eq:l0trivial}
  T_k(\bsy_{[k]})=-1+\int_{-1}^{y_k}\partial_kT_k(\bsy_{[k-1]},t)\dd\mu(t)\simeq
  -1+\int_{-1}^{y_k}l_{k,\bsnul}L_\bsnul(\bsy_{[k]})\dd\mu(t)=y_k.
\end{equation}
Hence, for large $\delta_k$ we can use the trivial approximation
$T_k(\bsy_{[k]})\simeq y_k$. To address this special role played by
the $k$th variable for the $k$th component we introduce
\begin{equation}\label{eq:gamma}
  \gamma(\bsvarrho,\bsnu)\dfn \varrho_k^{-\max\{1,\nu_k\}}\prod_{j=1}^{k-1}\varrho_j^{-\nu_j} \qquad\qquad\forall \bsvarrho\in (1,\infty)^\N,~\bsnu\in\N_0^k,
\end{equation}
which, up to constants, corresponds to the minimum of
\eqref{eq:legbound} and \eqref{eq:legbound2}.  This quantity can be
interpreted as measuring the importance of the monomial $\bsy^\bsnu$
in the ansatz space used for the approximation of $T_k$, and we will
use it to construct such ansatz spaces.
\begin{remark}
  To explain the key ideas, in this section we presented the
  approximation of $T_k$ via a Legendre expansion of $\partial_k T_k$.
  For the proofs of our approximation results in
  Sec.~\ref{sec:polinfty} we instead approximate
  $\sqrt{\partial_k T_k}-1$ %
  with truncated Legendre expansions. This will guarantee the
  approximate transport to satisfy the monotonicity property as
  explained in Sec.~\ref{sec:polinfty}.
  \end{remark}

\section{Convergence of the transport}\label{sec:polinfty}

{
We are now in position to state an algebraic convergence result for
approximations of infinite dimensional transport maps $T:U\to U$
associated to densities of the type \eqref{eq:density}.

For a triangular approximation $\tilde T=(\tilde T_k)_{k\in\N}$ to $T$
it is desirable that it retains the monotonicity and
bijectivity properties, i.e., $\partial_k\tilde T_k>0$ and
$\tilde T:U\to U$ is bijective. The first guarantees that $\tilde T$
is injective and easy to invert (by subsequently solving the one
dimensional equations $x_k=\tilde T_k(y_1,\dots,y_k)$ for $y_k$
starting with $k=1$), and for the purpose of generating samples, the
second property ensures that for $\bsy\sim \measi$, the transformed
sample $\tilde T(\bsy)\sim \tilde T_\sharp\measi$ also belongs to $U$.
These constraints are hard to enforce for polynomial approximations.
For this reason, we use the same rational parametrization we
introduced in \cite{zm1} for the finite dimensional case:
For a
set of $k$-dimensional multiindices $\Lambda\subseteq \N_0^k$,
define
\begin{equation*}
  \bbP_\Lambda\dfn {\rm span}\set{\bsy^\bsnu}{\bsnu\in\Lambda}.
\end{equation*}
The dimension of this space is equal to the cardinality of $\Lambda$,
which we denote by $|\Lambda|$. Let $p_k\in \bbP_\Lambda$ (where
$\Lambda$ remains to be chosen) be a polynomial approximation to
$\sqrt{\partial_kT_k}-1$. Set for $\bsy\in \U{k}$
\begin{equation}\label{eq:tTk}
  \tilde T_k(\bsy) \dfn -1 + 2 \frac{\int_{-1}^{y_k}\int_{\U{{k-1}}}(p_k(\bsy_{[k-1]},t)+1)^2\dd\mu(\bsy_{[k-1]})\dd\mu(t)}{\int_{\U{k}}(p_k(\bsy)+1)^2\dd\mu(\bsy)}.
\end{equation}
It is easily checked that $\tilde T_k$ satisfies both monotonicity and
bijectivity as long as $p_k\neq -1$. Thus we end up with a rational
function $\tilde T_k$, but we emphasize that the use of rational
functions instead of polynomials is not due to better approximation
capabilities, but solely to guarantee bijectivity of
$\tilde T:U\to U$.
\begin{remark}
Observe that $\Lambda=\emptyset$ gives the trivial
approximation $p_k\dfn 0\in\bbP_\emptyset$ and $\tilde T_k(\bsy)=y_k$.
\end{remark}

The following theorem yields an algebraic convergence rate
\emph{independent of the dimension} (since the dimension is infinity)
in terms of the total number of degrees of freedom for the
approximation of $T$.
Therefore the curse of dimensionality is overcome for
densities as in Assumption \ref{ass:density}.}

\begin{theorem}\label{THM:TINF}
  Let $f_\measi$, $f_\measii:U\to (0,\infty)$ be two probability
  densities satisfying Assumption \ref{ass:densities} for some
  $p\in (0,1)$.
  Set
  $b_j\dfn
  \max\{\norm[Z]{\psi_{\measi,j}},\norm[Z]{\psi_{\measii,j}}\}$,
  $j\in\N$.

  There exist $\alpha>0$ and $C>0$ such
  that the following holds: For $j\in\N$ set
  \begin{equation}\label{eq:xij}
    \varrho_j\dfn 1+ \frac{\alpha}{b_j},
  \end{equation}
  and with $\gamma(\bsvarrho,\bsnu)$ as in \eqref{eq:gamma} define
  \begin{equation*}
    \Lambda_{\eps,k}\dfn \set{\bsnu\in\N_0^k}{\gamma(\bsvarrho,\bsnu)\ge\eps}\qquad\forall k\in\N.
  \end{equation*}
  
  For each $k\in\N$ there exists a polynomial
  $p_k\in\bbP_{\Lambda_{\eps,k}}$ such that with the components
  $\tilde T_{\eps,k}$ as in \eqref{eq:tTk},
  $\tilde T_\eps=(\tilde T_{\eps,k})_{k\in\N}:U\to U$ is a monotone triangular
  bijection.
  For all $\eps>0$, it holds that $N_\eps\dfn \sum_{k\in\N} |\Lambda_{\eps,k}|<\infty$
  and
  \begin{subequations}\label{eq:Talgebraic}
  \begin{equation}\label{eq:Talgebraica}
    \sum_{k\in\N}\norm[{L^{\infty}(\U{k})}]{T_k-\tilde T_{\eps,k}}\le C N_\eps^{-\frac{1}{p}+1}
  \end{equation}
  and
  \begin{equation}\label{eq:Talgebraicb}
    \sum_{k\in\N}\norm[{L^{\infty}(\U{k})}]{\partial_kT_k-
      \partial_k\tilde T_{\eps,k}}\le C N_\eps^{-\frac{1}{p}+1}.
  \end{equation}
  \end{subequations}
\end{theorem}

\begin{remark}\label{rmk:kidentity}
  Fix $\eps>0$.  Since $N_\eps<\infty$, there exists $k_0\in\N$ such
  that for all $k\ge k_0$ holds $\Lambda_{\eps,k}=\emptyset$ and thus
  $\tilde T_{\eps,k}(\bsy_{[k]})=y_k$.
\end{remark}

Switching the roles of $\measi$ and $\measii$, Thm.~\ref{THM:TINF}
also yields an approximation result for the inverse transport
$S=T^{-1}$ by some rational functions $\tilde S_k$ as in
\eqref{eq:tTk}. Moreover, if $\tilde T$ is the rational approximation
from Thm.~\ref{THM:TINF}, then its inverse $\tilde T^{-1}:U\to U$
(whose components are not necessarily rational functions) also
satisfies an error bound of the type \eqref{eq:Talgebraic} as we show
next.

{
\begin{corollary}\label{COR:SINF}
    Consider the setting of Thm.~\ref{THM:TINF}. Denote
    $S\dfn T^{-1}:U\to U$ and $\tilde S_\eps\dfn \tilde T_\eps^{-1}:U\to U$.
    Then there exists a constant $C$ such that for all $\eps>0$
  \begin{subequations}\label{eq:Salgebraic}
  \begin{equation}\label{eq:Ssumbound}
    \sum_{k\in\N}\norm[{L^{\infty}(\U{k})}]{S_k-\tilde S_{\eps,k}}\le C N_\eps^{-\frac{1}{p}+1}
  \end{equation}
  and
  \begin{equation}\label{eq:dSsumbound}
    \sum_{k\in\N}\norm[{L^{\infty}(\U{k})}]{\partial_kS_k-
      \partial_k\tilde S_{\eps,k}}\le C N_\eps^{-\frac{1}{p}+1}.
  \end{equation}
  \end{subequations}
\end{corollary}

Note that both $S$ and $\tilde S$ in Cor.~\ref{COR:SINF} are
monotonic, triangular bijections as they are the inverses of such
maps.}

\section{Convergence of the pushforward measures}\label{sec:measinfty}
Thm.~\ref{THM:TINF} established smallness of
$\sum_{k\in\N}|\partial_k(T_k-\tilde T_k)|$. The relevance of this
term stems from the formal calculation (cp.~\eqref{eq:detinf})
\begin{equation*}
  |\det dT-\det d\tilde T|=\left|\prod_{k\in\N}\partial_kT_k-
    \prod_{k\in\N}\partial_k\tilde T_k\right|
  \le \sum_{k\in\N}|\partial_kT_k-\partial_k\tilde T_k| \prod_{j<k}|\partial_j T_j|\prod_{i>k}|\partial_iT_i|.
\end{equation*}
Assuming that we can bound the last two products, the determinant
$\det d\tilde T$ converges to $\det d T$ at the rate given in
Thm.~\ref{THM:TINF}. This will allow us to bound the Hellinger
distance (H), the total variation distance (TV), and the
Kullback-Leibler divergence (KL) between $\tilde T_\sharp\measi$ and
$\measii$, as we show in the following theorem. Recall that for two
probability measures $\nu\ll\mu$, $\eta\ll\mu$ on $U$ with densities
$f_\nu=\frac{\ddd\nu}{\ddd\mu}$, $f_\eta=\frac{\ddd\eta}{\ddd\mu}$,
\begin{equation*}
  {\rm H}(\nu,\eta)=\frac{1}{\sqrt{2}}\norm[L^2(U,\mu)]{\sqrt{f_\nu}-\sqrt{f_\eta}},\quad
  {\rm TV}(\nu,\eta)=\frac{1}{2}\norm[L^1(U,\mu)]{f_\nu-f_\eta},\qquad
  {\rm KL}(\nu,\eta)=\int_U \log\left(\frac{f_\nu}{f_\eta}\right)\dd\nu.
\end{equation*}

\begin{theorem}\label{THM:MEASCONVINF}
  Let $f_\measi$, $f_\measii$ satisfy Assumption \ref{ass:densities}
  for some $p\in (0,1)$, and let $\tilde T_\eps:U\to U$ be the
  approximate transport from Thm.~\ref{THM:TINF}.

  Then there exists $C>0$ such that for
  ${\rm dist}\in\{{\rm H},{\rm TV},{\rm KL}\}$ and every $\eps>0$
  \begin{equation}\label{eq:measdiffinf}
    {\rm dist}((\tilde T_\eps)_\sharp \mu,\measii)
    \le C N_\eps^{-\frac{1}{p}+1}.
  \end{equation}
\end{theorem}

{Next we treat the Wasserstein distance. Recall that for a Polish
  space $(M,d)$ (i.e., $M$ is separable and complete with the metric
  $d$ on $M$) and for $q\in [1,\infty)$, the $q$-Wasserstein distance
  between two probability measures $\nu$ and $\eta$ on $M$ (equipped
  with the Borel $\sigma$-algebra) is defined as
  \cite[Def.~6.1]{MR2459454}
  \begin{equation*}
    W_q(\nu,\eta)\dfn \inf_{\gamma\in\Gamma} \left(\int_M d(x,y)^q\dd\gamma(x,y)\right)^{1/q},
  \end{equation*}
  where $\Gamma$ stands for the couplings between $\eta$ and $\nu$,
  i.e., the set of probability measures on $M\times M$ with marginals
  $\nu$ and $\eta$, cp.~\cite[Def.~1.1]{MR2459454}.

  To bound the Wasserstein distance, we employ the following
  proposition. It has been similarly stated in \cite[Theorem
  2]{MR4120535}, but for measures on
  $\R^d$. %
  To fit our setting, we extend the result to compact metric
  spaces,\footnote{The author of \cite{MR2459454} mentions that such a
    result is already known, but without providing a reference. For
    completeness we have added the proof.}
  but emphasize that the proof closely follows that of \cite[Theorem
  2]{MR4120535}, and the argument is essentially the same. As pointed
  out in \cite{MR4120535}, the bound in the proposition is sharp.
    
  \begin{proposition}\label{PROP:WASSERSTEIN}
    Let $(M,d)$ be a compact Polish space. Let $T:M\to M$ and
    $\tilde T:M\to M$ be two continuous functions and let $\nu$ be a
    probability measure on $M$ equipped with the Borel
    $\sigma$-algebra.
    Then for every $q\in [1,\infty)$
    \begin{equation}
      W_q(T_\sharp\nu,\tilde T_\sharp\nu)\le \sup_{x \in M}d(T(x),\tilde T(x))<\infty.
    \end{equation}
  \end{proposition}

  To apply Prop.~\ref{PROP:WASSERSTEIN} we first have to equip $U$
  with a metric. For a sequence $(c_j)_{j\in\N}\in\ell^1(\N)$ of
  positive numbers set
  \begin{equation}\label{eq:prodmet}
    d(\bsx,\bsy)\dfn \sum_{j\in\N}c_j|x_j-y_j|\qquad\forall\;\bsx,\bsy\in U.
  \end{equation}
  By Lemma \ref{lemma:producttop}, $d$ defines a metric that induces
  the product topology on $U$. Since $U$ with the product topology is
  a compact space by Tychonoff's theorem \cite[Thm.~37.3]{munkres},
  $(U,d)$ is a compact Polish space. Moreover:

  \begin{lemma}\label{LEMMA:TCONT}
    Let $f_\measi$, $f_\measii$ satisfy Assumption \ref{ass:densities}
    and consider the metric \eqref{eq:prodmet} on $U$. Then $T:U\to U$
    and the approximation $\tilde T_\eps:U\to U$ from
    Thm.~\ref{THM:TINF} are continuous with respect to $d$. Moreover,
    if there exists $C>0$ such that with
    \begin{equation}\label{eq:bj}
      b_j\dfn
      \max\{\norm[X]{\psi_{\measi,j}},\norm[Y]{\psi_{\measii,j}}\}
    \end{equation}
    holds $b_j\le Cc_j$ for all $j\in\N$ (cp.~Assumption
    \ref{ass:densities}), then $T$ and $\tilde T_\eps$ are Lipschitz
    continuous.
  \end{lemma}

  With $d:U\times U\to\R$ as in Lemma \ref{LEMMA:TCONT}, $(U,d)$ is a
  compact Polish space and $T$ and $\tilde T_\eps$ are continuous, so
  that we can apply Prop.~\ref{PROP:WASSERSTEIN}. Using
  Thm.~\ref{THM:TINF} and $\sup_jc_j\in (0,\infty)$,
  \begin{equation}\label{eq:Wq}
    W_q(T_\sharp\mu,(\tilde T_\eps)_\sharp\mu)\le \sup_{\bsy\in U} d(T(\bsy),\tilde T_\eps(\bsy))\le \sum_{k\in\N}\norm[{L^\infty(\U{k})}]{T_k-\tilde T_{\eps,k}}c_k
    \le CN_\eps^{-\frac{1}{p}+1}.
  \end{equation}
  
  Next let us discuss why $c_j\dfn b_j$ as in \eqref{eq:bj} is a
  natural choice in our setting. Let $\Phi:U\to X$ be the map
  $\Phi(\bsy)=\sum_{j\in\N}y_j\psi_{\measii,j}\in Y$. In the inverse
  problem discussed in Ex.~\ref{ex:bayes}, we try to recover an
  element $\Phi(\bsy)\in Y$. For computational purposes, the problem
  is set up to recover instead the expansion coefficients $\bsy\in
  U$. Now suppose that $\measii$ is the posterior measure on $U$. Then
  $\Phi_\sharp\measii=(\Phi\circ T)_\sharp \measi$ is the
  corresponding posterior measure on $Y$ (the space we are actually
  interested in). The map $\Phi:U\to Y$ is Lipschitz continuous
  w.r.t.\ the metric $d$ on $U$, since for $\bsx$, $\bsy\in U$ due to
  $\norm[Y]{\psi_{\measii,j}}\le b_j$,
  \begin{equation}\label{eq:before}
    \norm[Y]{\Phi(\bsx)-\Phi(\bsy)}
    =\normc[Y]{\sum_{j\in\N}(x_j-y_j)\psi_{\measii,j}}
    \le \sum_{j\in\N}|x_j-y_j|b_j=d(\bsx,\bsy).
  \end{equation}
  Therefore, $\Phi\circ T:U\to Y$ and $\Phi\circ\tilde T_\eps:U\to Y$
  are Lipschitz continuous by Lemma \ref{LEMMA:TCONT}. Moreover,
  compactness of $U$ and continuity of $\Phi:U\to Y$ imply that
  $\Phi(U)\subseteq Y$ is compact.  Hence we may apply
  Prop.~\ref{PROP:WASSERSTEIN} also w.r.t.\ the maps
  $\Phi\circ T:U\to Y$ and $\Phi\circ\tilde T_\eps:U\to Y$.  This
  gives a bound of the pushforward measures on the Banach space $Y$.
  Specifically, since
  $\norm[Y]{\Phi(T(\bsy))-\Phi(\tilde T_\eps(\bsy))}\le
  d(T(\bsy),\tilde T_\eps(\bsy))$, which can be bounded as in
  \eqref{eq:Wq}, we have shown:

  \begin{theorem}\label{thm:wassersteinconv}
    Let $f_\measi$, $f_\measii$ satisfy Assumption \ref{ass:densities}
    for some $p\in (0,1)$, let $\tilde T_\eps:U\to U$ be the
    approximate transport and let $N_\eps\in\N$ be the number of
    degrees of freedom as in Thm.~\ref{THM:TINF}.

    Then there exists $C>0$ such that for every $q\in [1,\infty)$ and
    every $\eps>0$
    \begin{equation*}
      W_q((\tilde T_\eps)_\sharp \mu,\measii)
      \le C N_\eps^{-\frac{1}{p}+1},
    \end{equation*}
    and for the pushforward measures on the Banach space $Y$
    \begin{equation}\label{eq:convpwY}
      W_q((\Phi\circ \tilde T_\eps)_\sharp \mu,\Phi_\sharp \measii)
      \le C N_\eps^{-\frac{1}{p}+1}.
    \end{equation}
  \end{theorem}

  Finally let us discuss how to efficiently sample from the measure
  $\Phi_\sharp \measii$ on the Banach space $Y$. As explained in the
  introduction, for a sample $\bsy\sim\measi$ we have
  $T(\bsy)\sim \measii$ and
  $\Phi(T(\bsy))=\sum_{j\in\N}T_j(\bsy_{[j]})\psi_{\measii,j}\sim
  \Phi_\sharp\measii$. To truncate this series, introduce
  $\Phi_s(\bsy_{[s]})\dfn \sum_{j=1}^s y_j\psi_{\measii,j}$.  As
  earlier, denote by $\measi_s$ the marginal measure of $\measi$ on
  $\U{s}$. For $\bsy_{[s]}\sim\measi_s$, the sample
  \begin{equation*}
    \Phi_{s}(\tilde T_{\eps,[s]}(\bsy_{[s]}))
    =\sum_{j=1}^{s}T_{\eps,j}(\bsy_{[j]})\psi_{\measii,j}
  \end{equation*}
  follows the distribution of
  $(\Phi_s\circ \tilde T_{\eps,[s]})_\sharp\measi_s$, where
  $\tilde T_{\eps,[s]}\dfn (\tilde T_{\eps,k})_{k=1}^s:\U{s}\to\U{s}$.
  In the next corollary we bound the Wasserstein distance between
  $(\Phi_s\circ \tilde T_{\eps,[s]})_\sharp\measi_s$ and
  $\Phi_\sharp\measii$.
  Note that the former is a measure on $Y$, and in contrast to the
  latter, is supported on an $s$-dimensional subspace. Thus in general
  neither of these two measures need to be absolutely continuous
  w.r.t.\ the other. This implies that the KL divergence, the total
  variation distance, and the Hellinger distance, in contrast with the
  Wasserstein distance, need not tend to $0$ as $\eps\to 0$ and
  $s\to\infty$.

  The corollary shows that the convergence rate in \eqref{eq:convpwY}
  can be retained by choosing the truncation parameter $s$ as $N_\eps$
  (the number of degrees of freedom in Thm.~\ref{THM:TINF}); in fact,
  it even suffices to truncate after the maximal $k$ such that
  $\Lambda_{k,\eps}\neq\emptyset$, as described in
  Rmk.~\ref{rmk:dNeps}.

\begin{corollary}\label{COR:MEASCONVINF}
  Consider the setting of Thm.~\ref{thm:wassersteinconv} and assume
  that $(b_j)_{j\in\N}$ in \eqref{eq:bj} is monotonically decreasing.
  Then there exists $C>0$ such that for every $q\in [1,\infty)$ and $\eps>0$
  \begin{equation*}
    W_q((\Phi_{N_\eps}\circ \tilde T_{\eps,[N_\eps]})_\sharp \measi_{N_\eps},\Phi_\sharp \measii)
    \le C N_\eps^{-\frac{1}{p}+1}.
  \end{equation*}
\end{corollary}
\begin{remark}
  Convergence in $W_q$ implies weak convergence \cite[Theorem
  6.9]{MR2459454}.
\end{remark}

\begin{remark}\label{rmk:dNeps}
  Checking the proof of Thm.~\ref{THM:TINF}, we have
  $N_\eps\le C\eps^{-p}$, cp.~\eqref{eq:nepsest}. Thus the maximal
  activated dimension (represented by the truncation parameter
  $s=N_\eps$) increases only algebraically as $\eps\to 0$.  The
  approximation error also decreases algebraically like $\eps^{1-p}$
  as $\eps\to 0$, cp.~\eqref{eq:sumest}. Moreover,
  the function $\Phi_{s_\eps}\circ \tilde T_{\eps,[s_\eps]}$ with
  $s_\eps\dfn \max\set{k\in\N}{\Lambda_{\eps,k}\neq\emptyset}$
  leads to the same convergence rate in Cor.~\ref{COR:MEASCONVINF}. In
  other words, we only need to use the components $\tilde T_{\eps,k}$
  for which $\Lambda_{\eps,k}\neq \emptyset$.
\end{remark}

}

  \section{Conclusions}\label{sec:conclusions}
  {The use of transportation methods to sample from
    high-dimensional distributions is becoming increasingly popular to
    solve inference problems and perform other machine learning tasks. Therefore,
    questions of when and how these methods can be successful are of
    great importance, but thus far not well understood. In the present
    paper we analyze the approximation of the KR transport in the
    high- (or infinite-) dimensional regime and on the bounded domain
    $[-1,1]^\N$.  Under the setting presented in Sec.~\ref{sec:main},
    it is shown that the transport can be approximated without
    suffering from the curse of dimension. Our approximation is based
    on polynomial and rational functions, and we provide an explicit
    \textit{a priori} construction of the ansatz space. Moreover, we show
    how these results imply that it is possible to efficiently sample
    from certain high dimensional distributions by transforming a
    lower dimensional latent variable.  

    As we have discussed in the finite dimensional case \cite[Sec.~5]{zm1}, from an approximation
    viewpoint there is also a link to neural networks, which can be
    established via \cite{yarotsky,pmlr-v70-telgarsky17a} where it is
    proven that ReLU neural networks are efficient at emulating
    polynomials and rational functions. While we have not developed
    this aspect further in the present manuscript, we mention that
    neural networks are used in the form of normalizing flows
    \cite{pmlr-v37-rezende15,papamakarios2019normalizing} to couple distributions in spaces of
    equal dimension, and for example in the form of generative
    adversarial networks \cite{NIPS2014_5ca3e9b1,pmlr-v70-arjovsky17a}
    and, more recently, injective flows \cite{2002.08927,2102.10461},
    to map lower-dimensional latent variables to samples from a %
    high-dimensional distribution. In Sec.~\ref{sec:measinfty} we
    provided some insight 
    (for the present setting, motivated by inverse problems in science and engineering)
    into how low-dimensional the latent variable can be, and how expressive the transport should be,
    to achieve a certain accuracy in the Wasserstein distance (see Cor.~\ref{COR:MEASCONVINF}).  
    Further examining this connection and generalizing our results to distributions on
    unbounded domains (such as $\R^\N$ instead of $[-1,1]^\N$) will be
    the topic of future research.}

\appendix

\section{Proofs of Sec.~\ref{SEC:Tinf}}
\subsection{%
  Lemma \ref{LEMMA:FC0}}\label{app:lemma:fc0}

\begin{lemma}\label{lemma:producttop}
  Let $(c_j)_{j\in\N}\in\ell^1(\N)$ be a sequence of positive numbers.
  Then $d(\bsx,\bsy)\dfn \sum_{j\in\N}c_j|x_j-y_j|$ defines a metric
  on $U$ that induces the product topology.
\end{lemma}
\begin{proof}
  Recall that the family of sets
  \begin{equation*}
    \set{\bsx\in U}{|x_j-y_j|<\eps~\forall j\le N}\qquad\bsy\in U,~\eps>0,~N\in\N,
  \end{equation*}
  forms a basis of the product topology on $U$. Fix $\bsy\in U$ and
  $\eps>0$, and let $N_\eps\in\N$ be so large that
  $\sum_{j>N_\eps}2c_j<\frac{\eps}{2}$. Let
  $C_0\dfn \sum_{j=1}^{N_\eps}c_j$. Then if $\bsx$, $\bsy\in U$
  satisfy $|x_j-y_j|<\frac{\eps}{2C_0}$ for all $j\le N_\eps$, we have
  $d(\bsx,\bsy)=\sum_{j\in\N}c_j|x_j-y_j|<\frac{\eps}{2}\frac{\sum_{j=1}^{N_\eps}c_j}{C_0}+\sum_{j>N_\eps}2c_j\le\eps$,
  and thus
  \begin{equation*}
    \setc{\bsx\in U}{|x_j-y_j|<\frac{\eps}{2 C_0}~\forall j\le N_\eps}\subseteq\setc{\bsx\in U}{\sum_{j\in\N}c_j|x_j-y_j|<\eps}
    =\set{\bsx\in U}{d(\bsx,\bsy)<\eps}.
  \end{equation*}
  On the other hand, if we fix $\bsy\in U$, $\eps>0$ and $N\in\N$,
  and set $C_0\dfn \min_{j=1,\dots,N} c_j>0$, then
  \begin{equation*}
    \setc{\bsx\in U}{d(\bsx,\bsy)<\eps C_0}
    =\setc{\bsx\in U}{\sum_{j\in\N}c_j|x_j-y_j|<\eps C_0}
    \subseteq
    \set{\bsx\in U}{|x_j-y_j|<\eps~\forall j\le N}.\qedhere
  \end{equation*}
\end{proof}

\begin{proof}[Proof of Lemma \ref{LEMMA:FC0}]
  {By \cite[Lemma 6.4.2 (ii)]{bogachev}, the Borel $\sigma$-algebra
    on $U$ (with the product topology) coincides with the product
    $\sigma$-algebra on $U$. Since $f:U\to\R$ is continuous, and
    because $U$ and $\R$ are equipped with the Borel
    $\sigma$-algebras, $f:U\to\R$ is measurable. Since $f$ is bounded
    it belongs to $L^2(U,\mu)$.
    
    Fix $(c_j)_{j\in\N}\in\ell^1(\N)$ with $c_j>0$ for all $j\in\N$,
    and let $d$ be the metric on $U$ from Lemma
    \ref{lemma:producttop}.}  Since $f\in C^0(U;\R_+)$ and $U$ is
  compact by Tychonoff's theorem \cite[Thm.~37.3]{munkres}, the
  Heine-Cantor theorem yields $f$ to be uniformly continuous. Thus for
  any $\eps>0$ there exists $\delta_\eps>0$ such that for all $\bsx$,
  $\bsy\in U$ with $d(\bsx,\bsy)<\delta_\eps$ it holds
  $|f(\bsx)-f(\bsy)|<\eps$. Now let $k\in\N$ and $\eps>0$ arbitrary.
  Then for all $\bsx_{[k]}$, $\bsy_{[k]}\in \U{k}$ such that
  $\sum_{j=1}^k c_j|x_j-y_j|<\delta_\eps$, we get
  \begin{equation*}
    |\hat f_k({\bsx_{[k]}})-\hat f_k({\bsy_{[k]}})|
    =\left|\int_U f(\bsx_{[k]},\bst)-f(\bsy_{[k]},\bst)\dd\mu(\bst)\right|
    \le \int_U |f(\bsx_{[k]},\bst)-f(\bsy_{[k]},\bst)|\dd\mu(\bst)
    \le \eps,
  \end{equation*}
  which shows continuity of $\hat f_k:\U{k}\to\R$.

  Next, using that $\inf_{\bsy\in U}f(\bsy)\dfnn r>0$ (due to
  compactness of $U$ and continuity of $f$), for $k>1$ we have
  $\hat f_{k-1}({\bsx_{[k-1]}})\ge \min\{r,1\}>0$ independent of
  $\bsx_{[k-1]}\in \U{{k-1}}$. This implies that also
  $\frac{\hat f_k}{\hat f_{k-1}}=f_k:\U{k}\to\R_+$ is continuous,
  {where the case $k=1$ is trivial since $\hat f_0\equiv 1$.}

  Finally we show \eqref{eq:unifconv}. Let again $\eps>0$ be arbitrary
  and $N_\eps\in\N$ so large that
  $\sum_{j>N_\eps}{2c_j}<\delta_\eps$.  Then for every $\bsx$,
  $\bst\in U$ and every $k>N_\eps$ we have
  $d((\bsx_{[k]},\bst),\bsx)\le
  \sum_{j>N_\eps}c_j|x_j-t_j|\le
  \sum_{j>N_\eps}2c_j<\delta_\eps$, which implies
  $|f(\bsx_{[k]},\bst)-f(\bsx)|<\eps$. Thus for every $\bsx\in U$ and
  every $k>N_\eps$
  \begin{equation*}
    |\hat f_k(\bsx_{[k]})-f(\bsx)|
    =\left|\int_U f(\bsx_{[k]},\bst)\dd\mu( \bst)-f(\bsx)\right|
    \le \int_U |f(\bsx_{[k]},\bst)-f(\bsx)|\dd\mu( \bst)
    < \eps,
  \end{equation*}
  which concludes the proof.
\end{proof}

\subsection{Thm.~\ref{THM:KNOTHEINF}}\label{app:thm:knotheinf}
With
$F_{*;k}(\bsx_{[k-1]},x_k)\dfn \int_{-1}^{x_k}
f_{*;k}(\bsx_{[k-1]},t_k)\dd t_k$, the construction of $T_k:\U{k}\to\U{k}$
described in Sec.~\ref{SEC:Tinf} amounts to the explicit formula
$T_1(x_1)\dfn (F_{\measii;1})^{-1}\circ F_{\measi;1}(x_1)$ and
inductively
\begin{equation}\label{eq:Texpl}
      T_k(\bsx_{[k-1]},\cdot)\dfn F_{\measii;k}(T_{[k-1]}(\bsx_{[k-1]}),\cdot)^{-1}\circ F_{\measi;k}(\bsx_{[k-1]},\cdot),
    \end{equation}
    where $F_{\measii;k}(T_{[k-1]}(\bsx_{[k-1]}),\cdot)$ denotes the
    inverse of
    $x_k\mapsto F_{\measii;k}(T_{[k-1]}(\bsx_{[k-1]}),x_k)$.
    \begin{remark}\label{rmk:Tkcont}
      If $f_{*;k}\in C^0(\U{k};\R_+)$ for $*\in\{\measi,\measii\}$,
      then by \eqref{eq:Texpl} it holds $T_k$,
      $\partial_kT_k\in C^0(\U{k})$.
    \end{remark}

\begin{proof}[Proof of Thm.~\ref{THM:KNOTHEINF}]
  We start with \ref{item:KT}.  As a consequence of
  Rmk.~\ref{rmk:Tkcont} and Lemma \ref{LEMMA:FC0},
  $T_k\in C^0(\U{k};\U{1})$ for every $k\in\N$. So each
  $T_k:\U{k}\to \U{1}$ is measurable and thus also
  ${T_{[n]}}=(T_k)_{k=1}^n:\U{n}\to \U{n}$ is measurable for
  each $n\in\N$. Furthermore $T:U\to U$ is bijective by Lemma
  \ref{lemma:bijective} and because for every $\bsx\in U$ and $k\in\N$
  it holds that $T_k(\bsx_{[k-1]},\cdot):\U{1}\to \U{1}$ is
  bijective.

  The product $\sigma$-algebra on $U$ is generated by the algebra (see
  \cite[Def.~1.2.1]{bogachev}) $\cA_0$ given as the union of the
  $\sigma$-algebras
  $\cA_n\dfn \set{A_n\times [-1,1]^\N}{A_n\in\cB(\U{n})}$,
  $n\in\N$, where $\cB(\U{n})$ denotes the Borel
  $\sigma$-algebra. For sets of the type
  $A\dfn A_n \times [-1,1]^\N\in\cA_n$ with $A_n\in\cB(\U{n})$,
  due to $T_j(\bsy_{[j]})\in \U{1}$ for all $\bsy\in U$ and $j>n$,
  we have
  \begin{equation*}
    T^{-1}(A)=\set{\bsy\in U}{T(\bsy)\in A}=\set{\bsy\in U}{%
      T_{[n]}(\bsy_{[n]})\in A_n
    }=(T_{[n]})^{-1}(A_n)\times [-1,1]^\N,
  \end{equation*}
  which belongs to $\cA_n$ and thus to the product $\sigma$-algebra on
  $U$ since $T_{[n]}$ is measurable. Hence $T:U\to U$ is measurable
  w.r.t.\ the product $\sigma$-algebra.

  Denote now by $\measii_n$ and $\measi_n$ the marginals on $\U{n}$
  w.r.t.\ the first $n$ variables, i.e., e.g.,
  $\measii_n(A)\dfn\measii(A\times [-1,1]^\N)$ for every
  $A\in\cB(\U{n})$. %
  By \eqref{eq:Tfinite} (see \cite[Prop.~2.18]{santambrogio}),
  $(T_{[n]})_\sharp\measi_n=\measii_n$. For sets of the type
  $A\dfn A_n \times [-1,1]^\N\in\cA_n$ with $A_n\in\cB(\U{n})$,
  \begin{align*}
    T_\sharp\measi(A)=\measi(\set{\bsy\in U}{T(\bsy)\in A})
    &=\measi(\set{\bsy\in U}{{T_{[n]}(\bsy_{[n]})}\in A_n})\nonumber\\
    &=\measi_n(\set{\bsy\in \U{n}}{T_{[n]}(\bsy)\in A_n})\nonumber\\
    &=\measii_n(A_n)\nonumber\\
    &=\measii(A).
  \end{align*}
  According to \cite[Thm.~3.5.1]{bogachev},
  the extension of a non-negative $\sigma$-additive set function on the
  algebra $\cA_0$ to the $\sigma$-algebra generated by $\cA_0$ is
  unique. Since $T:U\to U$ is bijective and measurable, it holds that
  both $\measii$ and $T_\sharp \measi$ are measures on $U$ and
  therefore $\measii=T_\sharp\measi$.

  Finally we show \ref{item:rdder}. Let
  $\hat f_{\measii,n}\in C^0(\U{n};\R_+)$ and
  $\hat f_{\measi,n}\in C^0(\U{n};\R_+)$ be as in \eqref{eq:fk2},
  i.e., these functions denote the densities of $\measii_n$, $\measi_n$.
  Since $(T_{[n]})_\sharp \measi_n=\measii_n$, %
  by a change of variables (see, e.g., \cite[Prop.~2.5]{bogachevtri}),
   for all $\bsx\in U$
  \begin{equation*}
    \hat
    f_{\measi,n}(\bsx_{[n]})=\hat f_{\measii,n}(T_{[n]}(\bsx_{[n]}))\det dT_{[n]}(\bsx_{[n]})=
    \hat f_{\measii,n}(T_{[n]}(\bsx_{[n]}))\prod_{j=1}^n\partial_jT_j(\bsx_{[j]}).
  \end{equation*}
  Therefore
  \begin{equation}\label{eq:detfrac}
    \prod_{j=1}^n\partial_jT_j(\bsx_{[j]})=\frac{\hat f_{\measi,n}(\bsx_{[n]})}{\hat f_{\measii,n}(T_{[n]}(\bsx_{[n]}))}.
  \end{equation}
  According to Lemma \ref{LEMMA:FC0} we have uniform convergence
  \begin{equation*}
    \lim_{n\to\infty}\hat f_{\measi,n}(\bsx_{[n]})=f_\measi(\bsx)\qquad\forall \bsx\in U
  \end{equation*}
  and uniform convergence of
  \begin{equation*}
    \lim_{n\to\infty}\hat f_{\measii,n}(\bsy_{[n]})= f_\measii(\bsy)\qquad\forall \bsy\in U.
  \end{equation*}
  The latter implies with $\bsy=T(\bsx)$ that
  \begin{equation*}
    \lim_{n\to\infty}\hat f_{\measii,n}(T_{[n]}(\bsx_{[n]}))=f_\measii(T(\bsx))\qquad\forall \bsx\in U
  \end{equation*}
  converges uniformly.  Since $f_\measii:U\to\R_+$ is continuous and
  $U$ is compact, we can conclude that $\hat f_{\measii,n}(\bsx)\ge r$
  (cp.~\eqref{eq:fk2}) for some $r>0$ independent of $n\in\N$ and
  $\bsx\in \U{n}$. Thus the right-hand side of \eqref{eq:detfrac}
  converges uniformly, and
  \begin{equation*}
    \det dT(\bsx)\dfn \lim_{n\to\infty} \prod_{j=1}^n\partial_jT_j(\bsx_{[j]})
    = {\frac{f_\measi(\bsx)}{f_\measii(T(\bsx))}} \in C^0(U;\R_+)
  \end{equation*}
  converges uniformly. Moreover
  $\det dT(\bsx)f_\measii(T(\bsx))=f_\measi(\bsx)$ for all $\bsx\in U$.
\end{proof}

\section{Proofs of Sec.~\ref{sec:infanalyticity}}
\subsection{Prop.~\ref{PROP:COR:DININF}}
\label{app:thm:cor:dinninf}

{The proposition is a consequence of the finite dimensional result
shown in \cite{zm1}. For better readability, we recall the statement
here together with its requirements; see \cite[Assumption 3.5, Thm.~3.6]{zm1}:

\begin{assumption}\label{ass:finite}
  Let $0<\hat M<\hat L$, ${\hat C_1}>0$,
  $k\in\N$ and
  $\bsdelta\in (0,\infty)^k$ be given. 
  For $*\in\{\measi,\measii\}$:
  \begin{enumerate}[label=(\alph*)]
  \item\label{item:cordinN:1}
    $\hat f_*:\U{k}\to\R_+$ is a probability density
    and
    $\hat f_{*}\in C^1(\cB_{\bsdelta}(\U{1});\C)$,
  \item\label{item:cordinN:2} $\hat M\le |\hat f_{*}(\bsx)|\le \hat L$ for
    $\bsx\in \cB_{\bsdelta}(\U{1})$,
  \item\label{item:cordinN:3}
    $\sup_{\bsy\in 
      \cB_{\bsdelta}}|\hat f_{*}(\bsx+\bsy)-\hat f_{*}(\bsx)| \le {\hat C_1}$
    for $\bsx\in\U{k}$,
  \item\label{item:cordinN:4}
    $      \sup_{\bsy\in
        \cB_{\bsdelta_{[k]}}\times \{0\}^{k-j}}|\hat f_{*}(\bsx+\bsy)-\hat f_{*}(\bsx)|\le {\hat C_1} \delta_{k+1}
$
    for $\bsx\in\U{k}$ and $j\in\{1,\dots,k-1\}$.
  \end{enumerate}
\end{assumption}

\begin{theorem}\label{THM:COR:DINN}
  Let $0<\hat M\le \hat L<\infty$, $k\in\N$ and $\bsdelta\in
  (0,\infty)^k$. There exist ${\hat C_1}$, ${\hat C_2}$ and ${\hat C_3}>0$ depending on $\hat M$
  and $\hat L$ (but not on $k$ or $\bsdelta$) such that if Assumption
  \ref{ass:finite} holds with ${\hat C_1}$, then:

  Let $H:\U{k}\to \U{k}$ be the KR-transport as in
  \eqref{eq:T} such that $H$ pushes forward the measure with density
  $\hat f_\measi$ to the one with density $\hat f_\measii$.  Set
  $R_k\dfn \partial_{k}H_k$.  With $\bszeta=(\zeta_j)_{j=1}^k$ where
  $\zeta_{j}\dfn {\hat C_2} \delta_{j}$,
  it holds for all $j\in\{1,\dots,k\}$:
  \begin{enumerate}
  \item\label{item:cordinN:a}
    $R_j\in C^1(\cB_{\bszeta_{[j]}}(\U{1});\cB_{{\hat C_3}}(1))$ and
    $\Re(R_j(\bsx))\ge \frac{1}{{\hat C_3}}$ for all
    $\bsx\in \cB_{\bszeta_{[j]}}(\U{1})$,
  \item\label{item:cordinN:b} if $j\ge 2$,
    $R_j:\cB_{\bszeta_{[j-1]}}(\U{1})\times \U{1}\to
    \cB_{\frac{{\hat C_3}}{\max\{1,\delta_j\}}}(1)$.
  \end{enumerate}
\end{theorem}}

\begin{proof}[Proof of Prop.~\ref{PROP:COR:DININF}]
  For $*\in\{\measi,\measii\}$ 
  and $\bsz\in\cB_\bsdelta(\U{1})\subseteq\C^k$
  let
  \begin{equation*}
    \hat f_{*,k}(\bsz)\dfn \int_U f_*(\bsz,\bsy)\dd\mu(\bsy)
  \end{equation*}
    be the extension of
  \eqref{eq:fk2} to complex numbers.  By Lemma \ref{LEMMA:FC0},
  $\hat f_{*,k}\in C^0(\U{k})$. Moreover with $\measi_k$ and
  $\measii_k$ being the marginal measures on $\U{k}$ in the first
  $k$ variables, by definition
  $\hat f_{\measi,k}=\frac{\ddd\measi_k}{\dd\mu}$ and
  $\hat f_{\measii,k}=\frac{\ddd\measii_k}{\dd\mu}$. In other words, these
  functions are the respective marginal densities in the first $k$
  variables.

  Let $H:\U{k}\to \U{k}$ be the KR-transport satisfying
  $H_\sharp\measi_k=\measii_k$, and let $T:U\to U$ be the KR-transport
  satisfying $T_\sharp\measi=\measii$.  By construction
  (cp.~\eqref{eq:T}) and uniqueness of the KR-transport, it holds
  $T_{[n]}=(T_j)_{j=1}^k=(H_j)_{j=1}^k=H$.  In order to complete the
  proof, we will apply Thm.~\ref{THM:COR:DINN} to $H$.  To this end we
  need to check Assumption \ref{ass:finite} for the densities
  $\hat f_{\measi,k}$, $\hat f_{\measii,k}:\U{k}\to\R$.  %
  We will do so with the constants
  \begin{equation}\label{eq:newconstants}
  \hat M\dfn \frac{M}{2},\qquad
  \hat L\dfn L+ \frac{M}{2},\qquad
  C_1(M,L)\dfn \min\left\{\frac{ M}{2},{\hat C_1}(\hat M,\hat L)\right\},
  \end{equation}
  where ${\hat C_1}(\hat M,\hat L)$ is as in Thm.~\ref{THM:COR:DINN}.
  Assume for the moment that $\hat f_{\measi,k}$,
  $\hat f_{\measii,k}:\U{k}\to\R$ satisfy Assumption
  \ref{ass:finite} with $\hat M$ and $\hat L$. Then
  Thm.~\ref{THM:COR:DINN} immediately implies the statement of
  Prop.~\ref{PROP:COR:DININF} with
  $C_2(M,L) \dfn {\hat C_2}(\hat M,\hat L)$ and
  $C_3(M,L)\dfn {\hat C_3}(\hat M,\hat L)$, where ${\hat C_2}$ and ${\hat C_3}$ are as
  in Thm.~\ref{THM:COR:DINN}.

  It remains to verify Assumption
  \ref{ass:finite}. %
  We do so item by item and fix $*\in\{\measi,\measii\}$:
  \begin{enumerate}[label=(\alph*)]
  \item By Lemma \ref{LEMMA:FC0}, $\hat f_{*,k}\in C^0(\U{k})$ and
    $\int_{\U{k}}\hat f_{*,k}(\bsx)\dd\mu(\bsx)=\int_U
    f_*(\bsy)\dd\mu(\bsy)=1$, so that $\hat f_{*,k}$ is a positive
    probability density on $\U{k}$.

    Fix $\bsz\in\cB_\bsdelta(\U{1})\subseteq \C^k$ and
    $i\in\{1,\dots,k\}$. We want to show that
    $z_i\mapsto \hat f_{*,k}(\bsz)\in\C$ is complex differentiable for
    $z_i\in\cB_{\delta_i}(\U{1})$. {It holds:
    \begin{itemize}
    \item By Assumption \ref{ass:infdens} \ref{item:cordinN:0inf},
      $\bsy\mapsto f_*(\bsz,\bsy):U\to\C$ is continuous and therefore
      measurable for all $z_i\in \cB_{\delta_i}(\U{1})$.
      \item By Assumption
    \ref{ass:infdens} \ref{item:cordinN:1inf}, for every fixed
    $\bsy\in U$,
    $z_i\mapsto f_*(\bsz,\bsy):\cB_{\delta_i}(\U{1})\to\C$ is
    differentiable.
    \item By Assumption \ref{ass:infdens} \ref{item:cordinN:0inf}
      $f_*:\cB_\bsdelta(\U{1})\times U\to\C$ is continuous.
      Thus, compactness of
    $\bar\cB_r(z_i)\times U$ (w.r.t.\ the product topology), implies
    that for every $z_i\in \cB_{\delta_i}(\U{1})$ with $r>0$ s.t.\
    $\bar\cB_r(z_i)\subseteq \cB_{\delta_i}(\U{1})$, holds
    $\sup_{x\in\cB_{r}(z_i)}\sup_{\bsy\in
      U}|f_*(\bsz_{[i-1]},x,\bsz_{[i+1:k]},\bsy)|<\infty$. Hence
    \begin{equation*}
      z_i\in\cB_{\delta_i}(\U{1})~\Rightarrow~\exists r>0:~
      \sup_{x\in\cB_{r}(z_i)}\int_{\bsy\in
      U}|f_*(\bsz_{[i-1]},x,\bsz_{[i+1:k]},\bsy)|\dd\mu(\bsy)<\infty.
    \end{equation*}
    \end{itemize}}
    According to the main theorem in \cite{MR1823156}, this implies
    $z_i\mapsto \hat f_{*,k}(\bsz)=\int_U f_*(\bsz,\bsy)\dd\mu(\bsy)$
    to be differentiable on $\cB_{\delta_i}(\U{1})$. Since
    $i\in\{1,\dots,k\}$ was arbitrary, Hartog's theorem,
    e.g.~\cite[Thm.~1.2.5]{krantz}, yields
    $\bsz\mapsto \hat f_{*,k}(\bsz):\cB_\bsdelta(\U{1})\to\C$ to be
    differentiable.
  \item By Assumption \ref{ass:infdens} \ref{item:cordinN:2inf}, and
    because $f_*(\bsy)\in\R_+$ for $\bsy\in U$, we have
    $M\le f_*(\bsy)\le L$ for all $\bsy\in U$. Thus
    $\hat f_{*,k}(\bsx)=\int_U f(\bsx,\bsy)\dd\mu(\bsy)\ge M$ and also
    $\hat f_{*,k}(\bsx)\le L$ for all $\bsx\in \U{k}$. Furthermore,
    for $\bsz\in \cB_{\bsdelta}\subseteq\C^k$ and $\bsx\in \U{k}$, by
    Assumption \ref{ass:infdens} \ref{item:cordinN:3inf}
    and \eqref{eq:newconstants}
    \begin{equation*}
      |\hat f_{*,k}(\bsx+\bsz)-\hat f_{*,k}(\bsx)|\le
      \int_U|f_*(\bsx+\bsz,\bsy)-f_*(\bsx,\bsy)|\dd\mu(\bsy)
      \le C_1 \le \frac{M}{2}.
    \end{equation*}
    Thus, %
    with $\hat M=\frac{M}{2}>0$ and $\hat L=L+\frac{\hat M}{2}$
    we have $\hat M\le |\hat f_{*,k}(\bsz)|\le \hat L$ for all
    $\bsz\in \cB_{\bsdelta}(\U{1})$.
  \item For $\bsx\in \U{k}$ by Assumption \ref{ass:infdens}
    \ref{item:cordinN:3inf} and \eqref{eq:newconstants}
    \begin{equation*}
      \sup_{\bsz\in 
        \cB_{\bsdelta}}|\hat f_{*,k}(\bsx+\bsz)-\hat f_{*,k}(\bsx)| \le
      \sup_{\bsy\in \cB_{\bsdelta}}
      \int_U |f_{*}(\bsx+\bsz,\bsy)-f_{*}(\bsx,\bsy)|\dd\mu(\bsy)
      \le C_1(M,L) \le {\hat C_1}(\hat M,\hat L).
    \end{equation*}    
  \item For $\bsx\in\U{k}$ and $j\in\{1,\dots,k-1\}$
    by Assumption \ref{ass:infdens} \ref{item:cordinN:4inf}
    and \eqref{eq:newconstants}
    \begin{align*}
      \sup_{\bsz\in
      \cB_{\bsdelta_{[j]}}\times \{0\}^{k-j}}|\hat f_{*,k}(\bsx+\bsz)-\hat f_{*,k}(\bsx)|&\le \sup_{\bsz\in
                                                                                           \cB_{\bsdelta_{[j]}}\times \{0\}^{\N}}
                                                                                           \int_U |f_{*}(\bsx+\bsz,\bsy)-f_{*}(\bsx,\bsy)|\dd\mu(\bsy)
                                                                                           \nonumber\\
                                                                                         &\le C_1(M,L) \delta_{j+1}
                                                                                           \le {\hat C_1}(\hat M,\hat L)\delta_{j+1}.\qedhere
    \end{align*}
  \end{enumerate}
\end{proof}

\subsection{Verifying Assumption \ref{ass:infdens}}\label{app:lemma:scrfhol}
In this section we show that densities as in Assumption
\ref{ass:density} satisfy Assumption \ref{ass:infdens}.

\begin{lemma}\label{LEMMA:SCRFHOL}
  {Let $f(\bsy)=\scr{f}(\sum_{j\in\N}y_j\psi_j)$ satisfy Assumption
    \ref{ass:density} for some $p\in (0,1)$ and $0<M\le L<\infty$.
    Let $(b_j)_{j\in\N}\subset (0,\infty)$ be summable and such that
    $b_j\ge \norm[Z]{\psi_j}$ for all $j\in\N$. Let $C_1=C_1(M,L)>0$
    be as in Prop.~\ref{PROP:COR:DININF}.}

  There
  exists a monotonically increasing sequence
  $(\kappa_j)_{j\in\N}\in (0,\infty)^\N$ and $\tau>0$ (depending on
  $(b_j)_{j\in\N}$, $C_1$ and $\scr{f}$) such that for every
  fixed $J\in\N$, $k\in\N$ and $\bsnu\in\N_0^k$, with
  \begin{equation}\label{eq:rhonu}
    \delta_j{=\delta_j(J,\bsnu)}\dfn \kappa_j+ %
    \begin{cases}
      0 &j<J,~j\neq k\\
      \frac{\tau \nu_j}{(\sum_{i=J}^{k-1}\nu_i)b_j} &j\ge J,~j\neq k\\
      \frac{\tau}{b_j} &j=k
    \end{cases}
    \qquad\qquad\forall j\in\{1,\dots,k\},
  \end{equation}
  $f$ satisfies Assumption \ref{ass:infdens}.
\end{lemma}

\begin{lemma}\label{lemma:kappajinfty}
  Let $\bsb=(b_j)_{j\in\N}\in\ell^1(\N)$ with $b_j\ge 0$ for all $j$,
  and let $\gamma>0$. There exists
  $(\kappa_j)_{j\in\N}\subset (0,\infty)$ monotonically increasing and
  such that $\kappa_j\to\infty$ and $\sum_{j\in\N}b_j\kappa_j<\gamma$.
\end{lemma}

\begin{proof}
  If there exists $d\in\N$ such that $b_j=0$ for all $j>d$, then the
  statement is trivial. Otherwise, for $n\in\N$ set
  $j_n\dfn {\min\set{j\in\N}{\sum_{i\ge j}b_i\le 2^{-n}}}$. Since
  $\bsb\in\ell^1(\N)$, $(j_n)_{n\in\N}$ is well-defined, monotonically
  increasing, and tends to infinity (it may have repeated
  entries). For $j\in\N$ let
  \begin{equation*}
    \tilde\kappa_j\dfn \begin{cases}
      1 &\text{if }j<j_1\\
      n &\text{if }j\in \N\cap [j_n,j_{n+1}),
    \end{cases}
  \end{equation*}
  which is well-defined since $j_n\to\infty$ so that
  \begin{equation*}
    \N=\{1,\dots,j_1\}\cup \bigcup_{n\in\N}(\N\cap [j_n,j_{n+1}))
  \end{equation*}
  and those sets are disjoint, in particular if $j_n=j_{n+1}$ then
  $[j_n,j_{n+1})\cap\N=\emptyset$. Then
  \begin{equation*}
    \sum_{j\in\N}b_j\tilde\kappa_j = \sum_{j=1}^{j_1-1}b_j+\sum_{j\ge j_1}b_j\tilde\kappa_j
    = \sum_{j=1}^{j_1-1}b_j+\sum_{n\in\N}\sum_{j=j_n}^{j_{n+1}-1} b_j\tilde\kappa_j\le \sum_{j=1}^{j_1-1}b_j+\sum_{n\in\N}n 2^{-n}<\infty.
  \end{equation*}
  Set
  $\kappa_j\dfn
  \frac{\gamma\tilde\kappa_j}{\sum_{j\in\N}b_j\tilde\kappa_j}$.
\end{proof}

\begin{proof}[Proof of Lemma \ref{LEMMA:SCRFHOL}]
  In Steps 1-2 we will construct
  $(\kappa_j)_{j\in\N}\subset (0,\infty)$ and $\tau>0$ independent of
  $J\in\N$, $k\in\N$ and $\bsnu\in\N_0^k$.  In Steps 3-4, we verify
  that $(\kappa_j)_{j\in\N}$ and $\tau$ have the desired properties.

  {Moreover, we will use that $Z$ is a Banach space, $Z_\C$ its
  complexification as introduced in (and before) Assumption
  \ref{ass:density}, and $\psi_j\in Z\subseteq Z_\C$ for all $j$.}
  
  {\bf Step 1.} Set
  $K\dfn \set{\sum_{j\in\N}y_j\psi_j}{\bsy\in U}\subseteq Z$.
  According to \cite[Rmk.~2.1.3]{JZdiss},
  $\bsy\mapsto\sum_{j\in\N}y_j\psi_j:U\to Z$ is continuous
  and $K\subseteq Z$ is compact (as the image of a compact set under a
  continuous map). Compactness of $K$ and continuity of $\scr{f}$
  imply $\sup_{\psi\in K}|\scr{f}(\psi)|<\infty$ and
  \begin{equation}\label{eq:choosealpha}
    \lim_{\eps\to 0}\sup_{\norm[Z]{\psi}<\eps}\sup_{\phi\in K}|\scr{f}(\psi+\phi)-\scr{f}(\phi)|=0.
  \end{equation}
  Hence there exists $r>0$ {such that with $O_Z\subseteq Z_\C$ from Assumption \ref{ass:density}
  \begin{equation}\label{eq:OZK}
    \set{\phi+\psi}{\phi\in K,~\norm[Z_\C]{\psi}<r}\subseteq O_Z
  \end{equation}
  and}
  \begin{equation}\label{eq:Cscrf}
    C_{\scr{f}}\dfn \sup_{\norm[Z]{\psi}<r}\sup_{\phi\in K}|\scr{f}(\phi+\psi)|<\infty
  \end{equation}
  and
  \begin{equation}\label{eq:choicer}
    \sup_{\norm[Z]{\psi}<r}\sup_{\phi\in K}|\scr{f}(\psi+\phi)-\scr{f}(\phi)|<{C_1}.
  \end{equation}

  {\bf Step 2.} We show the existence of
  $\bskappa=(\kappa_j)_{j\in\N}\subset (0,\infty)$ monotonically
  increasing, %
  and $\tau>0$ such that with
  $r>0$ from Step 1
  \begin{equation}\label{eq:kappabjalphar}
    \sum_{j\in\N}\kappa_jb_j+2\tau<r,
  \end{equation}
  and additionally for every $j\in\N$ with $K\subseteq Z$ from Step 1
  \begin{equation}\label{eq:toshowkappa}
    \sup_{\bsz\in\cB_{\bskappa_{[j]}}\times\{0\}^\N}
    \sup_{\norm[Z_\C]{\psi}<2\tau}
    \sup_{\phi\in K}
    \left|\scr{f}\left(\phi+\psi+\sum_{j\in\N}z_j\psi_j\right)-\scr{f}\left(\phi\right)\right|\le {C_1} \kappa_{j+1}.
  \end{equation}

  Let $(\tilde\kappa_j)_{j\in\N}\to\infty$ be as in Lemma
  \ref{lemma:kappajinfty} such that
  $\sum_{j\in\N}\tilde\kappa_jb_j<\frac{r}{3}$ and with
  $\tilde\tau\dfn \frac{r}{3}$ it holds
  \begin{equation*}
    \sum_{j\in\N}\tilde\kappa_jb_j+2\tilde\tau<r.
  \end{equation*}
  Since $\tilde\kappa_j\to\infty$ as $j\to\infty$, %
  there exists $d\in\N$ such that
  ${C_1}\tilde\kappa_{j+1}\ge 2 C_{\scr{f}}$ for all $j\ge d$
  (with $C_{\scr{f}}$ as in \eqref{eq:Cscrf}). For all
  $\bsz\in\cB_{\tilde\bskappa}{\subseteq\C^\N}$ using
  $\norm[Z]{\psi_j}\le b_j$
  \begin{equation*}
    \sup_{\norm[Z_\C]{\psi}<2\tilde\tau}
    \normc[Z_\C]{\psi+\sum_{j\in\N}z_j\psi_j}
    \le 2\tilde\tau+\sum_{j\in\N}\tilde\kappa_j \norm[Z]{\psi_j}
    \le 2\tilde\tau+\sum_{j\in\N}\tilde\kappa_jb_j\le r
  \end{equation*}
  and thus by \eqref{eq:Cscrf}
  for $\phi\in K$ and $\norm[Z_\C]{\psi}<2\tilde\tau$
  \begin{equation}\label{eq:geJ}
    \left|\scr{f}\left(\phi+\psi+\sum_{j\in\N}z_j\psi_j\right)
      -\scr{f}\left(\phi\right)\right|\le 2C_{\scr{f}}
    \le {C_1}\tilde\kappa_{j+1}
  \end{equation}
  for all $j\ge d$. Hence \eqref{eq:toshowkappa} holds for
  $\tilde\bskappa$
  for all
  $j\ge d$.

  To finish the construction of $\bskappa$, first define
  $\kappa_j\dfn \tilde\kappa_j$ for all $j\ge d$. For
  $k\in\{1,\dots,d-1\}$, inductively (starting with $k=d-1$ and going
  backwards) let $\tilde\tau_k>0$ and $\kappa_k\in (0,\kappa_{k+1})$
  be so small that
  \begin{equation}\label{eq:leJ}
    \sup_{\substack{|z_j|\le\kappa_k\\ \forall j\le k}}
    \sup_{\norm[Z]{\psi}<2\tilde \tau_k}
    \sup_{\phi\in K}
    \left|\scr{f}\left(\phi+\psi+\sum_{j=1}^kz_j\psi_j\right)
      -\scr{f}\left(\phi\right)\right|\le {C_1}\kappa_{k+1}
  \end{equation}
  which is possible due to \eqref{eq:choosealpha} and because
  ${C_1}\kappa_{k+1}>0$. Letting
  $\tau\dfn\min\{\tilde\tau,\tilde\tau_1,\dots,\tilde\tau_d\}$,
  it now follows by \eqref{eq:geJ} and \eqref{eq:leJ} that
  \eqref{eq:toshowkappa} holds for all $j\in\N$.

  {\bf Step 3.} We verify Assumption \ref{ass:infdens}
  \ref{item:cordinN:0inf}, \ref{item:cordinN:1inf}
  and \ref{item:cordinN:2inf}. Fix $J\in\N$,
  $k\in\N$ and $\bsnul\neq \bsnu\in\N_0^k$. By definition of
  $\bsdelta\in\R^k$ in %
  \eqref{eq:rhonu} and with $b_j\ge \norm[Z]{\psi_j}$
  \begin{align*}
    \sup_{\bsz\in\cB_\bsdelta\times\{0\}^\N}\sup_{\bsy\in U}\normc[Z]{\sum_{j\in\N}(y_j+z_j)\psi_j}
    &\le\sum_{j\in\N}b_j(1+\kappa_j)+\tau \frac{b_k}{b_k}
    +\sum_{j=J}^{k-1}\frac{\tau\nu_j b_j}{(\sum_{i=J}^{k-1}\nu_i) b_j}\nonumber\\
    &\le \sum_{j\in\N}b_j+\sum_{j\in\N}\kappa_jb_j+2\tau
    <\infty
  \end{align*}
  and, similarly, by \eqref{eq:kappabjalphar}
  \begin{equation}\label{eq:psilessr}
    \sup_{\bsz\in\cB_\bsdelta\times\{0\}^\N}\sup_{\bsy\in U}\normc[Z]{\sum_{j\in\N}(y_j+z_j)\psi_j-\sum_{j\in\N}y_j\psi_j}
    \le \sum_{j\in\N}\kappa_jb_j+2\tau<r.
  \end{equation}
  Thus by \eqref{eq:OZK},
  $f(\bsy+\bsz)=\scr{f}(\sum_{j\in\N}(y_j+z_j)\psi_j)$ is well-defined
  for all $\bsy\in U$, $\bsz\in\cB_\bsdelta\times\{0\}^\N$,
  {since then $\sum_{j\in\N}(y_j+z_j)\psi_j\in O_Z$, where
    $O_Z\subseteq Z_\C$ is the domain of definition of $\frf$.}
  Summability of $(\norm[Z]{\psi_j})_{j\in\N}$ implies continuity of
  $f:\cB_{\bsdelta}(\U{1})\times U\to O_Z$ w.r.t.\ the product topology
  on $\cB_{\bsdelta}(\U{1})\times U\subseteq\C^\N$ (see, e.g.,
  \cite[Rmk.~2.1.3]{JZdiss}).  Continuity of
  $\scr{f}: O_Z\to \C$ thus
  implies $f\in C^0(\cB_\bsdelta\times\{0\}^\N;\C)$, i.e.,
  Assumption \ref{ass:infdens} \ref{item:cordinN:0inf} holds.

  Differentiability of $\scr{f}$
  implies that $f(\bsz)=\scr{f}(\sum_{j\in\N}z_j\psi_j)$ is
  differentiable in each $z_j$ for
  $\bsz\in\cB_\bsdelta(\U{1})\times U$, proving
  Assumption \ref{ass:infdens} \ref{item:cordinN:1inf}. Finally
  Assumption \ref{ass:infdens} \ref{item:cordinN:2inf} is a consequence of
  Assumption \ref{ass:density} \ref{item:densityML}.
  
  {\bf Step 4.} We show Assumption \ref{ass:infdens}
  \ref{item:cordinN:3inf} and \ref{item:cordinN:4inf}. Fix $k\in\N$,
  $\bsnul\neq\bsnu\in\N_0^k$ and $J\in\N$. Then for any 
  $\bsz\in\cB_{\bsdelta}\subseteq\C^k$,
  by \eqref{eq:rhonu}
  we can write
  $z_i=z_{i,1}+z_{i,2}$ with
  {
  \begin{equation*}
    |z_{i,1}|\le 
      \kappa_i,\qquad
      |z_{i,2}|\le \begin{cases}
        0 &i<J,~i\neq k\\
      \frac{\tau\nu_i}{(\sum_{r=J}^{k-1}\nu_r)b_i}&i\ge J,~i\neq k\\
      \frac{\tau}{b_k}&\text{if }i=k
    \end{cases}
    \qquad\qquad\forall i\le k.
    \end{equation*}}
    Thus for any
    $\bsy\in U$, $j\in\{1,\dots,k-1\}$ and $\bsz\in\cB_{\bsdelta_{[j]}}\times\{0\}^\N$
  \begin{equation*}
    \sum_{i\in\N}(y_i+z_i)\psi_i
    = \underbrace{\left(\sum_{i\in\N}y_i\psi_i\right)}_{\dfnn\phi}+
    \underbrace{\left(\sum_{i=1}^j z_{i,2} \psi_i\right)}_{\dfnn
      \psi}+
    \left(\sum_{i=1}^j z_{i,1} \psi_i\right).
  \end{equation*}
  With $K\subseteq Z$ from Step 1,
  $\norm[Z_\C]{\psi}\le \sum_{i\in\N}\tau
  \frac{\nu_ib_i}{(\sum_{r=J}^{k-1}\nu_r)b_i}+\tau\frac{b_k}{b_k}\le2\tau$
  and $\phi\in K$. Thus \eqref{eq:toshowkappa} %
  implies Assumption \ref{ass:infdens} \ref{item:cordinN:4inf}.
  Finally, %
  Assumption \ref{ass:infdens} \ref{item:cordinN:3inf} is a
  consequence of \eqref{eq:choicer} and \eqref{eq:psilessr}.
\end{proof}

\section{Proofs of Sec.~\ref{sec:polinfty}}
\subsection{Thm.~\ref{THM:TINF}}\label{app:tinf}
First we show two summability results, similar to \cite[Lemma
7.1]{CDS10} and \cite[Theorem 7.2]{CDS10}. {In the following we write
$\N=\{1,2,3,\dots\}$ and $\N_0=\{0,1,2,\dots\}$. For an index set
$I\subseteq\N$, $\bsnu=(\nu_j)_{j\in I}\in\N_0^I$ and
$\bsvarrho=(\varrho_j)_{j\in I}\in [0,\infty)^I$ we use the notation
\begin{equation*}
  \supp\bsnu=\set{j\in I}{\nu_j\neq 0},\qquad
|\bsnu|\dfn \sum_{j\in \supp\bsnu}\nu_j,\qquad
\bsvarrho^{\bsnu}\dfn\prod_{j\in \supp\bsnu}\varrho_j^{\nu_j},
\end{equation*}
where empty sums equal $0$ and empty products equal $1$.
}

\begin{lemma}\label{lemma:Sxi0}
  Let $\tau>0$ and let $\bsvarrho\in (1,\infty)^\N$ be such that
  $(\varrho_j^{-1})\in\ell^p(\N)$ for some $p\in (0,1]$ and
  additionally $\sup_{j\in\N}\varrho_j^{-1}<1$. Then with
  $\gamma(\bsvarrho,\bsnu)$ as in \eqref{eq:gamma},
  \begin{equation*}
    \sum_{k\in\N} \sum_{\bsnu\in\N_0^k}\gamma(\bsvarrho,\bsnu)^p\prod_{j=1}^k(1+2\nu_j)^\tau<\infty.
  \end{equation*}
\end{lemma}
\begin{proof}
  {The assumptions on $\bsvarrho$ imply
  \begin{equation*}
    C_0\dfn \sum_{\set{\bsnu\in\N_0^\N}{|\bsnu|<\infty}}\bsvarrho^{-p\bsnu}\prod_{j\in\N}(1+2\nu_j)^{\tau}<\infty
  \end{equation*}
  see for example \cite[Lemma 3.10]{ZS17}.  Let
  $\bsnul\dfn (0)_{j\in\N}\in\N_0^\N$. For any
  $\bsnul\neq \bsnu\in\N_0^\N$ with $|\bsnu|<\infty$, we have
  $\bsnu=(\bseta,\bsnul)$ with $\bseta\in\N_0^{k-1}\times\N$ and
  $k\dfn \max_j\nu_j\neq 0$. Thus with the convention
  $\N_0^0\times\N\dfn \N$,
  \begin{equation*}
    \set{\bsnu\in\N_0^\N}{|\bsnu|<\infty}=\{\bsnul\}\cup\bigcup_{k\in\N}\set{(\bsnu,\bsnul)}{\bsnu\in\N_0^{k-1}\times\N}.
  \end{equation*}
  Hence
\begin{equation*}
  1+\sum_{k\in\N}\sum_{\bsnu\in\N_0^{k-1}\times\N}\bsvarrho_{[k]}^{-p\bsnu}\prod_{j=1}^k(1+2\nu_j)^{\tau}= C_0.
\end{equation*}
Using the convention $\bsvarrho_{[0]}^{-\bsnu_{[0]}}=1$, by definition
  \begin{equation*}
    \gamma(\bsvarrho,\bsnu)=
    \begin{cases}
      \varrho_k^{-1}\bsvarrho_{[k-1]}^{-\bsnu_{[k-1]}} &\text{if }\nu_k=0\\
      \bsvarrho_{[k]}^{-\bsnu} &\text{if }\nu_k>0
    \end{cases}
    \qquad\qquad\forall\bsnu\in\N_0^k.
  \end{equation*}
  Partitioning $\N_0^k=(\N_0^{k-1}\times \{0\})\cup (\N_0^{k-1}\times\N)$
  we get}
  \begin{align*}
    &\sum_{k\in\N} \sum_{\bsnu\in\N_0^k}\gamma(\bsvarrho,\bsnu)^p
      \prod_{j=1}^k(1+2\nu_j)^\tau\\
    &\qquad\qquad=\sum_{k\in\N}\varrho_k^{-p}\sum_{\bsnu\in\N_0^{k-1}\times\{0\}}
      \bsvarrho_{[k-1]}^{-p\bsnu_{[k-1]}}\prod_{j=1}^{k-1}(1+2\nu_j)^\tau
      +\sum_{k\in\N}\sum_{\bsnu\in\N_0^{k-1}\times\N}
      \bsvarrho_{[k]}^{-p\bsnu}\prod_{j=1}^k(1+2\nu_j)^\tau\\
    &\qquad\qquad\le
      \sum_{k\in\N}\varrho_k^{-p}C_0+ C_0<\infty,
  \end{align*}
  since $\sum_{k\in\N}\varrho_k^{-p}<\infty$.
\end{proof}

\begin{lemma}\label{lemma:Sxi}
  Let $\tau>0$ and let $\bsvarrho\in (1,\infty)^\N$ be such that
  $(\varrho_j^{-1})\in\ell^p(\N)$ for some $p\in (0,1]$ and
  additionally $\sum_{j\in\N}\varrho_j^{-1}<1$. Then with
  $\gamma(\bsvarrho,\bsnu)$ as in \eqref{eq:gamma}
  \begin{equation*}
    \sum_{k\in\N} \sum_{\bsnu\in\N_0^k}\left(\frac{|\bsnu|^{|\bsnu|}}{\bsnu^{\bsnu}}\gamma(\bsvarrho,\bsnu)\right)^p\prod_{j=1}^k(1+2\nu_j)^\tau<\infty.
  \end{equation*}
\end{lemma}
\begin{proof}
  By \cite[Lemma 3.11]{ZS17},
  the assumptions on $\bsvarrho$ {imply with $w_\bsnu=\prod_j(1+2\nu_j)^\tau$}
  \begin{equation*}
    1+
    \sum_{k\in\N}\sum_{\bsnu\in\N_0^{k-1}\times\N}\left(\frac{|\bsnu|^{|\bsnu|}}{\bsnu^{\bsnu}}\bsvarrho_{[k]}^{-\bsnu}\right)^p{w_\bsnu}
    =
    \sum_{\set{\bsnu\in\N_0^\N}{|\bsnu|<\infty}}
    \left(\frac{|\bsnu|^{|\bsnu|}}{\bsnu^{\bsnu}}\bsvarrho^{-\bsnu}\right)^p
    {w_\bsnu}
    \dfnn C_0<\infty.
  \end{equation*}
  {Hence, similar as in the proof of Lemma \ref{lemma:Sxi0}
  \begin{align*}
    &\sum_{k\in\N} \sum_{\bsnu\in\N_0^k}\left(\frac{|\bsnu|^{|\bsnu|}}{\bsnu^{\bsnu}}\gamma(\bsvarrho,\bsnu)\right)^pw_\bsnu\\
    &\qquad=\sum_{k\in\N}\varrho_k^{-p} \sum_{\bsnu\in\N_0^{k-1}\times\{0\}}      
      \left(\frac{|\bsnu|^{|\bsnu|}}{\bsnu^\bsnu}\bsvarrho_{[k-1]}^{-\bsnu_{[k-1]}}\right)^pw_\bsnu
      +\sum_{k\in\N}\sum_{\bsnu\in\N_0^{k-1}\times\N}
      \left(\frac{|\bsnu|^{|\bsnu|}}{\bsnu^\bsnu}\bsvarrho_{[k]}^{-\bsnu}\right)^p
      w_\bsnu\\
    &\qquad\le\sum_{k\in\N}\varrho_k^{-p}C_0+C_0.\qedhere
  \end{align*}}
\end{proof}

{
  In the following we denote by $L_n:\U{1}\to\R$ for $n\in\N_0$
  the $n$-th Legendre polynomial normalized in $L^2(\U{1},\mu)$.
  Then $(L_n)_{n\in\N_0}$ forms an orthonormal basis of this space.
  More generally, setting $L_\bsnu(\bsx)\dfn \prod_{j=1}^kL_{\nu_j}(x_j)$
  with $\bsnu\in\N_0^k$ for $\bsx\in \U{k}$, the family
  $(L_\bsnu)_{\bsnu\in\N_0^k}$ forms an orthonormal basis of $L^2(\U{k},\mu)$,
  and any function $f$ in this space allows the representation
  $f(\bsx)=\sum_{\bsnu\in\N_0^k}L_\bsnu(\bsx)l_{\bsnu}$
  with the coefficients $l_\bsnu=\int_{\U{k}}$.
  We have \cite[\S 18.2(iii) and \S
  18.3]{nist}
  \begin{equation}\label{eq:Lkbound}
    \norm[{L^{\infty}(\U{k})}]{L_\bsnu}\le \prod_{j=1}^k
    (1+2\nu_j)^{\frac{1}{2}}.
  \end{equation}

  To prove Thm.~\ref{THM:TINF} we will bound the Legendre coefficients
  of $\sqrt{\partial_kT_k}-1$. To this end we will use the next lemma,
  which we have also used in the analysis of the finite dimensional
  case; see \cite[Lemma 4.1]{zm1}. For a proof in the one
  dimensional case we refer to Chapter 12 in \cite{davis}; see the
  calculation in equations (12.4.24)--(12.4.26). The multidimensional
  case follows by applying the result in each variable separately,
  e.g., \cite{chkifa} or \cite[Cor.~B.2.7]{JZdiss}.

\begin{lemma}\label{lemma:legest}
  Let $\bszeta\in (0,\infty)^k$.
  Let $f:\cB_\bszeta(\U{1})\to \cB_{r_1}$
  be differentiable. Then %
  \begin{enumerate}
  \item\label{item:legest:1} for %
    all $\bsnu\in\N_0^k$
  \begin{equation}\label{eq:lknubound}
    \left|\int_{\U{d}} f(\bsy) L_\bsnu(\bsy)\dd\mu(\bsy) \right|
    \le 
    r_1
    \prod_{j\in\supp\bsnu}\left(\frac{2(\zeta_j+1)}{\zeta_j}(1+2\nu_j)^{3/2}\right)
    \prod_{j=1}^k(1+\zeta_j)^{-\nu_j},
  \end{equation}
\item\label{item:legest:2}
  if $f:\cB_{\bszeta_{[k-1]}}(\U{1})\times [-1,1]\to \cB_{r_2}$
  then for all $\bsnu\in\N_0^{k-1}\times\{0\}$
  \begin{equation}\label{eq:lknubound2}
    \left|\int_{\U{d}} f(\bsy) L_\bsnu(\bsy)\dd\mu(\bsy) \right|
    \le 
    r_2
    \prod_{j\in\supp\bsnu}\left(\frac{2(\zeta_j+1)}{\zeta_j}(1+2\nu_j)^{3/2}\right)
    \prod_{j=1}^{k}(1+\zeta_j)^{-\nu_j}.
  \end{equation}  
  \end{enumerate}
\end{lemma}
}

\begin{proof}[Proof of Thm.~\ref{THM:TINF}]
  We first define some constants used throughout the proof. Afterwards the
  proof proceeds in 5 steps.
  
  Let ${M}\le|\scr{f}_\measi(\psi)|,|\scr{f}_\measii(\psi)| \le {L}$
  as stated in Assumption \ref{ass:densities}. Let ${C_1}$, ${C_2}$,
  ${C_3}>0$ be the constants from Prop.~\ref{PROP:COR:DININF}
  {depending on $M$ and $L$}. Let
  $(\kappa_j)_{j\in\N}\subset (0,\infty)$ (monotonically increasing) and
  $\tau>0$ be as in Lemma \ref{LEMMA:SCRFHOL} (depending on
  $(b_j)_{j\in\N}$, ${C_1}$ and $\scr{f}_\measi$, $\scr{f}_\measii$).
  Then $\kappa_{\rm min}\dfn \min_{j\in\N}\kappa_j>0$. Fix $J\in\N$ so large and
  $\alpha>0$ so small that
  \begin{equation}\label{eq:choiceJ}
    \sum_{j\ge J} \left(\frac{b_j}{{C_2}\tau}\right)^{p}<1
    \quad\qquad\text{and}\qquad\quad
    1+\frac{\alpha}{b_j} <
    \begin{cases}
      1+{C_2}\kappa_{\rm min} &j<J\\
      \frac{{C_2}\tau}{b_j} &j\ge J
    \end{cases}\qquad\forall j\in\N.
  \end{equation}
  This is possible because $\bsb\in\ell^p(\N)$, since
  $b_j=\max\{\norm[Z_\measi]{\psi_{j,\measi}},\norm[Z_\measii]{\psi_{j,\measii}}\}$
  (cp.~Assumption \ref{ass:densities}). Then by Lemma
  \ref{LEMMA:SCRFHOL},
  $f_\measi(\bsy)=\scr{f}_\measi(\sum_{j\in\N}y_j\psi_{\measi,j})$ and
  $f_\measii(\bsy)=\scr{f}_\measii(\sum_{j\in\N}y_j\psi_{\measii,j})$
  satisfy %
  {Assumption \ref{ass:infdens}} with $(\delta_j)_{j\in\N}$ as in
  \eqref{eq:rhonu} (and with our above choice of $J\in\N$).
  
  {\bf Step 1.} We provide bounds on the Legendre coefficient
  \begin{equation}\label{eq:legrksqrt1}
    l_{k,\bsnu}\dfn\int_{\U{k}}(\sqrt{R_k(\bsx)}-1)L_\bsnu(\bsx)\dd\mu( \bsx)
  \end{equation}
  with $R_k=\partial_kT_k$ and $\bsnu\in\N_0^k$ for $k\in\N$.

  Fix $k\in\N$ and $\bsnul\neq \bsnu\in\N_0^k$,
  {and let $\delta_j=\delta_j(J,\bsnu)$ be as in \eqref{eq:rhonu}}.
  According to Prop.~\ref{PROP:COR:DININF} (applied with $j=k$)
  \begin{enumerate}
  \item\label{item:Rkzetak}
    $R_k\in C^1(\cB_{\bszeta_{[k]}}(\U{1});\cB_{{C_3}}(1))$ and
    $\Re(R_k(\bsx))\ge \frac{1}{{C_3}}$ for all
    $\bsx\in \cB_{\bszeta_{[k]}}(\U{1})$,
  \item\label{item:Rkzetak-1} if $k\ge 2$,
    $R_k:\cB_{\bszeta_{[k-1]}}(\U{1})\times \U{1}\to
    \cB_{\frac{{C_3}}{\delta_k}}(1)$,
  \end{enumerate}
  where $\zeta_j={C_2} \delta_j$, $j\in\{1,\dots,k\}$, and the
  constants ${C_2}$ and ${C_3}$ solely depend on ${M}$ and ${L}$ but
  not on $k$ or $\bsnu$. %
  In particular for
  \begin{equation}\label{eq:Rktilde}
    Q_k\dfn \sqrt{R_k}-1=\sqrt{\partial_kT_k}-1
  \end{equation}
  we get with ${C_4}\dfn \sqrt{1+{C_3}}+1$
  \begin{equation}\label{eq:claimtildeRk0}
    Q_k:\cB_{\bszeta_{[k]}}(\U{1})\to\cB_{{C_4}},
  \end{equation}
  which follows by \ref{item:Rkzetak} and
  $|\sqrt{R_k(\bsx)}-1|\le |\sqrt{R_k(\bsx)}|+1\le \sqrt{{C_3}+1}+1$ for all $\bsx\in\cB_{\bszeta_{[k]}}(\U{1})$. We claim
  that with $r\dfn 2{C_3}{C_4}\ge {C_3}$ if $k\ge 2$
  \begin{equation}\label{eq:claimtildeRk}
    Q_k:\cB_{\bszeta_{[k-1]}}(\U{1})\times \U{1}\to
    \cB_{\frac{r}{\delta_k}}.
  \end{equation}
  To show it fix $\bsx\in \cB_{\bszeta_{[k-1]}}(\U{1})\times \U{1}$.
  We distinguish between $\frac{{C_3}}{\delta_k}\le \frac{1}{2}$ and
  $\frac{{C_3}}{\delta_k}> \frac{1}{2}$.
  For any $q\in\C$ with $|q|\le \frac{1}{2}$ we have with
  $g(q)\dfn \sqrt{1+q}-1$ that $g(0)=0$ and $|g'(q)|\le \sqrt{\frac{1}{2}}$. Thus
  $|\sqrt{1+q}-1|\le |q|$ for all $|q|\le \frac{1}{2}$. Therefore if
  $\frac{{C_3}}{\delta_k}\le \frac{1}{2}$ then by \ref{item:Rkzetak-1}
  $|R_k(\bsx)-1|\le {\frac{C_3}{\delta_k}\le}\frac{1}{2}$ and thus
  \begin{equation*}
    |Q_k(\bsx)| =|\sqrt{(R_k(\bsx)-1)+1}-1|\le |R_k(\bsx)-1|
    \le \frac{{C_3}}{\delta_k}
    \le \frac{r}{\delta_k}.
  \end{equation*}
  For the second case $\frac{{C_3}}{\delta_k}>\frac{1}{2}$, by
  \ref{item:Rkzetak-1} we have
  $|\sqrt{R_k(\bsx)}-1|\le 1+|\sqrt{R_k(\bsx)}|\le 1+\sqrt{{C_3}+1}={C_4}$. Since $\frac{{C_3}}{\delta_k}>\frac{1}{2}$ and thus
  $\delta_k\le 2{C_3}$, we can bound ${C_4}$ by
  ${C_4}=\frac{r}{2{C_3}} \le \frac{r}{\delta_k}$, which concludes the
  proof of \eqref{eq:claimtildeRk}.

  The fact that $R_k$ has non-negative real part implies that its
  composition with the square root, i.e., the map
  $\bsx\mapsto \sqrt{R_k(\bsx)}$, is well-defined and differentiable
  on $\cB_{\bszeta_{[k]}}(\U{1})$.  With
  $\kappa_{\rm min}=\min_{j\in\N}\kappa_j>0$ set
  $\zeta_{\rm min}\dfn C_2 \kappa_{\rm min}>0$ and observe that
  $\zeta_j=C_2\delta_j \ge\zeta_{\rm min}$ for all $j\in\N$
  (cp.~\eqref{eq:rhonu}).  {Let
  \begin{equation}\label{eq:wbsnu}
    w_\bsnu=\prod_{j=1}^k(1+2\nu_j)^{\theta},
  \end{equation}
  with
  $\theta=\frac{3}{2}+\log_3(\frac{2(1+\zeta_{\rm min})}{\zeta_{\rm min}})$.
  Then
  \begin{equation*}
    \prod_{j\in\supp\bsnu} \frac{2(\zeta_j+1)}{\zeta_j}(1+2\nu_j)^{3/2}
    =
    \prod_{j\in\supp\bsnu} 3^{\log_3(\frac{2(\zeta_j+1)}{\zeta_j})}
    (1+2\nu_j)^{3/2}
    \le \prod_{j\in\supp\bsnu} (1+2\nu_j)^{\theta}=w_\bsnu.
  \end{equation*}
  }

  {With Lemma \ref{lemma:legest} \ref{item:legest:1} and
  \eqref{eq:claimtildeRk0} we obtain for the Legendre coefficients of}
  $Q_k$ in \eqref{eq:legrksqrt1}
  \begin{equation}\label{eq:lkbsnu1}
    |l_{k,\bsnu}|\le 
    w_\bsnu C_4
    \prod_{j=1}^k(1+\zeta_j)^{-\nu_j}
    \qquad\forall\bsnu\in\N_0^k.
  \end{equation}
  Moreover with Lemma \ref{lemma:legest} \ref{item:legest:2}
  and
  \eqref{eq:claimtildeRk}
  \begin{equation}\label{eq:lkbsnu2}
    |l_{k,\bsnu}|\le w_\bsnu
    \frac{r}{\delta_k} \prod_{j=1}^{k}(1+\zeta_j)^{-\nu_j}
    \qquad\forall \bsnu\in\N_0^{k-1}\times\{0\}.
  \end{equation}

  {\bf Step 2.} We provide a bound on $|l_{k,\bsnu}|$
  in terms of $\gamma(\tilde \bsvarrho,\bsnu)$ for some
  $\tilde\bsvarrho$.
  
  Fix again $k\in\N$ and $\bsnu\in\N_0^k$. For $j\in\N$ by definition
  of $\zeta_j={C_2}\delta_j$ and $\delta_j=\delta_j(J,\bsnu)$ in
  \eqref{eq:rhonu} %
  \begin{equation*}
    \zeta_j = {C_2}\delta_j =
    {C_2}
    \begin{cases}
      \kappa_j+0 &j\neq k,~j< J\\
      \kappa_j+\frac{\tau\nu_j}{(\sum_{i=J}^{k-1}\nu_i)b_j} %
      &j\neq k,~j\ge J\\
      \kappa_k+\frac{\tau}{b_k} &j=k.
    \end{cases}
  \end{equation*}
  Since ${C_2}\in (0,1]$ (see Prop.~\ref{PROP:COR:DININF}) and
  {$\kappa_{\rm min}\le\kappa_j$}, it holds with
  $|\bsnu_{[J:k]}|=\sum_{j=J}^k\nu_j\ge \sum_{j=J}^{k-1}\nu_j$
  \begin{equation*}
    \zeta_j\ge
\begin{cases}
      {C_2}\kappa_{\rm min} &j<J\\
      \frac{{C_2} \tau \nu_j}{|\bsnu_{[J:k]}|b_j} &j\ge J
    \end{cases}\qquad\forall j\in\{1,\dots,k\}
  \end{equation*}
  and additionally
  \begin{equation*}
    \frac{r}{\delta_k}=\frac{r}{\kappa_k+\tau/b_k}\le \frac{b_k
      r}{\tau}.
  \end{equation*}
  Thus by \eqref{eq:lkbsnu1} and \eqref{eq:lkbsnu2}
  for $k\in\N$ and $\bsnu\in\N_0^k$
  \begin{equation}\label{eq:lkbsnularge}
    |l_{k,\bsnu}|\le {C_4} w_\bsnu
    \prod_{j=1}^{J-1}
    (1+{C_2}\kappa_{\rm min})^{-\nu_j}
    \prod_{i=J}^{k}\frac{|\bsnu_{[J:k]}|^{\nu_i}}{\nu_i^{\nu_i}}
    \prod_{i=J}^{k} \left(\frac{b_i}{{C_2}\tau}\right)^{\nu_i}\cdot
    \begin{cases}
      1 &k\in\supp\bsnu\\
      \frac{b_k r}{\tau}&k\notin\supp\bsnu
    \end{cases}
  \end{equation}
  with empty products equal to $1$ by convention.

  Defining
  \begin{equation}\label{eq:txi}
    \tilde\varrho_j\dfn     \begin{cases}
      1+{C_2}\kappa_{\rm min} &j<J\\
      \frac{{C_2}\tau}{b_j} &j\ge J
      \end{cases}\qquad\forall j\in\N.
    \end{equation}
    the bound \eqref{eq:lkbsnularge} becomes with
    $\gamma(\tilde\bsvarrho,\bsnu)
    =\tilde\varrho_k^{-\max\{1,\nu_k\}}\prod_{j=1}^{k-1} \tilde
    \varrho_j^{-\nu_j}$
    \begin{align}\label{eq:lkbsnufinal}
      |l_{k,\bsnu}|&\le {C_4} w_\bsnu
                     \prod_{i=J}^{k}\frac{|\bsnu_{[J:k]}|^{\nu_i}}{\nu_i^{\nu_i}}
                     \prod_{j\in\supp\bsnu}\tilde\varrho_j^{-\nu_j}\cdot
      \begin{cases}
        1 &k\in\supp\bsnu\\
        \frac{b_kr}{\tau}&k\notin\supp\bsnu
      \end{cases}\nonumber\\
                   &={C_4} w_\bsnu
                     \prod_{i=J}^{k}\frac{|\bsnu_{[J:k]}|^{\nu_i}}{\nu_i^{\nu_i}}
      \gamma(\tilde\bsvarrho,\bsnu)
        \cdot\begin{cases}
          1&k\in\supp\bsnu\\
          \tilde\varrho_k\frac{b_k r}{\tau} &k\notin\supp\bsnu
        \end{cases}\nonumber\\
      &\le {C_5}w_\bsnu\gamma(\tilde\bsvarrho,\bsnu)\prod_{i=J}^{k}\frac{|\bsnu_{[J:k]}|^{\nu_j}}{\nu_i^{\nu_i}}
    \end{align}
    with
    \begin{equation*}
    {C_5}\dfn
    {C_4}\sup_{k\in\N}\tilde\varrho_k\frac{b_kr}{\tau}<\infty
    \end{equation*}
    which is finite by
    definition of $\tilde\varrho_j$ in \eqref{eq:txi}.
  
    With $\alpha$ in \eqref{eq:choiceJ} introduce
  \begin{equation}\label{eq:xi}
    \varrho_j\dfn 1+\frac{\alpha}{b_j} <
    \tilde \varrho_j\qquad\forall j\in\N.
    \end{equation}
    Then 
\begin{equation}\label{eq:tildebsxi}
  \gamma(\tilde\bsvarrho,\bsnu)\le\gamma(\bsvarrho,\bsnu).
\end{equation}

  {\bf Step 3.} We show a summability result for the
  Legendre coefficients.

    For notational convenience we introduce the shortcuts
  \begin{equation*}
    \bsnu_E\dfn \bsnu_{[J-1]},\quad
    \bsnu_F\dfn \bsnu_{[J:k]},\quad
    \tilde\bsvarrho_E\dfn \tilde\bsvarrho_{[J-1]},\quad
    \tilde\bsvarrho_F\dfn \tilde\bsvarrho_{[J:k]}.    
  \end{equation*}
  Hence $\tilde\bsvarrho_E^{-\bsnu_E}=\prod_{j=1}^{J-1}\varrho_j^{-\nu_j}$,
  $\bsnu_F^{\bsnu_F}=\prod_{j\ge J}\nu_j^{\nu_j}$,
  $\gamma(\bsvarrho_F,\bsnu_F)=\varrho_k^{-\max\{1,\nu_k\}}\prod_{j=J}^{k-1}\varrho_j^{-\nu_j}$
  in case $k\ge J$ etc. For $k\ge J$
  and $\bsnu\in\N_0^k$
  \begin{equation*}
    \gamma(\tilde\bsvarrho,\bsnu) = \tilde\bsvarrho_{[k]}^{-\bsnu}\cdot
    \begin{cases}
      \tilde\varrho_k^{-1} &k\notin\supp\bsnu\\
      1 &k\in\supp\bsnu
      \end{cases}
      = \tilde\bsvarrho_E^{-\bsnu_E}\tilde\bsvarrho_F^{-\bsnu_F}\cdot
    \begin{cases}
      \tilde\varrho_k^{-1} &k\notin\supp\bsnu\\
      1 &k\in\supp\bsnu
    \end{cases}
    = \tilde\bsvarrho_E^{-\bsnu_E}\gamma(\tilde\bsvarrho_F,\bsnu_F).
  \end{equation*}
  By \eqref{eq:tildebsxi} and because $p<1$ it holds
  $\gamma(\bsvarrho,\bsnu)^{p-1}\le \gamma(\tilde \bsvarrho,\bsnu)^{p-1}$.
  Thus by \eqref{eq:lkbsnufinal} and \eqref{eq:xi}
  \begin{align}\label{eq:jointsummability}
    &\sum_{k\in\N}\sum_{\bsnu\in\N_0^k}w_\bsnu
      |l_{k,\bsnu}|\gamma(\bsvarrho,\bsnu)^{p-1}
      \le{C_5}
      \sum_{k\in\N}\sum_{\bsnu\in\N_0^k}w_\bsnu^2
      \gamma(\tilde\bsvarrho_{[k]},\bsnu)
      \frac{|\bsnu_F|^{|\bsnu_F|}}{\bsnu_F^{\bsnu_F}}
      \gamma(\bsvarrho,\bsnu)^{p-1}
      \nonumber\\
    &\quad
      \le{C_5}\sum_{k=1}^{J-1}\sum_{\bsnu\in\N_0^k}w_\bsnu^2
      \gamma(\tilde\bsvarrho_{[k]},\bsnu)^p
      +{C_5}\sum_{k\ge J}\sum_{\bsnu\in\N_0^k}w_\bsnu^2      
      \gamma(\tilde\bsvarrho_{[k]},\bsnu)^p
      \frac{|\bsnu_F|^{|\bsnu_F|}}{\bsnu_F^{\bsnu_F}}      
      \nonumber\\
    &\quad\le {C_5} \sum_{k=1}^{J-1}\sum_{\bsnu\in\N_0^k}
      w_\bsnu^2 \gamma(\tilde\bsvarrho_{[k]},\bsnu)^p
      +{C_5} \sum_{k\ge J}\sum_{\bsnu\in\N_0^k}w_\bsnu^2
      \tilde\bsvarrho_E^{-p \bsnu_E}
      \gamma(\tilde\bsvarrho_F,\bsnu_F)^p
      \frac{|\bsnu_{F}|^{|\bsnu_{F}|}}{\bsnu_{F}^{\bsnu_{F}}}.
  \end{align}
  By Lemma \ref{lemma:Sxi0} (here we use that
  $\sup_{j\in\{1,\dots,J-1\}}\tilde\varrho_j^{-1}<1$, see
  \eqref{eq:choiceJ} and \eqref{eq:txi}), the first sum is bounded.
  For the second sum in \eqref{eq:jointsummability}
  \begin{equation}\label{eq:newbound}
    \sum_{k\ge J}\sum_{\bsnu\in\N_0^k}w_\bsnu
    \tilde\bsvarrho_E^{-p \bsnu_E}
    \gamma(\tilde\bsvarrho_F,\bsnu_F)^p
    \frac{|\bsnu_F|^{|\bsnu_F|}}{\bsnu_F^{\bsnu_F}}      
    = \sum_{k\ge J}\left(\sum_{\bsnu\in\N_0^{J-1}}w_{\bsnu}
      \tilde\bsvarrho_{E}^{-p\bsnu}\right)
    \left(\sum_{\bsmu\in\N_0^{k-J+1}}
      \frac{|\bsmu|^{|\bsmu|}}{\bsmu^\bsmu}
      \gamma(\tilde\bsvarrho_{[J:k]},\bsmu)^p \right).
  \end{equation}
  E.g., by \cite[Lemma 3.10]{ZS17} (again due to
  $\sup_{j\in\{1,\dots,J-1\}}\tilde\varrho_j^{-1}<1$)
  \begin{equation*}%
    \sum_{\bsnu\in\N_0^{J-1}}w_\bsnu\tilde\bsvarrho_E^{-p\bsnu}
    \dfnn C_0<\infty,
  \end{equation*}
  and thus \eqref{eq:newbound} is bounded by
    \begin{equation}\label{eq:sum1bound}
      C_0 \sum_{k\ge J}\sum_{\bsmu\in\N_0^{k-J+1}}
      \frac{|\bsmu|^{|\bsmu|}}{\bsmu^\bsmu}
      \gamma(\tilde\bsvarrho_{[J:k]},\bsmu)^p
      = C_0 \sum_{k\in\N}\sum_{\bsmu\in\N_0^k}
      \frac{|\bsmu|^{|\bsmu|}}{\bsmu^\bsmu}
      \gamma(\tilde\bsvarrho_{[J:J+k]},\bsmu)^p<\infty
    \end{equation}
    by Lemma \ref{lemma:Sxi} and because
    $\sum_{j\ge J}(\tilde\varrho_j^{p})^{-1}<1$ by \eqref{eq:choiceJ} and
    \eqref{eq:txi}. In all
    \begin{equation}\label{eq:jointsummability2}
      \sum_{k\in\N}\sum_{\bsnu\in\N_0^k}w_\bsnu|l_{k,\bsnu}|\gamma(\bsvarrho,
      \bsnu)^{p-1}\dfnn C_6<\infty.
    \end{equation}

  {\bf Step 4.} %
  As before, by Lemma \ref{lemma:Sxi0} and because
  $\sup_{j\in\N}\varrho_j^{-1}<1$ and $(\varrho_j^{-1})_{j\in\N}\in\ell^p(\N)$
  (cp.~\eqref{eq:xi}),
  \begin{equation*}
    \sum_{k\in\N}\sum_{\bsnu\in\N_0^k} \gamma(\bsvarrho,\bsnu)^{p}\dfnn C_7<\infty.
  \end{equation*}

  For $k\in\N$ and $\eps>0$ set
  \begin{equation*}
    \Lambda_{\eps,k}=\set{\bsnu\in\N_0^k}{\gamma(\bsvarrho,\bsnu)\ge
    \eps}\qquad\text{and}\qquad
    N_\eps\dfn\sum_{k\in\N}|\Lambda_{\eps,k}|.
  \end{equation*}
  Then
  \begin{equation*}
    N_\eps = \sum_{\set{(k,\bsnu)}{\gamma(\bsvarrho,\bsnu)\ge \eps}}\gamma(\bsvarrho,\bsnu)^{p}\gamma(\bsvarrho,\bsnu)^{-p}
    \le \eps^{-p} \sum_{k\in\N}\sum_{\bsnu\in\N_0^k}\gamma(\bsvarrho,\bsnu)^{p}
    = C_7 \eps^{-p}
  \end{equation*}
  and thus
  \begin{equation}\label{eq:nepsest}
    \eps\le \left(\frac{N_\eps}{C_7}\right)^{-\frac{1}{p}} \qquad\forall \eps>0.
  \end{equation}
  
  On the other hand, assuming $\eps>0$ to be so small that $N_\eps>0$,
  by \eqref{eq:jointsummability2}
  \begin{align}\label{eq:sumest}
    \sum_{k\in\N}\sum_{\bsnu\in\N_0^k\backslash \Lambda_{\eps,k}}w_\bsnu |l_{k,\bsnu}|
    &= \sum_{\set{(k,\bsnu)}{\bsnu\in\N_0^k,~\gamma(\bsvarrho,\bsnu)<\eps}}w_\bsnu|l_{k,\bsnu}|\nonumber\\
    &= \sum_{\set{(k,\bsnu)}{\bsnu\in\N_0^k,~\gamma(\bsvarrho,\bsnu)<\eps}}
      w_\bsnu|l_{k,\bsnu}|
      \gamma(\bsvarrho,\bsnu)^{p-1}
      \gamma(\bsvarrho,\bsnu)^{1-p}\nonumber\\
    &\le C_6 \eps^{1-p}\le (C_6C_7^{\frac{1}{p}-1}) N_\eps^{-\frac{1}{p}+1}.
  \end{align}

  {\bf Step 5.} We finish the proof and verify \eqref{eq:Talgebraic}.

  {For $k\in\N$ and $\eps>0$} define
  $p_{\eps,k}\dfn \sum_{\bsnu\in\Lambda_{\eps,k}} l_{k,\bsnu}
  L_\bsnu\in\bbP_{\Lambda_{\eps,k}}$. We have
  $\sqrt{\partial_kT_k}-1=Q_k= \sum_{\bsnu\in\N_0^k}
  l_{k,\bsnu}L_\bsnu$. Since
  $\norm[L^\infty({\U{k}})]{L_\bsnu}\le w_\bsnu$ by
  \eqref{eq:Lkbound} and \eqref{eq:wbsnu}, by
  \eqref{eq:jointsummability2} and because
  $\gamma(\bsvarrho,\bsnu)^{p-1}\ge 1$
  \begin{equation}\label{eq:boundtRk}
    \sup_{k\in\N}\norm[{L^{\infty}(\U{k})}]{Q_k}
    \le \sup_{k\in\N}\sum_{\bsnu\in\N_0^k}w_\bsnu |l_{k,\bsnu}|
    \le \sum_{k\in\N}\sum_{\bsnu\in\N_0^k}w_\bsnu |l_{k,\bsnu}|\le
    C_6<\infty.
  \end{equation}
  Similarly
  \begin{equation}\label{eq:tRk-pkeps}
    \norm[L^{\infty}({\U{k}})]{Q_k-p_{\eps,k}}
    \le\sum_{\bsnu\in\N_0^k\backslash \Lambda_{\eps,k}}w_\bsnu |l_{k,\bsnu}|.
  \end{equation}

  {In \cite[Lemma C.4]{zm1} we showed that there exists $K\in (0,1]$
    and $C_K>0$ (both independent of $k$) such that
  \begin{equation}\label{eq:lemmaptildeRK}
    \norm[L^\infty(\U{k})]{Q_k-p_{\eps,k}}<\frac{K}{1+\norm[L^\infty(\U{k})]{Q_k}}
  \end{equation}
  implies
  \begin{equation}\label{eq:ptildeR1}
    \norm[L^\infty(\U{k})]{T_k-\tilde T_{\eps,k}}\le C_K
    (1+\norm[L^\infty(\U{k})]{Q_k})^{3}
    \norm[L^\infty(\U{k})]{Q_k-p_{\eps,k}}
  \end{equation}
  and
  \begin{equation}\label{eq:ptildeR2}
    \norm[L^\infty(\U{k})]{\partial_{k} T_k-\partial_{k} \tilde T_{\eps,k}}\le C_K
    (1+\norm[L^\infty(\U{k})]{Q_k})^{3}
    \norm[L^\infty(\U{k})]{Q_k-p_{\eps,k}}.
  \end{equation}}
  We
  distinguish between two cases, first assuming
  \begin{equation}\label{eq:asslemmaptildeR}
    \sum_{\bsnu\in\N_0^k\backslash \Lambda_{\eps,k}}w_\bsnu
  |l_{k,\bsnu}|<
  \frac{K}{1+\norm[L^\infty({\U{k}})]{Q_k}}.
\end{equation}
By \eqref{eq:tRk-pkeps} and \eqref{eq:asslemmaptildeR},
\eqref{eq:lemmaptildeRK} holds.  Now, \eqref{eq:boundtRk},
\eqref{eq:tRk-pkeps} and \eqref{eq:ptildeR1} imply
    \begin{equation}\label{eq:Tkerr1}
      \norm[{L^{\infty}(\U{k})}]{T_k-\tilde T_{\eps,k}}\le C_K
      (1+C_6)^3 \sum_{\bsnu\in\N_0^k\backslash \Lambda_{\eps,k}}w_\bsnu
      |l_{k,\bsnu}|,
    \end{equation}
    and by \eqref{eq:ptildeR2}
    \begin{equation}\label{eq:Tkerr2}
      \norm[{L^{\infty}(\U{k})}]{\partial_kT_k-\partial_k\tilde T_{\eps,k}}\le C_K
      (1+C_6)^3 \sum_{\bsnu\in\N_0^k\backslash \Lambda_{\eps,k}}w_\bsnu
      |l_{k,\bsnu}|.
    \end{equation}

  In the second case where
  \begin{equation}\label{eq:secondcase}
    \sum_{\bsnu\in\N_0^k\backslash \Lambda_{\eps,k}}w_\bsnu
    |l_{k,\bsnu}|> \frac{K}{1+\norm[L^\infty({\U{k}})]{Q_k}},
    \end{equation}
  we redefine $p_{\eps,k}\dfn 0$, so that $\tilde T_{\eps,k}(\bsx)=x_k$
  (cp.~Rmk.~\ref{rmk:kidentity}). Since $T_k:\U{k}\to \U{1}$ and
  $\tilde T_{\eps,k}:\U{k}\to \U{1}$, we get
  $\norm[{L^{\infty}(\U{k})}]{T_k-\tilde T_{\eps,k}}\le 2$, and therefore
  by \eqref{eq:boundtRk}
    \begin{align}\label{eq:Tkerr2}
      \norm[{L^{\infty}(\U{k})}]{T_k-\tilde T_{\eps,k}}
      &\le\frac{2}{\frac{K}{1+\norm[L^\infty({\U{k}})]{Q_k}}}
        \frac{K}{1+\norm[L^\infty({\U{k}})]{Q_k}}\nonumber\\
      &\le \frac{2(1+C_6)}{K}
        \sum_{\bsnu\in\N_0^k\backslash \Lambda_{\eps,k}}w_\bsnu
        |l_{k,\bsnu}|.
    \end{align}
    Next, using $Q_k=\sqrt{\partial_kT_k}-1$, by
    \eqref{eq:boundtRk} it holds
    $\norm[L^\infty({\U{k}})]{\sqrt{\partial_kT_k}}\le 1+C_6$
    as well as
    $\norm[L^\infty({\U{k}})]{\partial_kT_k}\le (1+C_6)^2$.
    Similarly
    $\norm[L^\infty({\U{k}})]{\sqrt{\partial_k\tilde
        T_{\eps,k}}}=\norm[L^\infty({\U{k}})]{p_{\eps,k}}\le 1+C_6$ and
    $\norm[L^\infty({\U{k}})]{\partial_k\tilde T_{\eps,k}}\le (1+C_6)^2$.
    Still assuming \eqref{eq:secondcase},
    we get analogous to \eqref{eq:Tkerr2}
    \begin{align*}
      \norm[{L^{\infty}(\U{k})}]{{\partial_kT_k}-
      {\partial_k\tilde T_{\eps,k}}}
      &\le
        \norm[{L^{\infty}(\U{k})}]{{\partial_kT_k}}
        +\norm[{L^{\infty}(\U{k})}]{{\partial_k\tilde T_{\eps,k}}}\nonumber\\
      &\le
        2(1+C_6)^2
        \le 
        \frac{2(1+C_6)^3}{K}
        \sum_{\bsnu\in\N_0^k\backslash \Lambda_{\eps,k}}w_\bsnu
        |l_{k,\bsnu}|.
    \end{align*}

  In total, by \eqref{eq:Tkerr1}, \eqref{eq:Tkerr2} and
  \eqref{eq:sumest}
  \begin{equation*}
    \sum_{k\in\N}\norm[{L^{\infty}(\U{k})}]{T_k-\tilde T_{\eps,k}}
    \le C \sum_{k\in\N}\sum_{\bsnu\in\N_0^k\backslash\Lambda_{\eps,k}}
    w_\bsnu |l_{k,\bsnu}|\le C N_\eps^{-\frac{1}{p}+1},
  \end{equation*}
  for some $C>0$ independent of $\eps>0$.  An analogous estimate is
  obtained for
  $\sum_{k\in\N}\norm[{L^{\infty}(\U{k})}]{\partial_kT_k-\partial_k\tilde
    T_k}$.
\end{proof}

{
  \subsection{Cor.~\ref{COR:SINF}}
  For the proof we'll need the following two lemmata. The first one is
  classical, e.g., \cite[Lemma 3.1]{zm1}.
\begin{lemma}\label{lemma:lip}
  Let $\zeta>0$. Assume that
  $f\in C^1(\cB_\zeta(\U{1});\C)$ such that
  $\sup_{x\in \cB_\zeta(\U{1})}|f(x)|\le L$. Then
  $\sup_{x\in \U{1}}|f'(x)|\le \frac{L}{\zeta}$ and $f:\U{1}\to\C$ is
  Lipschitz continuous with Lipschitz constant $\frac{L}{\zeta}$.
\end{lemma}

For a function $g$ denote by ${\rm Lip}[g]\in [0,\infty]$ its
Lipschitz constant.
\begin{lemma}\label{lemma:TSlip}
  Let $f_\measi$, $f_\measii$ satisfy Assumption \ref{ass:densities}.
  Then there exists $K>0$ such that for all $k\in\N$, all $j<k$ and
  all $\bsx\in U$ with
  $b_j\dfn
  \max\{\norm[Z]{\psi_{\measi,j}},\norm[Z]{\psi_{\measii,j}}\}$
  \begin{equation}\label{eq:lip1}
    {\rm Lip}\big[\U{1}\ni x_j\mapsto \partial_kT_k(\bsx_{[k]})\big]\le Kb_kb_j,\qquad
    {\rm Lip}\big[\U{1}\ni x_k\mapsto \partial_kT_k(\bsx_{[k]})\big]\le Kb_k
  \end{equation}
  and
  \begin{equation}\label{eq:lip2}
    {\rm Lip}\big[\U{1}\ni x_j\mapsto T_k(\bsx_{[k]})\big]\le 2Kb_kb_j,\qquad
    {\rm Lip}\big[\U{1}\ni x_k\mapsto T_k(\bsx_{[k]})\big]\le 1+K.
  \end{equation}
\end{lemma}
\begin{proof}  
  Fix $k>1$ and $j\in\{1,\dots,k-1\}$.  First applying Lemma
  \ref{LEMMA:SCRFHOL} with $J\dfn 1$ and the multiindex
  $\bsnu\in\N_0^k$ with $\nu_i=0$ if $i\neq j$ and $\nu_j=1$, and then
  applying Prop.~\ref{PROP:COR:DININF} \ref{item:cordinN:binf} it
  holds for some $\bszeta\in (0,\infty)^k$ where in particular
  \begin{equation*}
    \zeta_{j} = \frac{C_2\tau}{b_j}
  \end{equation*}
  that
  \begin{equation*}
    \partial_kT_k-1:\cB_{\bszeta_{[k-1]}}(\U{1})\times \U{1}\to \cB_{\frac{C_3 b_k}{\tau}}.
  \end{equation*}
  Moreover this function is complex differentiable in
  $x_j\in \cB_{\zeta_j}(\U{1})$.  Here the constants $C_2$, $C_3$ and
  $\tau$ solely depend on $\measi$ and $\measii$, and we point out
  that we used the trivial lower bounded $\kappa_j\ge 0$
  for $\kappa_j$ in Lemma
  \ref{LEMMA:SCRFHOL}. By Lemma \ref{lemma:lip}
  \begin{equation*}
    {\rm Lip}\big[\U{1}\ni x_j\mapsto \partial_kT_k(\bsx_{[k]})\big]
    =
    {\rm Lip}\big[\U{1}\ni x_j\mapsto \partial_kT_k(\bsx_{[k]})-1\big]
    \le  \frac{\frac{C_3 b_k}{\tau}}{\zeta_j}
    \le K b_k b_j,
  \end{equation*}
  with
  $K\dfn \max\{\frac{C_3}{C_2},\frac{C_3}{C_2\tau},C_3\}\ge
  \frac{C_3}{C_2}$. This shows the first inequality in \eqref{eq:lip1}.

  Fix $k\in\N$. Similar as above, choosing $\bsnu\in\N_0^k$ such
  that $\nu_i=0$ if $i\neq k$ and $\nu_k=1$ in Lemma
  \ref{LEMMA:SCRFHOL}, we find with Prop.~\ref{PROP:COR:DININF}
  \ref{item:cordinN:ainf} that
  \begin{equation}\label{eq:uselater}
    \partial_kT_k-1:\cB_{\bszeta_{[k]}}(\U{1})\to \cB_{C_3},
  \end{equation}
  where now $\zeta_k=C_2\frac{\tau}{b_k}$.
  Again by Lemma \ref{lemma:lip}
  \begin{equation*}
    {\rm Lip}\big[\U{1}\ni x_k\mapsto \partial_kT_k(\bsx_{[k]})\big]
    =
    {\rm Lip}\big[\U{1}\ni x_k\mapsto \partial_kT_k(\bsx_{[k]})-1\big]
    \le  \frac{C_3}{\frac{C_2\tau}{b_k}}
    \le K b_k,
  \end{equation*}
  which shows the second inequality in \eqref{eq:lip1}.

  Next we show the first inequality in \eqref{eq:lip2} and fix $j<k$.
  For $\bsy\in \U{k}$ and with
  $\tilde\bsy\dfn (\bsy_{[j-1]},\tilde y_j,\bsy_{[j+1:k]})$
  \begin{align*}
    |T_k(\bsy)-T_k(\tilde \bsy)|
    &\le \int_{-1}^{y_k}|\partial_kT_k(\bsy_{[k-1]},t)-\partial_kT_k(\tilde \bsy_{[k-1]},t)|\dd t\nonumber\\
    &\le \int_{-1}^{1} Kb_kb_j|y_j-\tilde y_j| \dd t\le 2Kb_kb_j|y_j-\tilde y_j|.
  \end{align*}
  For the second inequality in \eqref{eq:lip2} let
  $\bsy\in \U{k}$ and
  $\tilde\bsy=(\bsy_{[k-1]},\tilde y_k)$. Then
  \begin{align*}
    |T_k(\bsy)-T_k(\tilde \bsy)|
    &\le \int_{\tilde y_k}^{y_k}|\partial_kT_k(\bsy_{[k-1]},t)|\dd t\nonumber\\
    &\le |y_k-\tilde y_k|\sup_{t\in \U{1}}|\partial_kT_k(\bsy,t)|\nonumber\\
    &\le |y_k-\tilde y_k|(1+C_3)\le
      |y_k-\tilde y_k|(1+K),
  \end{align*}
  where we used
  \eqref{eq:uselater} to bound
  $\sup_{t\in \U{1}}|\partial_kT_k(\bsy,t)|\le 1+C_3$.
\end{proof}

\begin{proof}[Proof of Cor.~\ref{COR:SINF}]
  For notational convenience we drop the index $\eps$ and write
  $\tilde T_{k}$ instead of $\tilde T_{\eps,k}$ etc.
  
  {\bf Step 1.} We show \eqref{eq:Ssumbound}.  Since the assumptions
  on $\measi$ and $\measii$ are the same (see Assumption
  \ref{ass:densities}), switching the roles of the measures,
  \eqref{eq:lip2} implies for the inverse transport $S=(S_k)_{k\in\N}$
  (the KR transport satisfying $S_\sharp\measii=\measi$)
  \begin{equation*}
    {\rm Lip}\big[\U{1}\ni x_j\mapsto S_k(\bsx_{[k]})\big]\le 2Kb_kb_j,\qquad
    {\rm Lip}\big[\U{1}\ni x_k\mapsto S_k(\bsx_{[k]})\big]\le 1+K.
  \end{equation*}
  Recall the notation
  $T_{[k]}=(T_i)_{i=1}^k:\U{k}\to \U{k}$ and
  $T_{[j:k]}=(T_i)_{i=j}^k:\U{k}\to \U{{k-j+1}}$
  for the components of the transport map.
  Then for any
  $k\in\N$ it holds on $\U{k}$
  \begin{align}\label{eq:diffSk}
    |S_k(\tilde T_{[k]})-S_k(T_{[k]})|
    &\le \sum_{j=1}^{k} |S_k(\tilde T_{[j]},T_{[j+1:k]})-S_k(\tilde T_{[j-1]},T_{[j:k]})|\nonumber\\
    &\le (1+K)|\tilde T_k-T_k|+\sum_{j=1}^{k-1} 2Kb_jb_k |\tilde T_j-T_j|.
  \end{align}

  Since $\tilde T:U\to U$ is a bijection (in particular
  $\tilde T_{[k]}:\U{k}\to \U{k}$ is bijective) we get
  \begin{align}\label{eq:sumsk-tsk}
    \sum_{k\in\N}\norm[{L^\infty(\U{k})}]{S_k-\tilde S_k}
    &=\sum_{k\in\N}\norm[{L^\infty(\U{k})}]{S_k\circ \tilde T_{[k]}-\tilde S_k\circ \tilde T_{[k]}}\nonumber\\
    &=\sum_{k\in\N}\norm[{L^\infty(\U{k})}]{S_k\circ \tilde T_{[k]}-S_k\circ T_{[k]}}\nonumber\\
    &\le
      2K \sum_{k\in\N}\sum_{j=1}^{k-1} b_jb_k\norm[{L^\infty(\U{k})}]{\tilde T_j-T_j}
      +(1+K)\sum_{k\in\N}\norm[{L^\infty(\U{k})}]{\tilde T_k-T_k}\nonumber\\
    &=2K \sum_{j\in\N}b_j\norm[{L^\infty(\U{k})}]{\tilde T_j-T_j} \sum_{k>j} b_k
      +(1+K)\sum_{k\in\N}\norm[{L^\infty(\U{k})}]{\tilde T_k-T_k}\nonumber\\
    &\le \left(1+K+2K\max_{j\in\N} b_j \sum_{i\in\N}b_i\right) \sum_{k\in\N}\norm[{L^\infty(\U{k})}]{\tilde T_k-T_k}.
  \end{align}
  Since $\sum_{i\in\N}b_i<\infty$ this together with
  \eqref{eq:Talgebraica} shows \eqref{eq:Ssumbound}.

  {\bf Step 2.} We show \eqref{eq:dSsumbound}. %
  For $\bsx\in\U{k}$ holds $S_{[k]}\circ T_{[k]}(\bsx)=\bsx$.  Thus
  $S_k(T_{[k]}(\bsx))=x_k$ and therefore
  $\partial_kS_k(T_{[k]}(\bsx))\partial_k T_k(\bsx)=1$, where we used
  that $T_{[j]}$ with $j<k$ only depends on $\bsx_{[j]}$. After
  applying $T_{[k]}^{-1}=S_{[k]}$ this reads
  \begin{equation*}
    \partial_kS_k(\bsx) = \frac{1}{\partial_k T_k(S_{[k]}(\bsx))}.
  \end{equation*}
  The second inequality in \eqref{eq:lip2} gives $|\partial_kT_k|\le 1+K$
  and thus $\partial_kS_k(\bsx)\ge \frac{1}{1+K}$ for all
  $k\in\N$ and all
  $\bsx\in \U{k}$. Similarly $\partial_k\tilde S_k(\bsx)=
  \frac{1}{\partial_k \tilde T_k(\tilde S_{[k]}(\bsx))}$.
  By \eqref{eq:Talgebraicb} (as long as $N_\eps\ge 1$) we have
  for $\bsx\in \U{k}$
  \begin{equation*}
    |\partial_k\tilde T_k(\bsx)|\le |\partial_kT_k(\bsx)|
    +|\partial_kT_k(\bsx)-\partial_k\tilde T_k(\bsx)|
    \le 1+K+C
  \end{equation*}
  with the constant $C$ from \eqref{eq:Talgebraicb}.  Thus
  $\partial_k\tilde S_k(\bsx)\ge \frac{1}{1+K+C}$ for $\bsx\in\U{k}$.  Since
  $x\mapsto \frac{1}{x}:[\frac{1}{1+K+C},\infty)\to\R$ has Lipschitz
  constant $(1+K+C)^{2}$, we get
  \begin{align}\label{eq:sumkpSk}
    \sum_{k\in\N}\norm[{L^\infty(\U{k})}]{\partial_kS_{[k]}-\partial_k\tilde S_{[k]}}&=
      \sum_{k\in\N}\normc[{L^\infty(\U{k})}]{\frac{1}{\partial_kT_k\circ S_{[k]}}-
      \frac{1}{\partial_k\tilde T_k\circ \tilde S_{[k]}}}\nonumber\\
&\le   (1+K+C)^2 \sum_{k\in\N}\norm[{L^\infty(\U{k})}]{\partial_kT_k\circ S_{[k]}-\partial_k\tilde T_k\circ\tilde S_{[k]}}\nonumber\\
    &\le (1+K+C)^2\sum_{k\in\N}\bigg(
      \norm[{L^\infty(\U{k})}]{\partial_kT_k\circ S_{[k]}-\partial_kT_k\circ\tilde S_{[k]}}\nonumber\\
    &\quad+\norm[{L^\infty(\U{k})}]{\partial_kT_k\circ \tilde S_{[k]}-\partial_k\tilde T_k\circ\tilde S_{[k]}}\bigg).
  \end{align}
  Using \eqref{eq:lip1} the same calculation as in \eqref{eq:diffSk} yields
  \begin{align*}
    |\partial_kT_k(\tilde S_{[k]})-\partial_kT_k(S_{[k]})|
    &\le \sum_{j=1}^{k} |\partial_kT_k(\tilde S_{[j]},S_{[j+1:k]})-\partial_kT_k(\tilde S_{[j-1]},S_{[j:k]})|\nonumber\\
    &\le Kb_k|\tilde S_k-S_k|+\sum_{j=1}^{k-1} Kb_jb_k |\tilde S_j-S_j|.
  \end{align*}
  Thus by \eqref{eq:sumkpSk} (similar as in \eqref{eq:sumsk-tsk})
  \begin{align*}
    \sum_{k\in\N}\norm[{L^\infty(\U{k})}]{\partial_kS_{[k]}-\partial_k\tilde S_{[k]}}%
        &\le (1+K+C)^2\sum_{k\in\N}\left(Kb_k \norm[{L^\infty(\U{k})}]{\tilde S_k-S_k}+\sum_{j=1}^{k-1}Kb_jb_k\norm[{L^\infty(\U{k})}]{\tilde S_k-S_k} \right)\nonumber\\
    &\quad+(1+K+C)^2\sum_{k\in\N}\norm[{L^\infty(\U{k})}]{\partial_kT_k-\partial_k\tilde T_k}\nonumber\\
    &\le (1+K+C)^2K\max_{j\in\N}b_j\left(1+\sum_{i\in\N}b_i\right)\sum_{k\in\N}\norm[{L^\infty(\U{k})}]{\tilde S_k-S_k}\nonumber\\ 
    &\quad+(1+K+C)^2\sum_{k\in\N}\norm[{L^\infty(\U{k})}]{\partial_kT_k-\partial_k\tilde T_k}.
  \end{align*}
  Applying \eqref{eq:Ssumbound} and \eqref{eq:Talgebraicb} shows
  \eqref{eq:dSsumbound} and concludes the proof.
\end{proof}
}

\section{Proofs of Sec.~\ref{sec:measinfty}}
\subsection{Thm.~\ref{THM:MEASCONVINF}}
\begin{lemma}\label{LEMMA:DET}%
  Let $(a_j)_{j\in\N}$, $(b_j)_{j\in\N}\subseteq (0,\infty)$ be such
  that $\lim_{n\to\infty}\sum_{j=1}^n \log(a_j)\in\R$ exists and
  $\sum_{j\in\N}|a_j-b_j|<\infty$. Then with
  $a_{\rm min}=\min_{j\in\N}a_j>0$, $b_{\rm min}=\min_{j\in\N}b_j>0$
  and
  \begin{equation*}
    C:=\frac{\exp\left(\sum_{j\in\N}\frac{|a_j-b_j|}{a_{\rm min}}\right)
      \lim_{n\to\infty}\prod_{j=1}^na_j}{\min\{a_{\rm min},b_{\rm min}\}}<\infty
  \end{equation*}
  the limit $\lim_{n\to\infty}\prod_{j=1}^n b_j\in\R$ exists and it
  holds
  \begin{equation}\label{eq:proddiff}
    \left|\lim_{n\to\infty}\prod_{j=1}^n a_j-
      \lim_{n\to\infty}\prod_{j=1}^n b_j\right|
    \le C \sum_{j\in\N}|a_j-b_j|.
  \end{equation}
\end{lemma}
\begin{proof}
  For $a>0$, $\log:[a,\infty)\to\R$ has Lipschitz constant $\frac{1}{a}$. Thus
  \begin{equation}\label{eq:logajepsjdiff}
    |\log(a_j)-\log(b_j)|
    \le
      \frac{|a_j-b_j|}{\min\{a_j,b_j\}}\qquad\forall j\in\N.
  \end{equation}
  For $a>0$, $\exp:(-\infty,a]\to\R$ has Lipschitz constant
  $\exp(a)$. Thus, since $a_j+|a_j-b_j|\ge\max\{a_j,b_j\}$,
  \begin{align}\label{eq:proddiffajepsj}
    \left|\prod_{j=1}^n a_j-\prod_{j=1}^n b_j\right|
    &= \exp\left(\sum_{j=1}^n\log(a_j)\right)-\exp\left(\sum_{j=1}^n\log(b_j)\right)\nonumber\\
    &\le \exp\left(\sum_{j=1}^n\log(a_j+|a_j-b_j|)\right)\sum_{j=1}^n \frac{|a_j-b_j|}{\min\{a_j,b_j\}}.
  \end{align}
  Since $\lim_{n\to\infty}\sum_{j=1}^n \log(a_j)\in\R$, it must
  hold $\log(a_j)\to 0$ and $a_j\to 1$ as $j\to\infty$. Hence
  $a_{\rm min}\dfn\min_{j\in\N}a_j>0$.
  Using $\log(1+x)\le x$ for $x\ge 0$ so that
  \begin{equation*}
    \log(a_j+|a_j-b_j|)=\log\left(a_j\left(1+\frac{|a_j-b_j|}{a_j}\right)\right)
    \le\log(a_j)+\frac{|a_j-b_j|}{a_{\rm min}}
  \end{equation*}
  we get
  \begin{align*}
        \lim_{n\to\infty}\exp\left(\sum_{j=1}^n\log(a_j+|a_j-b_j|)\right)&\le \lim_{n\to\infty}\exp\left(\sum_{j=1}^n\Bigg(\log(a_j)+\frac{|a_j-b_j|}{a_{\rm min}}\Bigg)\right)\nonumber\\
      &= \exp\left(\sum_{j\in\N}\frac{|a_j-b_j|}{a_{\rm min}}\right)
        \lim_{n\to\infty}\prod_{j=1}^na_j<\infty.
  \end{align*}
  Equation \eqref{eq:proddiff} follows by taking the limit
  $n\to\infty$ in \eqref{eq:proddiffajepsj}.
\end{proof}

{
\begin{lemma}\label{lemma:deterr}
  Let $T$, $\tilde T_\eps$ be as in Thm.~\ref{THM:TINF} and
  $S\dfn T^{-1}$ and $\tilde S_\eps\dfn \tilde T_\eps^{-1}$.  Then
  there exists $C$ such that for all $\eps>0$
  \begin{equation}\label{eq:prod-prod}
    \sup_{\bsy\in U}\left|\lim_{n\to\infty}\prod_{j=1}^n\partial_jS_j(\bsy_{[j]})-
      \lim_{n\to\infty}\prod_{j=1}^n\partial_j\tilde S_{\eps,j}(\bsy_{[j]})\right|
    \le C N_\eps^{-\frac{1}{p}+1}.
  \end{equation}
\end{lemma}
\begin{proof}
  If $N_\eps=0$ then \eqref{eq:prod-prod} is trivial.  As in Step 2 of
  the proof of Cor.~\ref{COR:SINF}, one shows that for any $k\in\N$
  and $\eps>0$ so small that $N_\eps\ge 1$ we have
  \begin{equation}\label{eq:Cepslower}
    \inf_{\bsy\in U}\min\left\{\partial_kS_k(\bsy_{[k]}),\partial_k\tilde S_{\eps,k}(\bsy_{[k]})\right\}\ge\frac{1}{\bar C}
  \end{equation}
  for a constant $\bar C<\infty$ independent of $k$ and $\eps$.

  By Lemma \ref{LEMMA:DET}
  \begin{equation*}
    \sup_{\bsy\in U}\left|\lim_{n\to\infty}\prod_{j=1}^n\partial_jS_j(\bsy_{[j]})-
      \lim_{n\to\infty}\prod_{j=1}^n\partial_j\tilde S_{\eps,j}(\bsy_{[j]})\right|
    \le C_\eps \sum_{j\in\N}\norm[L^\infty(U_j)]{\partial_j S_j-\partial_j\tilde S_{\eps,j}}
  \end{equation*}
  with
  \begin{equation*}
    C_\eps = \bar C \exp\left(\bar C\sum_{j\in\N}\norm[L^\infty(U_j)]{\partial_j S_j-\partial_j\tilde S_{\eps,j}} \right)\sup_{\bsy\in U}\lim_{n\to\infty}\partial_j S_j(\bsy).
  \end{equation*}
  By \eqref{eq:Cepslower} and using Cor.~\ref{COR:SINF}, we conclude
  that $C_\eps$ is uniformly bounded for all $\eps>0$ so small that
  $N_\eps\ge 1$. Thus it holds \eqref{eq:prod-prod}.
\end{proof}

\begin{proof}[Proof of Thm.~\ref{THM:MEASCONVINF}]
  Throughout we denote
  $\tilde S_\eps=(\tilde S_{\eps,j})_{j\in\N}\dfn \tilde
  T_\eps^{-1}:U\to U$.
  
  {\bf Step 1.} By Thm.~\ref{THM:KNOTHEINF}
  (cp.~Rmk.~\ref{rmk:knotheS}),
  $\det dS(\bsy)\dfn
  \lim_{n\to\infty}\prod_{j=1}^n\partial_jS_j(\bsy_{[j]})\in
  C^0(U;\R)$ exists and (cp.~Assumption \ref{ass:densities})
  \begin{equation}\label{eq:targettransform}
    \frac{\ddd\measii}{\ddd\mu}(\bsy) = f_\measii(\bsy)=\det dS(\bsy)f_\measi(S(\bsy)) =\det dS(\bsy)
    \frf_{\measi}\left(\sum_{j\in\N}S_j(\bsy_{[j]})\psi_{\measi,j}\right)\qquad\forall \bsy\in U.
  \end{equation}
  Next we claim
  \begin{equation}\label{eq:pushforwardclaim}
    \frac{\ddd(\tilde T_\eps)_\sharp\measi}{\ddd\mu}(\bsy)=
    \det d\tilde S_\eps(\bsy)f_\measi(\tilde S_\eps(\bsy))
    =\det d\tilde S_\eps(\bsy)
    \frf_{\measi}\left(\sum_{j\in\N}\tilde S_{\eps,j}(\bsy_{[j]})\psi_{\measi,j}\right)
    \qquad\forall \bsy\in U.
  \end{equation}
  
  By Rmk.~\ref{rmk:kidentity}, there exists $k_0\in\N$ such that
  $\tilde T_{\eps,k}(\bsy_{[k]})=x_k$ for all $k\ge k_0$, and thus
  \begin{equation}\label{eq:Sepsid}
    \tilde S_{\eps,k}(\bsy_{[k]})=x_k\qquad \forall k\ge k_0.
  \end{equation}
  Fix $n_0\ge k_0$ and let $A\subseteq U$ be measurable and of the
  type $A=\times_{j=1}^{n_0}A_j\times U$ with $A_j\subseteq \U{1}$.
  To show \eqref{eq:pushforwardclaim}, e.g., by
  \cite[Thm.~3.5.1]{bogachev}, it suffices to show
  \begin{equation}\label{eq:suffices}
    (\tilde T_\eps)_\sharp\measi(A)=\int_A \det d\tilde
    S_\eps(\bsy)f_\measi(\tilde S_\eps(\bsy))\dd\mu(\bsy),
  \end{equation}
  since these sets form an algebra that generate the $\sigma$-algebra
  on $U$.  For any such $A$
  \begin{equation*}
    (\tilde T_\eps)_\sharp\measi(A)
    = \measi(\set{\bsy\in U}{\tilde T_\eps(\bsy)\in A})
    = \measi(\tilde S_\eps(A))
    = \int_{\tilde S_\eps(A)}f_\measi(\bsy)\dd\mu(\bsy).
  \end{equation*}
  By \eqref{eq:Sepsid}
  we have
  $\tilde S_\eps(A)=\tilde S_{\eps,[n_0]}(\times_{j=1}^{n_0}A_j)\times
  U$ (here
  $\tilde S_{\eps,[n_0]}=(\tilde S_{\eps,j})_{j=1}^{n_0}:\U{j}\to
  \U{j}$) and thus with $\bsy_{[n_0+1:]}=(x_j)_{j>n_0}$
  \begin{align}\label{eq:suffices1}
    (\tilde T_\eps)_\sharp\measi(A) &= \int_{U}\int_{\tilde S_{\eps,[n_0]}(\times_{j=1}^{n_0} A_j)}
                                      f_\measi(\bsy_{[n_0]},\bsy_{[n_0+1:]})\dd\mu(\bsy_{[n_0]})\dd\mu(\bsy_{[n_0+1:]})\nonumber\\
                                    &=\int_{U}\int_{\times_{j=1}^{n_0} A_j}
                                      f_\measi(\tilde S_{\eps,[n_0]}(\bsy_{[n_0]}),\bsy_{[n_0+1:]})
                                      \det d\tilde S_{\eps,[n_0]}(\bsy_{[n_0]})
                                      \dd\mu(\bsy_{[n_0]})\dd\mu(\bsy_{[n_0+1:]}).
  \end{align}
  Again by \eqref{eq:Sepsid}
  we have
  \begin{equation*}
    \det d\tilde S_\eps(\bsy)\dfn
    \lim_{m\to\infty}\prod_{j=1}^m\partial_j \tilde
    S_{\eps,j}(\bsy_{[j]}) = \prod_{j=1}^{k_0}\partial_j \tilde
    S_{\eps,j}(\bsy_{[j]})
    =\prod_{j=1}^{n_0}\partial_j \tilde
    S_{\eps,j}(\bsy_{[j]})
    =\det d\tilde S_{\eps,[n_0]}(\bsy_{[n_0]}).
  \end{equation*}
  Since
  $(\tilde S_{\eps,[n_0]}(\bsy_{[n_0]}),\bsy_{[n_0+1:]})=\tilde
  S_{\eps}(\bsy)$, \eqref{eq:suffices1} shows
  \eqref{eq:suffices}.
  
  {\bf Step 2.} By Lemma \ref{lemma:deterr}
  \begin{equation}\label{eq:deterr}
    \sup_{\bsy\in U}|\det dS(\bsy)-\det d\tilde S_\eps(\bsy)|=
    \sup_{\bsy\in U}\left|\lim_{n\to\infty}\prod_{j=1}^n\partial_j S_j(\bsy_{[j]})-\lim_{n\to\infty}\prod_{j=1}^n\partial_j \tilde S_j(\bsy_{[j]})\right|\le
    C N_\eps^{-\frac{1}{p}+1}.
  \end{equation}
  Using that the differentiable function $\frf_\measi:O_X\to\C$ has
  some Lipschitz constant $r<\infty$ on the compact set
  $\set{\sum_{j\in\N}y_j\psi_{\measi,j}}{\bsy\in U}\subseteq
  O_X\subseteq X_\C$ (cp.~Assumption \ref{ass:density}), we have for
  all $\bsy\in U$ with $b_j\dfn \norm[X]{\psi_{\measi,j}}$
  \begin{align}\label{eq:denserr}
    \left|\frf_\measi\left(\sum_{j\in\N} S_{j}(\bsy_{[j]})\psi_{\measi,j}\right)-
    \frf_\measi\left(\sum_{j\in\N}\tilde S_{\eps,j}(\bsy_{[j]})\psi_{\measi,j}\right)\right|
    &\le r \sum_{j\in\N}|S_{j}(\bsy_{[j]})-\tilde S_{\eps,j}(\bsy_{[j]})|b_j \nonumber\\
    &\le C N_\eps^{-\frac{1}{p}+1}
  \end{align}
  by Cor.~\ref{COR:SINF}, and for some $C$ depending on $r$ and
  $\sup_{j\in\N}b_j<\infty$.
  
  Therefore, using \eqref{eq:targettransform},
  \eqref{eq:pushforwardclaim}, \eqref{eq:deterr}, \eqref{eq:denserr}
  and the triangle inequality we find
  \begin{equation}\label{eq:totalerr}
    \sup_{\bsy\in U}\left|f_\measii(\bsy)-\frac{\ddd(\tilde T_\eps)_\sharp\measi}{\ddd\mu}(\bsy)\right|=    
    \sup_{\bsy\in U}\left|\frac{\ddd\measii}{\ddd\mu}(\bsy)-\frac{\ddd(\tilde T_\eps)_\sharp\measi}{\ddd\mu}(\bsy)\right|
    \le C N_\eps^{-\frac{1}{p}+1}
  \end{equation}
  for some suitable constant $C<\infty$ and all $\eps>0$.

  Equation \eqref{eq:totalerr} yields \eqref{eq:measdiffinf} for the
  total variation distance. Moreover
  \begin{equation*}
    \sup_{\bsy\in U}\left|\sqrt{f_\measii(\bsy)}-\sqrt{\frac{\ddd(\tilde T_\eps)_\sharp\measi}{\ddd\mu}(\bsy)}\right|
    \le \sup_{\bsy\in U}\frac{\left|{f_\measii(\bsy)}-\frac{\ddd(\tilde T_\eps)_\sharp\measi}{\ddd\mu}(\bsy)\right|}{|\sqrt{f_\measii(\bsy)}|}
    \le \frac{CN_\eps^{-\frac{1}{p}+1}}{\inf_{\bsy\in U}\sqrt{f_\measii(\bsy)}},
  \end{equation*}
  which gives \eqref{eq:measdiffinf} for the Hellinger distance since
  $\inf_{\bsy\in U}f_\measii(\bsy)\ge M>0$ by Assumption
  \ref{ass:density}.

  Finally, for the KL divergence, using that
  $a|\log(a)-\log(b)|\le (1+\frac{|a-b|}{b})|a-b|$ for all $a$, $b>0$
  (see \cite[Lemma E.2]{zm1}), by \eqref{eq:totalerr}
  \begin{align*}
    {\rm KL}((\tilde T_\eps)_\sharp \measi||\measii)
    &\le \int_{U}\left|\frac{\ddd(\tilde T_\eps)_\sharp\measi}{\ddd\mu}(\bsy)\right|\left|
      \log\left(\frac{\ddd(\tilde T_\eps)_\sharp\measi}{\ddd\mu}(\bsy)\right)-\log(f_\measii(\bsy))\right|\dd\mu(\bsy)\nonumber\\
    &\le \left(1+\frac{\norm[L^\infty{(U)}]{f_\measii-\frac{\ddd(\tilde T_\eps)_\sharp\measi}{\ddd\mu}}}{\inf_{\bsy\in U}f_\measii(\bsy)}\right)
      \normc[L^\infty{(U)}]{f_\measii-\frac{\ddd(\tilde T_\eps)_\sharp\measi}{\ddd\mu}} \nonumber\\
    &\le C N_\eps^{-\frac{1}{p}+1}.\qedhere
  \end{align*}
\end{proof}
}%

{
\subsection{Prop.~\ref{PROP:WASSERSTEIN}}
  \begin{proof}[Proof of Prop.~\ref{PROP:WASSERSTEIN}]
    The continuous function $(x,y)\mapsto d(x,y)$ is bounded on the
    compact set $M\times M$. Thus the Wasserstein distance
    $W_q(T_\sharp\nu,\tilde T_\sharp\nu)$ is well-defined and finite.

    Fix $\eps>0$ and let $(B_\eps(x_i))_{i=1}^n$
    with $B_\eps(x_i)\dfn\set{x\in M}{d(x,x_i)<\eps}$
    be a finite cover of
    $M$. %
    Such a cover exists by compactness of $M$.  Define
    $I_1\dfn B_\eps(x_1)$ and inductively set
    $I_j\dfn B_\eps(x_j)\backslash\bigcup_{i=1}^{j-1}I_i$, so that
    $(I_j)_{j=1}^n$ is a (measurable) partition of $M$.

    Denote by $\mu_j$ the measure
    $\mu_j(A)\dfn \nu(T^{-1}(A)\cap I_j)$
    and by $\tilde\mu_j$ the measure
    $\tilde \mu_j(A)\dfn \nu(\tilde T^{-1}(A)\cap I_j)$
    for all measurable
    $A\subseteq M$. %
    Then $\sum_{j=1}^n\mu_j(A)=\nu(T^{-1}(A))$ for all measurable $A$,
    i.e., $T_\sharp\nu=\sum_{j=1}^n\mu_j$.
    Similarly
    $\tilde T_\sharp\nu=\sum_{j=1}^n\tilde\mu_j$.
    Note that
    \begin{equation*}
      \mu_j(M)=\nu(T^{-1}(M)\cap I_j)=\nu(I_j)=\nu(\tilde T^{-1}(M)\cap I_j)=\tilde\mu_j(M).
    \end{equation*}
    Wlog $\mu_j(M)=\nu(I_j)=\tilde\mu_j(M)>0$ for all $j\in\{1,\dots,n\}$
    (otherwise we can omit $\mu_j$, $\tilde\mu_j$).
    Let $\Gamma_j$ be the couplings between $\mu_j$ and $\tilde\mu_j$
    (the measures on $M\times M$ with marginals $\mu_j$ and
    $\tilde\mu_j$).  Note that $\Gamma_j$ is not empty since
      $\frac{1}{\nu(I_j)}\mu_j\otimes\tilde\mu_j\in\Gamma_j$. Let
    $\Gamma$ be the couplings between $T_\sharp \nu$ and
    $\tilde T_\sharp \nu$. Then
    $\set{\sum_{j=1}^n\gamma_j}{\gamma_j\in\Gamma_j}\subseteq\Gamma$. Thus
    \begin{align*}
      W_q(T_\sharp\nu,\tilde T_\sharp\nu)^q&=\inf_{\gamma\in\Gamma}
                                             \int_{M\times M} d(x,y)^q\dd\gamma(x,y)\nonumber\\
                                           &\le\inf_{\gamma_j\in\Gamma_j}
                                             \sum_{j=1}^n\int_{M\times M} d(x,y)^q\dd\gamma_j(x,y).%
    \end{align*}
    If $A\subseteq M$ has empty intersection with $T(I_j)=\set{T(x)}{x\in I_j}$,
    then $T^{-1}(A)=\set{x}{T(x)\in A}$ has empty intersection
    with $I_j$. Thus ${\rm supp}(\mu_j)\subseteq T(I_j)$ and
    similarly
    ${\rm supp}(\tilde\mu_j)\subseteq \tilde T(I_j)$. Hence
    \begin{equation*}
      \int_{M\times M}d(x,y)^q\dd\gamma_j(x,y)
      =\int_{T(I_j)\times \tilde T(I_j)}d(x,y)^q\dd\gamma_j(x,y)
      \le \mu_j(T(I_j)) \sup_{(x,y)\in T(I_j)\times \tilde T(I_j)}d(x,y)^q.
    \end{equation*}
    Here we used that $\gamma_j$ has marginal $\mu_j$ in the first
    argument. %
    In total, using $\mu_j(T(I_j))=\nu(I_j)$
    \begin{equation}\label{eq:Wqbound}
      W_q(T_\sharp\nu,T_\sharp\mu)^q\le \nu(M)\max_{j=1,\dots,n}
      \sup_{(x,y)\in T(I_j)\times \tilde T(I_j)}d(x,y)^q.
    \end{equation}

    To bound the supremum we first note that continuity of $T$
    on the compact set $M$ implies uniform continuity, i.e.,
    \begin{equation*}
      \lim_{\delta\to 0}\sup_{x\in M}\sup_{y\in B_\delta(x)\cap M}d(T(x),T(y)) = 0
    \end{equation*}
    and the same holds for $\tilde T$.
    Now fix $x_j\in I_j$ for each $j$. Then
    \begin{align*}
      &\sup_{(x,y)\in T(I_j)\times \tilde T(I_j)}d(x,y)
      \le \sup_{x,y\in I_j}\Big(d(T(x),T(x_j))+
        d(T(x_j),\tilde T(x_j))+d(\tilde T(x_j),\tilde T(y))\Big)\nonumber\\
      &\qquad\le \sup_{x\in M}\sup_{y\in B_\eps(x)\cap M}d(T(x),T(y)) + \sup_{x\in M}d(T(x),\tilde T(x)) +\sup_{x\in M}\sup_{y\in B_\eps(x)\cap M}d(\tilde T(x),\tilde T(y))\nonumber\\
      &\qquad=\sup_{x\in M}d(T(x),\tilde T(x))+o(1)\qquad\text{as }\eps\to 0.
    \end{align*}
    Together with \eqref{eq:Wqbound} this concludes the proof.
  \end{proof}

  \subsection{Lemma \ref{LEMMA:TCONT}}
    \begin{proof}[Proof of Lemma \ref{LEMMA:TCONT}]
    By Lemma \ref{lemma:producttop} $d$ induces the product topology
    on $U$ (independent of the choice of positive and summable
    sequence $(c_j)_{j\in\N}$ in \eqref{eq:prodmet}). Thus, to prove
    the lemma, it suffices to check Lipschitz continuity in case
    $b_j\le Cc_j$.
    
    We begin with $\tilde T$. By Rmk.~\ref{rmk:kidentity} there exists
    $k_0$ such that $T_k(\bsy)=y_k$ for all $\bsy\in \U{k}$. By
    construction each $\tilde T_k:\U{k}\to \U{1}$ is a rational
    function with positive denominator (in particular $C^\infty$) and
    thus there exists $L>0$ such that $L$ is a Lipschitz constant of
    $\tilde T_k:\U{k}\to \U{1}$ for all $k\in\{1,\dots,k_0\}$
    w.r.t.\ the Euclidean norm $\norm[]{\cdot}$.  Thus for all $\bsx$,
    $\bsy\in U$
    \begin{align*}
      d(\tilde T(\bsx),\tilde T(\bsy))
      &=\sum_{k<k_0}c_k|\tilde T_k(\bsx_{[k]})-\tilde T_k(\bsy_{[k]})|
        +\sum_{k\ge k_0}c_k|\tilde T_k(\bsx_{[k]})-\tilde T_k(\bsy_{[k]})|\nonumber\\
      &\le \sum_{k<k_0}c_kL\norm[]{\bsx_{[k]}-\bsy_{[k]}}
        +\sum_{k\ge k_0}c_k|x_k-y_k|\nonumber\\
      &\le \sum_{k<k_0}c_kL\sum_{j=1}^k|x_j-y_j|
        +\sum_{k\ge k_0}c_k|x_k-y_k|\nonumber\\
      &\le C_0\sum_{k\in\N}c_k|x_k-y_k|,
    \end{align*}
    where $C_0\dfn 1+L\sum_{j=1}^{k_0-1}$.
    
    The argument for $T$ is similar as in the proof of
    Cor.~\ref{COR:SINF}. By \eqref{eq:lip2} for all $\bsx$,
    $\bsy\in U$ and all $k\in\N$
    \begin{align*}
      |T_k(\bsx_{[k]})-T_k(\bsy_{[k]})|
      &\le \sum_{j=1}^k
        |T_k(\bsx_{[j]},\bsy_{[j+1:k]})-T_k(\bsx_{[j-1]},\bsy_{[j:k]})|\nonumber\\
      &\le (1+K)|x_k-y_k|+\sum_{j=1}^{k-1}2Kb_kb_j|x_j-y_j|
    \end{align*}
    and thus since $b_j\le Cc_j$
    \begin{align*}
            d(T(\bsx),T(\bsy))&=
                               \sum_{k\in\N}c_k |T_k(\bsx_{[k]})-T_k(\bsy_{[k]})|\nonumber\\
      &\le \sum_{k\in\N}(1+K)c_k|x_k-y_k|
                                                          +\sum_{k\in\N}\sum_{j=1}^{k-1} 2Kb_kb_jc_k|x_j-y_j|\nonumber\\
      &=(1+K) d(\bsx,\bsy)
        +\sum_{j\in\N}b_j|x_j-y_j|\sum_{k>j} 2Kb_kc_k\nonumber\\
                                                        &\le (1+K)d(\bsx,\bsy)
                                                          +C\sum_{j\in\N}c_j|x_j-y_j|\sum_{k\in\N} 2Kb_kc_k\nonumber\\
      &= \left(1+K+2CK\sum_{k\in\N}b_kc_k\right)d(\bsx,\bsy).\qedhere
    \end{align*}
  \end{proof}

  \subsection{Cor.~\ref{COR:MEASCONVINF}}
  \begin{proof}[Proof of Cor.~\ref{COR:MEASCONVINF}]
Fix $\eps>0$.  Set $H:U\to U$ via
$H_j\dfn \tilde T_{\eps,j}$ if $j\le N_\eps$ and $H_j(\bsy)\dfn 0$ for
$j>N_\eps$. Then $\Phi(H(\bsy))=\Phi_{N_\eps}(\tilde T_{\eps,[N_\eps]}(\bsy_{[N_\eps]}))$
for all $\bsy\in U$. Thus $(\Phi\circ H)_\sharp\measi=(\Phi_{N_\eps}\circ
\tilde T_{\eps,[N_\eps]})_\sharp\measi_{N_\eps}$, and it suffices to bound
the difference between $(\Phi\circ H)_\sharp\measi$
and $(\Phi\circ T)_\sharp\measi=\Phi_\sharp\measii$. To this end
we compute similar as in \eqref{eq:before} for all $\bsy\in U$
with $b_j$ in \eqref{eq:bj}
\begin{align}\label{eq:truncated}
  \norm[Y]{\Phi(T(\bsy))-\Phi(H(\bsy))}
  &=\normc[Y]{\sum_{j=1}^{N_\eps}(T_{j}(\bsy_{[j]})-\tilde T_{\eps,j}(\bsy_{[j]}))\psi_{\measii,j}+\sum_{i>{N_\eps}}T_{j}(\bsy_{[j]})\psi_{\measii,j}}\nonumber\\
  &\le \sum_{j=1}^{N_\eps}b_j\norm[L^\infty(\U{j})]{T_j-\tilde T_{\eps,j}}+\sum_{i>{N_\eps}}b_j.
\end{align}
In case the $(b_j)_{j\in\N}$ are monotonically decreasing, Stechkin's
lemma, which is easily checked, states that
$\sum_{i>{N_\eps}}b_j\le \norm[\ell^p(\N)]{(b_j)_{j\in\N}}
{N_\eps}^{-\frac{1}{p}+1}$. The $\ell^p$-norm is finite by Assumption
\ref{ass:densities}. Thus by Thm.~\ref{THM:TINF} the last term in
\eqref{eq:truncated} is bounded by
$C(N_\eps^{-\frac{1}{p}+1}+ {N_\eps}^{-\frac{1}{p}+1})$.  An application of
Prop.~\ref{PROP:WASSERSTEIN} yields the same bound for the Wasserstein
distance.  %
  \end{proof}  
}%
\bibliographystyle{abbrv} \bibliography{main}

\def\cprime{$'$} \def\cprime{$'$}
\begin{thebibliography}{10}

\bibitem{1509.07526}
A.~Alexanderian.
\newblock A brief note on the {K}arhunen-{L}o\`eve expansion, 2015.

\bibitem{pmlr-v70-arjovsky17a}
M.~Arjovsky, S.~Chintala, and L.~Bottou.
\newblock {W}asserstein generative adversarial networks.
\newblock In D.~Precup and Y.~W. Teh, editors, {\em Proceedings of the 34th
  International Conference on Machine Learning}, volume~70 of {\em Proceedings
  of Machine Learning Research}, pages 214--223. PMLR, 06--11 Aug 2017.

\bibitem{bogachev}
V.~I. Bogachev.
\newblock {\em Measure theory. {V}ol. {I}, {II}}.
\newblock Springer-Verlag, Berlin, 2007.

\bibitem{bogachevtri}
V.~I. Bogachev, A.~V. Kolesnikov, and K.~V. Medvedev.
\newblock Triangular transformations of measures.
\newblock {\em Mat. Sb.}, 196(3):3--30, 2005.

\bibitem{chkifa}
A.~Chkifa.
\newblock Sparse polynomial methods in high dimension: application to
  parametric {PDE}, 2014.
\newblock Ph.D.\ thesis, UPMC, Universit\'e Paris 06, Paris, France.

\bibitem{CCS15}
A.~Cohen, A.~Chkifa, and C.~Schwab.
\newblock Breaking the curse of dimensionality in sparse polynomial
  approximation of parametric pdes.
\newblock {\em Journ. Math. Pures et Appliquees}, 103(2):400--428, 2015.

\bibitem{CDS10}
A.~Cohen, R.~DeVore, and C.~Schwab.
\newblock Convergence rates of best {$N$}-term {G}alerkin approximations for a
  class of elliptic s{PDE}s.
\newblock {\em Found. Comput. Math.}, 10(6):615--646, 2010.

\bibitem{CDS11}
A.~Cohen, R.~Devore, and C.~Schwab.
\newblock Analytic regularity and polynomial approximation of parametric and
  stochastic elliptic {PDE}'s.
\newblock {\em Anal. Appl. (Singap.)}, 9(1):11--47, 2011.

\bibitem{MR3839555}
M.~Dashti and A.~M. Stuart.
\newblock The {B}ayesian approach to inverse problems.
\newblock In {\em Handbook of uncertainty quantification. {V}ol. 1, 2, 3},
  pages 311--428. Springer, Cham, 2017.

\bibitem{davis}
P.~Davis.
\newblock {\em Interpolation and Approximation}.
\newblock Dover Books on Mathematics. Dover Publications, 1975.

\bibitem{MR3349831}
G.~De~Philippis and A.~Figalli.
\newblock Partial regularity for optimal transport maps.
\newblock {\em Publ. Math. Inst. Hautes \'{E}tudes Sci.}, 121:81--112, 2015.

\bibitem{doersch2016tutorial}
C.~Doersch.
\newblock Tutorial on variational autoencoders.
\newblock {\em arXiv preprint arXiv:1606.05908}, 2016.

\bibitem{MR2972870}
T.~A. El~Moselhy and Y.~M. Marzouk.
\newblock Bayesian inference with optimal maps.
\newblock {\em J. Comput. Phys.}, 231(23):7815--7850, 2012.

\bibitem{NIPS2014_5ca3e9b1}
I.~Goodfellow, J.~Pouget-Abadie, M.~Mirza, B.~Xu, D.~Warde-Farley, S.~Ozair,
  A.~Courville, and Y.~Bengio.
\newblock Generative adversarial nets.
\newblock In Z.~Ghahramani, M.~Welling, C.~Cortes, N.~Lawrence, and K.~Q.
  Weinberger, editors, {\em Advances in Neural Information Processing Systems},
  volume~27. Curran Associates, Inc., 2014.

\bibitem{goodfellow2014generative}
I.~J. Goodfellow, J.~Pouget-Abadie, M.~Mirza, B.~Xu, D.~Warde-Farley, S.~Ozair,
  A.~Courville, and Y.~Bengio.
\newblock Generative adversarial networks.
\newblock {\em arXiv preprint arXiv:1406.2661}, 2014.

\bibitem{jaini2019sum}
P.~Jaini, K.~A. Selby, and Y.~Yu.
\newblock Sum-of-squares polynomial flow.
\newblock {\em ICML}, 2019.

\bibitem{JSZ16}
C.~Jerez-Hanckes, C.~Schwab, and J.~Zech.
\newblock Electromagnetic wave scattering by random surfaces: Shape holomorphy.
\newblock {\em Math. Mod. Meth. Appl. Sci.}, 27(12):2229--2259, 2017.

\bibitem{padraig}
P.~Kirwan.
\newblock Complexifications of multilinear and polynomial mappings, 1997.
\newblock Ph.D. thesis, National University of Ireland, Galway.

\bibitem{2102.10461}
K.~Kothari, A.~Khorashadizadeh, M.~de~Hoop, and I.~Dokmani\'c.
\newblock Trumpets: Injective flows for inference and inverse problems, 2021.

\bibitem{krantz}
S.~G. Krantz.
\newblock {\em Function theory of several complex variables}.
\newblock AMS Chelsea Publishing, Providence, RI, 2001.
\newblock Reprint of the 1992 edition.

\bibitem{2002.08927}
A.~Kumar, B.~Poole, and K.~Murphy.
\newblock Regularized autoencoders via relaxed injective probability flow,
  2020.

\bibitem{MR3308418}
G.~J. Lord, C.~E. Powell, and T.~Shardlow.
\newblock {\em An introduction to computational stochastic {PDE}s}.
\newblock Cambridge Texts in Applied Mathematics. Cambridge University Press,
  New York, 2014.

\bibitem{MR3821485}
Y.~Marzouk, T.~Moselhy, M.~Parno, and A.~Spantini.
\newblock Sampling via measure transport: an introduction.
\newblock In {\em Handbook of uncertainty quantification. {V}ol. 1, 2, 3},
  pages 785--825. Springer, Cham, 2017.

\bibitem{MR1823156}
L.~Mattner.
\newblock Complex differentiation under the integral.
\newblock {\em Nieuw Arch. Wiskd. (5)}, 2(1):32--35, 2001.

\bibitem{groundwater}
D.~McLaughlin and L.~R. Townley.
\newblock A reassessment of the groundwater inverse problem.
\newblock {\em Water Resources Research}, 32(5):1131--1161, 1996.

\bibitem{munkres}
J.~R. Munkres.
\newblock {\em Topology}.
\newblock Prentice Hall, Inc., Upper Saddle River, NJ, 2000.

\bibitem{munoz99}
G.~A. Mu{\~n}oz, Y.~Sarantopoulos, and A.~Tonge.
\newblock Complexifications of real {B}anach spaces, polynomials and
  multilinear maps.
\newblock {\em Studia Math.}, 134(1):1--33, 1999.

\bibitem{nist}
F.~W.~J. Olver, D.~W. Lozier, R.~F. Boisvert, and C.~W. Clark, editors.
\newblock {\em N{IST} handbook of mathematical functions}.
\newblock U.S. Department of Commerce, National Institute of Standards and
  Technology, Washington, DC; Cambridge University Press, Cambridge, 2010.

\bibitem{papamakarios2019normalizing}
G.~Papamakarios, E.~Nalisnick, D.~J. Rezende, S.~Mohamed, and
  B.~Lakshminarayanan.
\newblock Normalizing flows for probabilistic modeling and inference.
\newblock {\em Journal of Machine Learning Research}, 22:1--64, 2021.

\bibitem{pmlr-v37-rezende15}
D.~Rezende and S.~Mohamed.
\newblock Variational inference with normalizing flows.
\newblock In F.~Bach and D.~Blei, editors, {\em Proceedings of the 32nd
  International Conference on Machine Learning}, volume~37 of {\em Proceedings
  of Machine Learning Research}, pages 1530--1538, Lille, France, 07--09 Jul
  2015. PMLR.

\bibitem{10.5555/1051451}
C.~P. Robert and G.~Casella.
\newblock {\em Monte Carlo Statistical Methods (Springer Texts in Statistics)}.
\newblock Springer-Verlag, Berlin, Heidelberg, 2005.

\bibitem{MR4120535}
A.~Sagiv.
\newblock The {W}asserstein distances between pushed-forward measures with
  applications to uncertainty quantification.
\newblock {\em Commun. Math. Sci.}, 18(3):707--724, 2020.

\bibitem{santambrogio}
F.~Santambrogio.
\newblock {\em Optimal transport for applied mathematicians}, volume~87 of {\em
  Progress in Nonlinear Differential Equations and their Applications}.
\newblock Birkh\"{a}user/Springer, Cham, 2015.
\newblock Calculus of variations, PDEs, and modeling.

\bibitem{CSStBIP2012}
C.~Schwab and A.~M. Stuart.
\newblock Sparse deterministic approximation of {B}ayesian inverse problems.
\newblock {\em Inverse Problems}, 28(4):045003, 32, 2012.

\bibitem{spantini2018inference}
A.~Spantini, D.~Bigoni, and Y.~Marzouk.
\newblock Inference via low-dimensional couplings.
\newblock {\em The Journal of Machine Learning Research}, 19(1):2639--2709,
  2018.

\bibitem{pmlr-v70-telgarsky17a}
M.~Telgarsky.
\newblock Neural networks and rational functions.
\newblock In D.~Precup and Y.~W. Teh, editors, {\em Proceedings of the 34th
  International Conference on Machine Learning}, volume~70 of {\em Proceedings
  of Machine Learning Research}, pages 3387--3393. PMLR, 06--11 Aug 2017.

\bibitem{MR2459454}
C.~Villani.
\newblock {\em Optimal transport}, volume 338 of {\em Grundlehren der
  Mathematischen Wissenschaften [Fundamental Principles of Mathematical
  Sciences]}.
\newblock Springer-Verlag, Berlin, 2009.
\newblock Old and new.

\bibitem{wehenkel2019unconstrained}
A.~Wehenkel and G.~Louppe.
\newblock Unconstrained monotonic neural networks.
\newblock {\em arXiv preprint arXiv:1908.05164}, 2019.

\bibitem{yarotsky}
D.~Yarotsky.
\newblock Error bounds for approximations with deep {ReLU} networks.
\newblock {\em Neural Netw.}, 94:103--114, 2017.

\bibitem{JZdiss}
J.~Zech.
\newblock Sparse-{G}rid {A}pproximation of {H}igh-{D}imensional {P}arametric
  {PDE}s, 2018.
\newblock Dissertation 25683, ETH Z\"urich,
  \url{http://dx.doi.org/10.3929/ethz-b-000340651}.

\bibitem{2006.06994}
J.~Zech and Y.~Marzouk.
\newblock Sparse approximation of triangular transports on bounded domains,
  2020.
\newblock arXiv:2006.06994v1.

\bibitem{zm1}
J.~Zech and Y.~Marzouk.
\newblock Sparse approximation of triangular transports. {P}art {I}: the finite
  dimensional case, 2021.

\bibitem{ZS17}
J.~Zech and C.~Schwab.
\newblock Convergence rates of high dimensional {S}molyak quadrature.
\newblock {\em ESAIM Math. Model. Numer. Anal.}, 54(4):1259--1307, 2020.

\end{thebibliography}
\end{document}